\documentclass{amsart}
\usepackage[foot]{amsaddr}
\usepackage{float, graphicx}
\usepackage[]{epsfig}
\usepackage{amsmath, amsthm, amssymb,}
\usepackage{epsfig}
\usepackage{verbatim}
\usepackage{url}
\usepackage{latexsym}
\usepackage{mathrsfs}
\usepackage{hyperref}
\usepackage{graphicx}
\usepackage{enumerate}
\usepackage[normalem]{ulem}
\usepackage{bm}
\usepackage{stmaryrd}
\usepackage{enumitem}
\usepackage{mathtools}
\usepackage{color}
\usepackage[letterpaper]{geometry}
\usepackage[noadjust]{cite}
\usepackage{cleveref}
\usepackage{esint}
\usepackage{dsfont}

\usepackage{tikz}
\usetikzlibrary{arrows}
\usetikzlibrary{arrows.meta}
\definecolor{wwhhii}{rgb}{1.,1.,1.}
\definecolor{rreedd}{rgb}{1.,0.,0.}
\definecolor{uuuuuu}{rgb}{0.26666666666666666,0.26666666666666666,0.26666666666666666}

\numberwithin{equation}{section}

\DeclareMathOperator{\OO}{\mathcal{O}}

\DeclareMathOperator{\argmin}{argmin}

\DeclareMathOperator{\del}{\partial}

\DeclareMathOperator{\Ai}{Ai}

\newtheorem{thm}{Theorem}[section]
\newtheorem{prop}[thm]{Proposition}
\newtheorem{lem}[thm]{Lemma}
\newtheorem{cor}[thm]{Corollary}

\theoremstyle{remark}

\newtheorem{remark}[thm]{Remark}
\theoremstyle{definition}
\newtheorem{definition}[thm]{Definition}
\newtheorem{assumption}[thm]{Assumption}

\newcommand\cA{{\mathcal A}}

\newcommand{\cC}{{\mathcal C}}

\newcommand{\cD}{{\mathcal D}}
\newcommand{\cE}{{\mathcal E}}
\newcommand{\hcE}{{\hat{\cE}}}
\newcommand{\cF}{{\mathcal F}}

\newcommand{\cK}{{\mathcal K}}
\newcommand{\cL}{{\mathcal L}}

\newcommand{\fa}{{\mathfrak a}}

\newcommand{\fd}{{\mathfrak d}}

\newcommand{\bmx}{{\bm{x}}}
\newcommand{\bmy}{{\bm{y}}}

\newcommand{\rd}{{\rm d}}

\newcommand{\ri}{\mathrm{i}}

\newcommand{\bC}{{\mathbb C}}

\newcommand{\bE}{\mathbb{E}}
\newcommand{\bH}{\mathbb{H}}
\newcommand{\bN}{\mathbb{N}}
\newcommand{\bP}{\mathbb{P}}
\newcommand{\bQ}{\mathbb{Q}}
\newcommand{\bR}{{\mathbb R}}

\newcommand{\bZ}{\mathbb{Z}}

\newcommand{\al}{\alpha}
\newcommand{\la}{\lambda}
\newcommand{\wt}{\widetilde}

\renewcommand{\Im}{{\rm Im}}
\renewcommand{\Re}{{\rm Re}}

\newcommand{\qq}[2]{[\![ #1 #2]\!]}

\title[ALE$_\beta$ via pole evolution]{A convergence framework for Airy$_\beta$ line ensemble\\ via pole evolution}
\author{Jiaoyang Huang}
\address{Department of Statistics and Data Science, University of Pennsylvania, Pennsylvania, PA, USA}
\email{huangjy@wharton.upenn.edu}

\author{Lingfu Zhang}
\address{Division of Physics, Mathematics and Astronomy, California Institute of Technology, Pasadena, CA, USA. \and Department of Statistics, University of California, Berkeley, CA, USA.}
\email{lingfuz@caltech.edu}

\begin{document}
\maketitle

\begin{abstract}
The Airy$_\beta$ line ensemble is an infinite sequence of random curves. It is a natural extension of the Tracy-Widom$_\beta$ distributions, and is expected to be the universal edge scaling limit of a range of models in random matrix theory and statistical mechanics. In this work, we provide a framework of proving convergence to the Airy$_\beta$ line ensemble, via a characterization through the pole evolution of meromorphic functions satisfying certain stochastic differential equations. Our framework is then applied to prove the universality of the Airy$_\beta$ line ensemble as the edge limit of various continuous time processes, including Dyson Brownian motions with general $\beta$ and potentials, Laguerre processes and Jacobi processes. 
\end{abstract}

\setcounter{tocdepth}{1}
\tableofcontents

\section{Introduction}
During the 18th century, De Moivre established the Gaussian distribution for sums of independent binomial variables, a concept later generalized by Laplace. Gauss popularized the central limit theorem, used for error evaluation in systems characterized by independence. 
 In recent decades, there has been increasing interest in understanding fluctuations in highly correlated systems, leading to the emergence of a different family of distributions known as Tracy-Widom$_\beta$, which are indexed by a positive parameter $\beta$.
Historically, such Tracy-Widom$_\beta$ distributions for $\beta=1,2,4$ were first observed in Random Matrix Theory, as the scaling limit of the extreme eigenvalues of the classical matrix ensembles \cite{tracy1994level,MR1385083,MR2641363,MR1737991,MR1900323, MR1863961}.
Later, such extreme eigenvalue statistics are proven to be universal, in the sense that Tracy-Widom$_\beta$ limits hold for a wide range of random matrix models, including adjacency matrices of random graphs, which are usually sparse. See e.g., \cite{MR1727234, MR2669449, MR3161313, MR2964770,
MR3034787,huang2022edge,he2024spectral,he2021fluctuations, huang2023edge,bauerschmidt2020edge}. 
Beyond matrix models, Tracy-Widom$_\beta$ distributions also appear in lots of different random systems, such as random tilings, 1+1 dimensional random growth models, exclusion processes, planar random geometry such as first/last passage percolation models, and square ice models (six-vertex models).
See e.g., \cite{MR1682248,MR1737991,amir2011probability, MR1900323,ACP,MR1969205,aggarwal2021edge, MR2165697, MR2363389,MR4421174,zhang2024cutoff,he2024boundary,imamura2022solvable, MR1845180,MR3855355,OSP,AOSP,aggarwal2024scaling,MR4541342,he2024boundary}. 
Many of these models are related to each other through their underlying integrable structures, and are in the so called Kardar-Parisi-Zhang (KPZ) universality class, a topic in probability theory that has been intensively studied in recent years. 

A motivation of the current paper is to better understand the mathematical structures behind Tracy-Widom$_\beta$ distributions, and to develop new methods of proving convergence to them.
The main object studied here is the Airy$_\beta$ line ensemble (ALE$_\beta$) $\{\cA^\beta_i(t)\}_{i\in\bN, t\in\bR}$, which can be defined as a random process on $\bN\times\bR$ or a family of continuous random processes on $\bR$, with $\beta>0$ being a parameter, and is ordered, i.e.,  $\cA^\beta_1(t)\geq \cA^\beta_2(t)\geq \cA^\beta_3(t)\geq\cdots$ for any $t$.
They are natural generalizations of the Tracy-Widom$_\beta$ distributions, akin to Brownian motions being generaliztions of Gaussian distributions, and have Tracy-Widom$_\beta$ as the one-point distribution of the top line $\cA^\beta_1$.
These ALE$_\beta$ are believed to be universal objects, in the sense of being the scaling limit of many random matrix models and interacting particle systems. However, basic properties of these processes as well as these convergences remain quite mysterious so far, except for the special setting of $\beta=2$, where a determinantal structure is present and has been largely exploited using algebraic methods.

In this paper, we take a new perspective, and our main result is a characterization of ALE$_\beta$ in terms of its Stieltjes transform and a system of stochastic differential equations (SDE). 
Our result provides a new framework to prove convergence. As some examples, we prove convergence to ALE$_\beta$ from various random processes, including the classical $\beta$-Hermite/Laguerre/Jacobi processes and their generalizations. We note that some of these were previously unknown even in the $\beta=1, 4$ cases, where such convergences can be interpreted as the joint convergence of extreme eigenvalues of correlated real/quaternion random matrices.
Beyond these, our framework should also be applicable to prove convergence to ALE$_\beta$ for many other models; and for some of them, even the Tracy-Widom$_\beta$ limits were previously unknown.
Moreover, our characterization also reveals some useful information for ALE$_\beta$, such as H\"older properties and collision of adjacent lines.

\subsection{Background} We next provide some setup, starting with some more classical processes.

\subsubsection{Edge limit of general $\beta$-ensembles}
Tracy-Widom$_\beta$ distributions for general $\beta>0$ was introduced and studied in \cite{MR2331033,MR2717319} by Edelman and Sutton, and \cite{MR3813993} by Ramirez, Rider and Vir{\'a}g, as the scaling limit of the extreme eigenvalue of Gaussian $\beta$-ensembles.
More generally, $\beta$-ensemble is a probability distribution on $n$ particle system $x_1\ge x_2 \ge \cdots \ge x_n$, with probability density:
\begin{equation}   \label{eq:betaem}
\frac{1}{Z_{n,\beta,W}}\prod_{i<j}|x_i-x_j|^\beta \prod_{i=1}^n W(x_i),
\end{equation}
where $Z_{n,\beta,W}$ is a renormalization constant, and $W\geq 0$ is the weight function.
There are three special cases referred to as the classical ensembles, which are defined by
\begin{align}\label{e:classicalensemble}
    W(x)=\left\{
    \begin{array}{cc}
        e^{-\beta x^2/4},& \text{Hermite/Gaussian ensemble}\\
        x^{\beta(m-n+1)/2-1}e^{-\beta x/2},& \text{Laguerre ensemble}\\
        x^{\beta(p-n+1)/2-1}(1-x)^{\beta(q-n+1)/2-1},& \text{Jacobi ensemble}\\
    \end{array}
    \right.
\end{align}
These classical ensembles for $\beta=1,2,4$ originated from the study of eigenvalue distributions of random matrices. They represent the joint distributions of the eigenvalues of size $n$ Gaussian, Wishart, and Jacobi matrices. These matrices and their extreme eigenvalues, with $\beta = 1$ corresponding to the real case, have been widely used in high-dimensional statistical inference (see the survey by Johnstone \cite{johnstone2006high}). Beyond random matrix theory, $\beta$-ensembles also describe the one-dimensional Coulomb gas in physics \cite{MR2641363}, and are connected to orthogonal polynomial systems \cite{konig2005orthogonal}.

As mentioned earlier, as $n \to \infty$, the distribution of the largest eigenvalues converges to the Tracy-Widom$_\beta$ distribution for $\beta = 1, 2, 4$, respectively. More generally, one can consider the edge limit, i.e., the joint scaling limit of the top $k$ eigenvalues for any arbitrary $k$, as a point process. It has been shown \cite{MR1702716, deift2007universality, deift2009random} that this edge limit does not depend on the potential function $W$, but it varies for each of $\beta = 1, 2, 4$. The $\beta = 2$ edge limit is also known as the Airy point process.

For $\beta$ other than $1, 2, 4$, obtaining such edge limit was a challenging problem, partly due to the relative lack of exact-solvable structures. 
Based on a tri-diagonal random matrix model discovered by Dumitriu and Edelman \cite{MR1936554}, this was resolved in \cite{MR3813993}, where the edge limit of Gaussian and Laguerre $\beta$-ensemble is shown to be the eigenvalues of the $\beta$ stochastic Airy operator (SAO$_\beta$), which is also called the Airy$_\beta$ point process. In particular, for each $\beta>0$, the law of the largest eigenvalue of SAO$_\beta$, i.e., the top particle in the Airy$_\beta$ point process, is then defined as the Tracy-Widom$_\beta$ distribution. 
Later, such Airy$_\beta$ point process limit is also extended to more general $W$
\cite{MR3433632,MR3253704,bekerman2015transport,bekerman2018transport}. Analogous edge limits for discrete $\beta$-ensembles have been derived in \cite{MR3987722}.

\subsubsection{Airy line ensemble}
In another direction, the Tracy-Widom$_2$ distribution and the ($\beta=2$) Airy point process are extended to the Airy line ensemble (ALE), an ordered family of random processes $\{\cA_i(t)\}_{i\in\bN, t\in\bR}$, where each $\cA_i$ is continuous, and they are jointly stationary in the $\bR$ direction (see \Cref{fig:ALE}).
ALE was introduced by Prah\"{o}fer and Spohn \cite{SIDP}, as the scaling limit for the multi-layer polynuclear growth (PNG) model from the Kardar-Parisi-Zhang (KPZ) universality class.
The top line $\cA_1$ is known as the stationary Airy$_2$ process, whose one-point marginal is the Tracy-Widom$_2$ distribution; and for any $t\in\bR$, $\{\cA_i(t)\}_{i\in\bN}$ is the Airy point process.
ALE plays a central role in KPZ, in particular through the construction of the directed landscape \cite{dauvergne2018directed}. (See also \cite{DV} for computing passage times in the directed landscape from ALE.) 

ALE is particularly useful in KPZ, partly due to its Brownian Gibbs property, which was recognized by Corwin and Hammond \cite{corwin2014brownian}. Specifically, for any positive integer $m$ and any interval $(x, y)$, conditional on $\{\cA_i(t)\}_{(i, t)\in (\bN\times \bR) \setminus (\{1, \ldots, m\}\times (x,y))}$, the processes $\cA_1(t)-t^2, \ldots, \cA_m(t)-t^2$ for $t\in [x,y]$ are Brownian bridges with the given endpoint values at $x$ and $y$, conditional on $\cA_1(t)\ge \cdots \ge \cA_m(t)$ for each $t\in [x,y]$.
This fact is later widely used as a powerful tool to study ALE and many KPZ class models. Aggarwal and the first-named author provided a strong characterization of ALE, demonstrating that ALE (minus a parabola) is the only random process on $\bN\times\bR$ with the Brownian Gibbs property as well as approximating a parabola \cite{aggarwal2023strong}. 
Such a strong characterization would be a powerful tool to prove convergence to ALE and establishing KPZ universality for various models; see e.g.,~\cite{aggarwal2024scaling}.

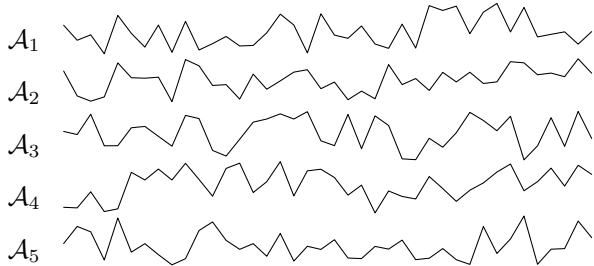
\begin{figure}[!ht]
	\centering
	\begin{tikzpicture}[line cap=round,line join=round,>=triangle 45,x=1.8cm,y=2cm]
	\clip(-0.45,-0.15) rectangle (4.15,1.77);

	\draw (0.,0.1) node[anchor=east]{$\cA_5$};
	\draw (0.,0.45) node[anchor=east]{$\cA_4$};
	\draw (0.,0.8) node[anchor=east]{$\cA_3$};
	\draw (0.,1.15) node[anchor=east]{$\cA_2$};
	\draw (0.,1.5) node[anchor=east]{$\cA_1$};

	\draw plot coordinates {(0.1, 0.1494) (0.2, 0.2631) (0.3, 0.2274) (0.4, 0.0399) (0.5, 0.3195) (0.6, 0.0923) (0.7, 0.1487) (0.8, 0.0753) (0.9, 0.0076) (1.0, 0.0486) (1.1, 0.2360) (1.2, 0.2931) (1.3, 0.1704) (1.4, 0.1108) (1.5, 0.1486) (1.6, 0.0648) (1.7, 0.2165) (1.8, 0.0328) (1.9, 0.1312) (2.0, 0.1102) (2.1, 0.1751) (2.2, 0.0529) (2.3, 0.0492) (2.4, 0.1304) (2.5, 0.1124) (2.6, 0.1723) (2.7, 0.0423) (2.8, 0.0516) (2.9, 0.1159) (3.0, 0.0418) (3.1, 0.0106) (3.2, 0.2627) (3.3, 0.0834) (3.4, 0.1796) (3.5, 0.3323) (3.6, 0.0115) (3.7, 0.1117) (3.8, 0.1157) (3.9, 0.2990) (4.0, 0.1884) };
	\draw plot coordinates {(0.1, 0.3905) (0.2, 0.3866) (0.3, 0.4938) (0.4, 0.3622) (0.5, 0.3802) (0.6, 0.6232) (0.7, 0.5722) (0.8, 0.6449) (0.9, 0.5720) (1.0, 0.6830) (1.1, 0.5755) (1.2, 0.4644) (1.3, 0.6492) (1.4, 0.6841) (1.5, 0.4888) (1.6, 0.5605) (1.7, 0.6943) (1.8, 0.4665) (1.9, 0.6330) (2.0, 0.6437) (2.1, 0.5887) (2.2, 0.4715) (2.3, 0.5379) (2.4, 0.3532) (2.5, 0.4991) (2.6, 0.4626) (2.7, 0.4460) (2.8, 0.5951) (2.9, 0.5151) (3.0, 0.4300) (3.1, 0.5039) (3.2, 0.5529) (3.3, 0.6249) (3.4, 0.6784) (3.5, 0.5005) (3.6, 0.5610) (3.7, 0.6510) (3.8, 0.5315) (3.9, 0.6700) (4.0, 0.6095) };
	\draw plot coordinates {(0.1, 0.8952) (0.2, 0.8745) (0.3, 1.0096) (0.4, 0.7997) (0.5, 0.8003) (0.6, 0.9195) (0.7, 0.9284) (0.8, 0.8623) (0.9, 0.7971) (1.0, 1.0020) (1.1, 0.9806) (1.2, 0.7688) (1.3, 0.7313) (1.4, 0.8429) (1.5, 0.9570) (1.6, 0.9778) (1.7, 1.0120) (1.8, 0.9778) (1.9, 1.0266) (2.0, 0.8328) (2.1, 0.8007) (2.2, 1.0083) (2.3, 0.7809) (2.4, 0.9987) (2.5, 0.9346) (2.6, 0.7127) (2.7, 0.7072) (2.8, 0.8498) (2.9, 0.7900) (3.0, 0.8866) (3.1, 1.0220) (3.2, 0.9680) (3.3, 0.8997) (3.4, 0.9961) (3.5, 0.7062) (3.6, 0.7993) (3.7, 0.9865) (3.8, 0.7958) (3.9, 1.0282) (4.0, 0.8479) };
	\draw plot coordinates {(0.1, 1.2969) (0.2, 1.1313) (0.3, 1.0965) (0.4, 1.1246) (0.5, 1.3511) (0.6, 1.2525) (0.7, 1.2493) (0.8, 1.2560) (0.9, 1.0934) (1.0, 1.3732) (1.1, 1.3316) (1.2, 1.2009) (1.3, 1.2065) (1.4, 1.1104) (1.5, 1.2764) (1.6, 1.1760) (1.7, 1.2314) (1.8, 1.2896) (1.9, 1.3057) (2.0, 1.1796) (2.1, 1.2235) (2.2, 1.1063) (2.3, 1.1633) (2.4, 1.1124) (2.5, 1.3398) (2.6, 1.2068) (2.7, 1.2470) (2.8, 1.1677) (2.9, 1.2881) (3.0, 1.2244) (3.1, 1.2890) (3.2, 1.2143) (3.3, 1.2250) (3.4, 1.3589) (3.5, 1.3516) (3.6, 1.2720) (3.7, 1.2833) (3.8, 1.2593) (3.9, 1.3781) (4.0, 1.2824) };
	\draw plot coordinates {(0.1, 1.6016) (0.2, 1.5017) (0.3, 1.5378) (0.4, 1.4118) (0.5, 1.6670) (0.6, 1.5475) (0.7, 1.4561) (0.8, 1.6029) (0.9, 1.4216) (1.0, 1.6277) (1.1, 1.4369) (1.2, 1.4794) (1.3, 1.5269) (1.4, 1.4637) (1.5, 1.4672) (1.6, 1.5512) (1.7, 1.6755) (1.8, 1.6032) (1.9, 1.4173) (2.0, 1.6778) (2.1, 1.5360) (2.2, 1.5146) (2.3, 1.5966) (2.4, 1.4767) (2.5, 1.4479) (2.6, 1.6089) (2.7, 1.4497) (2.8, 1.7320) (2.9, 1.6999) (3.0, 1.7287) (3.1, 1.5454) (3.2, 1.6884) (3.3, 1.7463) (3.4, 1.5572) (3.5, 1.7202) (3.6, 1.5239) (3.7, 1.5411) (3.8, 1.5576) (3.9, 1.4759) (4.0, 1.5593) };

	\end{tikzpicture}
	\caption{An illustration of ALE}   \label{fig:ALE}
\end{figure}

\subsection{Airy$_\beta$ line ensemble}
From the success of ALE, the next question is to construct a time dependent evolution for Tracy-Widom$_\beta$ (and more generally Airy$_\beta$ point process) for any $\beta>0$.
Following \cite{GXZ} where this is formally introduced, we call it the Airy$_\beta$ line ensemble (ALE$_\beta$). 
There are several problems in this program:
\begin{itemize}
    \item \textit{Construction.} What should it be? How to construct it?
    \item \textit{Description.} What are its properties? Ideally, can some precise information be given?
    \item \textit{Universality.} Why is ALE$_\beta$ natural and interesting? Can it be shown to be the universal scaling limit of many natural random processes, as the $\beta=2$ case?
\end{itemize}
Towards these questions, there have been many results focusing on different aspects of ALE$_\beta$ (some tracing back to the studies of Airy$_\beta$ point process, or for special $\beta$): infinite-dimensional SDE \cite{landon2020edge,osada2013interacting,osada2020infinite,kawamoto2022infinite,osada2013interactingII,osada2024infinite}, correlation function \cite{MR1969205,MR1682248,MR2018275,SIDP,katori2009zeros,osada2016strong}, Laplace transform \cite{MR1727234,MR1802530,okounkov2002generating,MR3403994,jeong2016limit,GXZ}, and tridiagonal matrices \cite{edelman2024limit,ViragICM}.
In this paper, we provide a different perspective using Stieltjes transform, tailored to the universality problem.

We now give a more detailed account on the state of the art, and further motivate our results.

\subsubsection{The edge limit construction and convergence}

One way to construct ALE$_\beta$ for general $\beta>0$ is by taking the edge scaling limit of the general $\beta$ Dyson Brownian motion (DBM):
\begin{equation} \label{eq:DBM}
		\rd \lambda_i (t) = \sqrt{\frac{2}{\beta}} \rd B_i (t)+
\sum_{\substack{1 \le j \le n \\ j \ne i}}	\frac{\rd t}{\lambda_i (t) - \lambda_j (t)} -\frac{1}{2}\lambda_i(t)\rd t,    
\end{equation}
where $n\in\bN$ and $\{B_i\}_{i=1}^n$ are independent two-sided Brownian motions.
There is a solution to this SDE, such that $\{\lambda_i(t)\}_{i=1}^n$ for any $t$ is a Gaussian $\beta$-ensemble; and this is known as the stationary DBM of size $n$ with parameter $\beta$. (See e.g., \cite{MR2760897} for some more backgrounds on DBM.)

\begin{definition}[Airy$_\beta$ Line Ensemble]\label{def:ALE}  \label{def:alebeta}
ALE$_\beta$ is an ordered families of continuous random processes
\begin{align}
    \cA^\beta_1(t)\geq \cA^\beta_2(t)\geq \cA^\beta_3(t)\geq \cdots, \text{ for } -\infty<t<\infty,
\end{align}
which is the edge limit of the stationary DBM \eqref{eq:DBM}, i.e., the limit of $\{n^{1/6}\lambda_i(tn^{-1/3})-2n^{2/3}\}_{i\in \bN, t\in \bR}$ as $n\to\infty$, in the uniform in compact topology. Namely for every compact $K\subset\bR$ and every finite $I\subset\{1,2,3,\cdots\}$,
\begin{align}
\lim_{n\to\infty}\ \max_{i\in I}\ \sup_{t\in K}\, \bigl|n^{1/6}\lambda_i(tn^{-1/3})-2n^{2/3}-\cA^\beta_i(t)\bigr| \;=\; 0.
\end{align}

\end{definition}
Note that for any fixed $t$, $\{\cA^\beta_i(t)\}_{i\in\bN}$  is the Airy$_\beta$ point process, and $\cA^\beta_1(t)$ follows the Tracy-Widom$_\beta$ distribution.
The construction of ALE$_\beta$ is to prove the above convergence.
\begin{thm}[\cite{SIDP,MR2018275, corwin2014brownian,MR3403994,landon2020edge,GXZ}] \label{thm:defalebeta}
For any $\beta>0$, the limit in \Cref{def:alebeta} exists. In particular, ALE$_{\beta=2}$ is ALE.
\end{thm}
This theorem was proved in a series of works, first for $\beta=2$ in \cite{SIDP,MR2018275, corwin2014brownian} (and the identification with ALE is through Fredholm determinant formulas). 
Then it was proved for $\beta=1$ by Sodin \cite{MR3403994} using related random matrix models, and the arguments should also go through for $\beta=4$.
Besides, in \cite{gorin2020universal} by Gorin and Kleptsyn, a convergence of DBM where $\beta\to\infty$ together with $n\to\infty$ is proved. 
For more general $\beta$, in \cite{landon2020edge} Landon proved \Cref{thm:defalebeta} for $\beta\ge 1$ via a coupling argument.
Finally, in \cite{GXZ} by Gorin, Xu, and the second-named author, \Cref{thm:defalebeta} was first proved for any $\beta>0$.
Additionally, the arguments in this paper provide an alternative proof of \Cref{thm:defalebeta}.

We note that in \cite{GXZ}, the definition of ALE$_\beta$ is different from \Cref{def:alebeta} and in terms of exact formulas, and it is then shown in \cite{GXZ} that the stationary DBM converges to the defined ALE$_\beta$.
Historically, the theory of Tracy-Widom$_\beta$ distributions used to largely rely on explicit formulas, based on determinantal/Pfaffian structures or matrix models (see e.g., \cite{tracy1994level,MR2187952,MR3729037,MR1727234,MR3813993}).
As for ALE$_\beta$, formulas used to be available only for $\beta=2$, i.e., ALE, in \cite{SIDP}; and for $\beta=\infty$ in \cite{gorin2020universal}.
There are various challenges in obtaining precise formulas for general $\beta$, primarily due to the lack of algebraic structures in this generality.
In \cite{GXZ} they are overcome by considering the moments and Laplace tranforms, and using Dunkl operators.

We remark that other dynamics also preserve the Airy$_\beta$ point process. For instance, one may consider the sequence of successive minors of the random tridiagonal matrix representation of $\beta$-ensembles. In \cite{lambert2020strong,ashbury2022random}, the edge limit of this process is taken, yielding a dynamics that can be described in terms of the characteristic polynomial (which is termed the stochastic Airy function).

\subsubsection{Uniqueness and universality}
As indicated by the convergence of both DBM and Gaussian corners process to the same limit, i.e., ALE$_\beta$, in \cite{GXZ}, it is natural to expect that  ALE$_\beta$ is also the scaling limit of many other well-known processes.
Some examples include DBM with general potentials, non-intersecting random walks \cite{gorin2019universality,konig2002non, huang2017beta}, various other $\beta$–corners processes \cite{MR3418747, borodin2015general,MR3877550} and measures on Gelfand-Tsetlin patterns \cite{bufetov2018asymptotics,petrov2015asymptotics}, Macdonald processes \cite{borodin2014macdonald}, and $(q,\kappa)$-distributions on lozenge tilings \cite{borodin2010q,dimitrov2019log}.

A robust and general approach to establishing convergence involves a suitable characterization of ALE$_\beta$. Specifically, this means identifying easily verifiable properties of ALE$_\beta$ and demonstrating that these properties uniquely determine ALE$_\beta$. To prove convergence, it would then suffice to establish tightness and confirm that any subsequential limit satisfies these properties.

For ALE$_\beta$ with $\beta=2$, an elegant characterization is the Brownian Gibbs property \cite{aggarwal2023strong} as mentioned above.
However, this does not hold for any $\beta\neq 2$.
The next natural candidate of characterization would be an `infinite dimensional DBM', by taking $n\to\infty$ in \eqref{eq:DBM}.
For example, \cite{landon2020edge} shows that ALE$_\beta$ (for $\beta\ge 1$), i.e., the edge limit of finite dimension DBM, is a solution to such an infinite dimensional DBM in a weak sense. 

However, there are several technical challenges in developing a characterization and convergence framework along this direction.
First, particles (i.e., those $\lambda_i$ in \eqref{eq:DBM}) may collide or adjacent particles may get too close, leading to singularities in the drift $1/(\lambda_i(t)-\lambda_j(t))$ term.
In \cite{landon2020edge}, for $\beta\ge 1$ such singularities are ruled out using existing level replusion estimates (see \cite[Theorem 2.2]{landon2020edge} and \cite[Theorem 3.2]{MR3253704}), which are known only for stationary DBM whose laws are given by $\beta$-ensembles. Deriving such estimates for other models could be difficult. Moreover, for $\beta<1$ collisions are inevitable.
Second, the long-range interactions introduce additional complications when analyzing infinitely many particles. In fact, even the well-posedness of the infinite-dimensional DBM starting from a fixed intial condition appears to be absent from the literature, except for the cases $\beta=1,2,4$ \cite{osada2013interacting,osada2020infinite,katori2009zeros,kawamoto2018finite} where the underlying algebraic structure is used. Note that for an analogous infinite-dimensional SDE corresponding to the bulk limit of DBM, such well-posedness has been achieved for $\beta\ge 1$ (see \cite{katori2010non,osada2012infinite,tsai2016infinite}). A key property used in the bulk setting is the cancellation of particle interactions from left and right, and that is absent at the edge. 
For example, in the recent paper~\cite{TZisde}, which studies a random matrix eigenvalue dynamics similar to DBM, an infinite-dimensional SDE obtained from the edge limit is shown to admit multiple solutions.
As a result, for the infinite-dimensional DBM, both proving the uniqueness of solution and verifying it for any subsequential limit face various barriers.

To overcome these difficulties, in this paper we take an alternative approach, and characterize ALE$_\beta$ as the pole dynamics of meromorphic functions, satisfying a function-valued SDE. 
In other words, we characterize ALE$_\beta$ via the dynamics of its Stieltjes transform.
This method completely circumvents the issue of long-range interactions and collisions, and is applicable for any $\beta>0$.

\subsection{Main characterization result}

In the rest of this paper, we fix $\beta>0$.
We study (infinite) line ensembles, by which we mean ordered families of continuous random processes, denoted by $\{\bmx(t)\}_{t\in I}=\{x_i(t)\}_{i\in\bN, t\in I}$, for $I=\bR$ or any interval, satisfying $x_1(t)\geq x_2(t)\geq x_{3}(t)\geq \cdots$.

Below we use $\bH$ to denote the open upper half complex plane.
\begin{definition}
   A function $Y:\bH \to \bH \cup \bR$ that is holomorphic is called a \emph{Nevanlinna function}.
Any Nevanlinna function $Y$ has the following integral representation
\begin{align}\label{e:Nevanlinna_representation}
    Y(w)=b+cw+\int_\bR\left( \frac{1}{x-w}-\frac{x}{1+x^2}\right)\rd \mu(x),
\end{align}
where $b,c\in \bR$, $c\geq 0$, and $\mu$ is a Borel measure satisfying the growth condition
\[
    \int_\infty^\infty \frac{\rd \mu(x)}{1+x^2}<\infty.
\]
\end{definition}

\begin{remark}\label{r:defOH}
Nevanlinna functions can be viewed as elements in the following function space 
\[ \mathcal{O}(\mathbb{H}) \coloneqq \{ f:\mathbb{H}\to\mathbb{C} \text{ holomorphic} \}. \]
If \(f\in\mathcal{O}(\mathbb{H})\), then for every \(k\ge 1\) its derivatives
\(\partial_z^{\,k} f\) also belong to \(\mathcal{O}(\mathbb{H})\).
We can equip $\mathcal{O}(\mathbb{H})$ with the topology of locally uniform convergence, i.e. a sequence \(\{f_i\}_{i\geq 1}\subset \mathcal{O}(\mathbb{H})\) converges to \(f\in\mathcal{O}(\mathbb{H})\) if and only if \(f_i \to f\) uniformly on every compact \(K\subset\mathbb{H}\), as $i\rightarrow \infty$. This makes $\mathcal{O}(\mathbb{H})$ a Fr{\'e}chet space. For any $f\in \mathcal O(\bH)$, we define values in the lower half–plane by Schwarz reflection $f(\overline w)=\overline{f(w)}$ and interpret boundary values on $\mathbb R$ as non–tangential limits.
\end{remark}

\begin{definition}[Configuration]
A \emph{configuration} on $\bR$ is a Radon measure of the form
\[
\mu = \sum_{x\in P} \delta_{x},
\]
where $P$ is a finite or countable multiset consisting of elements in $\bR$, and
\[
\mu(K) < \infty \quad \text{for every compact set } K \subset \bR.
\]
That is, $\mu$ is a purely atomic Radon measure consisting of unit
masses at its atoms, with only finitely many atoms in each bounded region.
\end{definition}

\begin{definition}
 We call a Navanlinna function $Y(w)$ \emph{particle-generated} if the associated measure $\mu=\sum_{x\in P}\delta_x$ is a configuration. With the expression \eqref{e:Nevanlinna_representation}, we view $Y(w)$ as a meromorphic function on $\bC$, satisfying $Y(\overline w)=\overline{Y(w)}$, with $P$ being all the poles, and each residue equals $1$\footnote{More precisely, for each $x\in \bR$, the residue at $x$ equals the multiplicity of $x$ in $P$.}
\end{definition}

We next recall the Airy function, denoted by~$\Ai$, which is a special function appearing in various areas of mathematics and physics.
It can be defined as the entire function, such that for $x\in \bR$, it is give by the improper Riemann integral:
\[
\Ai(x) = \frac{1}{\pi} \lim_{y\to \infty} \int_0^y \cos\left(\frac{t^3}{3}+xt\right) dt.
\]
It also solves the Airy equation: $\Ai''(w)-w\Ai(w)=0$, with $\Ai(w)\to 0$ as $w\to \infty$ along $\bR_+$. 
All the zeros of $\Ai$ are on the real line, and are all negative, and we denote them as $0>\fa_1>\fa_2>\fa_3>\cdots$. See e.g., \cite[Chapter 9]{NIST:DLMF} for more background and properties on the Airy function.

It is known that $\fa_i$ is around $-(3i\pi/2)^{2/3}$. More precisely, for any $i\in \bN$ we have (see e.g., \cite[eq.(9.9.6) and eq.(9.9.18)]{NIST:DLMF})
\begin{align}\label{e:aklocation}
    \left|\fa_i+\left(\frac{3\pi i}{2}\right)^{2/3}\right|\lesssim i^{-1/3}.
\end{align}
We remark that 
$-\Ai'(w)/\Ai(w)$ is a particle-generated Navanlinna function, with the associated configuration $\mu=\sum_{i=1}^\infty\delta_{\fa_i}$.
More precisely, its Navanlinna representation is
\begin{align}  \label{eq:weairy}
-\frac{\Ai'(w)}{\Ai(w)}=\sum_{i=1}^\infty\left(\frac{1}{\fa_i-w} -\frac{1}{\fa_i}\right) -\frac{\Ai'(0)}{\Ai(0)}
\end{align}
We also have the following square root behavior \footnote{Here and throughout this paper, $\sqrt{w}$, or any rational power of $w\in \bC$, is taken to be the branch on $\bC\setminus\bR_-$}
\begin{align}\label{e:aiasymp}
\left| \frac{\Ai'(w)}{\Ai(w)} + \sqrt{w} \right|
\lesssim |w|^{-1},
\end{align}
for any $w\in\bC$ with $|w|$ large enough, and $|\arg(w)| < \pi - |w|^{-9/7}$.
Both \eqref{eq:weairy} and \eqref{e:aiasymp} will be proved in \Cref{ss:afas}.

If $\{\bmx(t)\}_{t\in\bR}$ were ALE$_\beta$, its one time slice should be the Airy$_\beta$ point process. Hence, each $x_i(t)$ should be close to the $i$-th airy zero $\fa_i$. (See \cite[Corollary 5.3]{BPLE} for $\beta=2$. For general $\beta>0$, this actually follows from our \Cref{p:xrigidity} and \Cref{p:tight} below.)
We can then introduce the normalized Stieltjes transform of each time slice of the infinite line ensemble as
\begin{align}\label{e:normalized_ST}
    \sum_{i=1}^\infty\left(\frac{1}{x_i(t)-w} -\frac{1}{\fa_i}\right) -\frac{\Ai'(0)}{\Ai(0)},
\end{align}
where we used the same normalization as in \eqref{eq:weairy}. The above normalized Stieltjes transform is a particle-generated Navanlinna function, with the associated configuration $\mu=\sum_{i=1}^\infty\delta_{x_i(t)}$.
If $x_i(t)$ are close to $\fa_i$,  the difference of \eqref{e:normalized_ST} and \eqref{eq:weairy} goes to zero as $w\rightarrow \infty$. 
Therefore, the normalized Stieltjes transform \eqref{e:normalized_ST} should exhibit similar asymptotic behaviors as $-\Ai'(w)/\Ai(w)$, which grows like $\sqrt{w}$ as $w\rightarrow \infty$.

We next give a more precise formulation of such asymptotic behaviors.

\begin{definition}  \label{defn:nvali}
For $\fd,C_*>0$, a Nevanlinna function $Y$ is \emph{$(\fd, C_*)$-Airy-like}, or simply \emph{Airy-like}, if 
\begin{enumerate}
\item \label{i:it1} it is particle-generated, and all the poles are $\le C_*$;
\item \label{i:it2} for all $w$ with $\Im[w]\geq C_*\sqrt{\Re[w]\vee 0+1}$,
\begin{align}\label{e:Yw-w}
|Y(w)-\sqrt{w}|\leq \frac{C_*\Im[\sqrt{w}]^{1-\fd}}{\Im[w]}.
\end{align}
\end{enumerate}
\end{definition}
In particular, thanks to \eqref{eq:weairy}, the function $-\frac{\Ai'(w)}{\Ai(w)}$ is Airy-like.

\begin{remark}  \label{rem:1}
The condition \Cref{i:it2} implies the existence of infinitely many poles $x_1\ge x_2\ge \cdots$.
As we will see shortly (\Cref{lem:parti-clo}), bounds similar to \Cref{i:it1} and \Cref{i:it2} (but may with a different domain for $w$) would imply that the density of these poles would be close to $\sqrt{-x}$ in $\bR_-$.
Such density closeness can be phrased as quantitative bounds on the distances between the poles and zeros of the Airy function $\fa_1>\fa_2>\cdots$, as stated in \Cref{lem:parti-clo}, and will be frequently used in our proofs.
Therefore, as a slight misuse of notions, we will refer to such density closeness as \emph{Airy-zero approximation}. 
\end{remark}

Take a family of random functions $\{Y_t\}_{t\in\bR}$. We next state two assumptions.

\begin{assumption}\label{a:Infinite}
For each $t\in \bR$, $Y_t$ is a particle-generated Nevanlinna function. Moreover, 
there exists a (deterministic) $\fd>0$, a sequence $t_1, t_2, \cdots \to -\infty$,
and a tight family of random numbers $\{C_{*,j}\}_{j\in\bN}$, so that each $Y_{t_j}$ is $(\fd, C_{*,j})$-Airy-like.
\end{assumption}

\begin{assumption}\label{a:SDE}
For each $t\in\mathbb R$, $Y_t$ is a Nevanlinna function and the path $t\mapsto Y_t$ is continuous in the locally uniform convergence topology on $\mathcal O(\mathbb H)$ (see \Cref{r:defOH}). 
Moreover, for every $w\in\mathbb C\setminus\mathbb R$ the process $(Y_t(w))_{t\in\mathbb R}$ satisfies the following semimartingale decomposition
\begin{equation}\label{e:defYt}
  Y_t(w)-Y_s(w)
  = \big(M_t(w)-M_s(w)\big)
  + \int_s^t\!\left(
      \frac{2-\beta}{2\beta}\,\partial_w^2 Y_u(w)
      + \frac{1}{2}\partial_w (Y_u(w)^2)
      - \frac{1}{2}
    \right)\,\mathrm d u .
\end{equation}
Here $M_t$ is an $\mathcal O(\mathbb H)$–valued {continuous martingale}, in the sense that for any finite set 
$\{w_1,\dots,w_k\}\subset\mathbb C\setminus\mathbb R$ the vector 
$\big(M_t(w_1),\dots,M_t(w_k)\big)$ is a $k$–dimensional continuous complex martingale whose quadratic (co)variation is absolutely continuous in $t$ and given by 
\footnote{Here and throughout this paper, for any two processes $f, g:\bR\to \bC$, we use $\langle f\rangle_t$ (resp.~$\langle f, g\rangle_t$) to denote the process such that $\langle f\rangle_0 = 0$ (resp.~$\langle f, g\rangle_0=0$), and  $\langle f\rangle_t - \langle f\rangle_s$ (resp.~$\langle f, g\rangle_t-\langle f, g\rangle_s$) is the (complex-valued) quadratic variation of $f$ (resp.~quadratic covariantion of $f$ and $g$) in $[s ,t]$, for any $s<t$.}
\begin{align}\label{eq:qvmwwp}
\begin{split}
\langle M(w)\rangle_t 
  &= \frac{1}{3\beta}\int_0^t \partial_w^{3} Y_s(w)\,\mathrm d s,\\
\langle M(w), M(w') \rangle_t 
  &= \frac{2}{\beta}\int_0^t
     \partial_w \partial_{w'} \!\left( \frac{Y_s(w)-Y_s(w')}{\,w-w'\,} \right)\,\mathrm d s,
     \qquad w\neq w',\ \ w,w'\in\mathbb C\setminus\mathbb R .
\end{split}
\end{align}
for $w \neq w', w,w'\in \bC\setminus \bR$, and all $w$–derivatives above are taken in the holomorphic variable.
\end{assumption}

\begin{remark}
The semimartingale decomposition \eqref{e:defYt} can be written in differential (It\^o) form as
\begin{align}\label{e:SDE_form}
    \rd  Y_t(w)
= \rd M_t(w)+\left(\frac{2-\beta}{2\beta }\del^2_w Y_t(w)+\frac{1}{2}\del_w (Y_t(w)^2)-\frac{1}{2}\right)\rd t,
\end{align}
We do not study the well-posedness of \eqref{e:SDE_form}; throughout, it is understood purely as shorthand for the integral semimartingale identity \eqref{e:defYt} (all equalities being in the sense of semimartingale differentials).
\end{remark}

\begin{remark}
   $\{M_t\}_{t\in \bR}$ is an $\mathcal O(\mathbb H)$–valued {continuous martingale}. 
For any continuous linear functional $\mathcal{L}:\mathcal{O}(\mathbb{H})\to\mathbb{C}$, the process 
$\{\mathcal{L}(M_t)\}_{t\in \bR}$ is a complex–valued martingale. 
In particular, for any $w\in\mathbb{H}$ the evaluation $f\mapsto f(w)$ and, by Cauchy’s estimates, the derivative 
functionals $f\mapsto \del_z^k f(w)$ are continuous; hence $M_t(w)$ and $(\partial_z^{\,k}M_t)(w)$ (for any $k\ge1$) are complex–valued martingales.

\end{remark}
\begin{remark}\label{r:conjugate}
In general, the quadratic variation does not determine a complex martingale. 
In our setting we impose the conjugation symmetry 
$\overline{M_t(w)} = M_t(\overline{w})$ for $w\in\mathbb C\setminus\mathbb R$. 
Together with \eqref{eq:qvmwwp}, this yields the mixed bracket
\begin{equation}\label{eq:qvmwwp_conjugate}
  \langle M(w), \overline{M(w')}\rangle_t
  = \langle M(w), M(\overline{w'})\rangle_t
  = \frac{2}{\beta}\int_0^t
    \partial_w \partial_{\overline{w'}} \!\left(
      \frac{Y_s(w)-Y_s(\overline{w'})}{\,w-\overline{w'}\,}
    \right)\,\mathrm d s,
\end{equation}
which is simply \eqref{eq:qvmwwp} with $w'$ replaced by $\overline{w'}$. 
The relations \eqref{eq:qvmwwp} and \eqref{eq:qvmwwp_conjugate}, together with the symmetry $\overline{M_t(w)}=M_t(\overline{w})$, determine the finite–dimensional laws of the complex martingale $M$.
\end{remark}

\begin{remark}We now explain where this SDE \eqref{e:SDE_form} comes from.
Take the $n$ dimensional stationary DBM $\{\lambda_i\}_{i=1}^n$ with parameter $\beta$, i.e., the stationary solution to \eqref{eq:DBM}. Let $m_t$ be the Stieltjes transform of  $\{\lambda_i(t)\}_{i=1}^n$, i.e., 
\[
m_t(z)=\sum_{i=1}^n \frac{1}{\lambda_i(t)-z},\quad z\in \bC.
\]
Then one can use It{\^o}'s formula to write out an SDE satisfied by $m_t$; by taking an appropriate scaling limit from there, one gets the SDE \eqref{e:defYt}. More details on this derivation can be found in \Cref{s:process}.
\end{remark}

Our main result states that these two assumptions are sufficient to determine ALE$_\beta$ uniquely.
\begin{thm}\label{t:characterize}
For any $\{Y_t\}_{t\in\bR}$ satisfying \Cref{a:Infinite} and \Cref{a:SDE}, its poles give a line ensemble, which has the same law as ALE$_\beta$.
\end{thm}
Several discussions are in line.
\smallskip

\noindent\textit{(i) Essentiality of \Cref{a:Infinite}  (in characterizing ALE$_\beta$).} Both \eqref{i:it1} and \eqref{i:it2} in \Cref{defn:nvali} are necessary:
without \eqref{i:it1}, the line ensemble may be ALE$_\beta$ plus some additional lines (see \cite{PWC,dimitrov2024airy} for an example in the $\beta=2$ setting); while \eqref{i:it2} rules out the possibility that the line ensemble is ALE$_\beta$ shifted by a (deterministic or random) constant. 
As already mentioned in \Cref{rem:1}, we’ve aimed to make \Cref{a:Infinite} as minimal as possible to ensure broad applicability of our convergence framework. As will be seen in \Cref{s:process}, Assumption 1.4 is straightforward to verify in these examples.
\smallskip

\noindent\textit{(ii) DBM convergence.} Our proof of \Cref{t:characterize} does not a priori assume the convergence at the edge of stationary DBM.
Instead, in \Cref{s:process}, we show the tightness at the edge of stationary DBM, and that any subsequential limit satisfies \Cref{a:Infinite} and \Cref{a:SDE}. Therefore, we essentially provide an alternative construction of ALE$_\beta$.
\smallskip

\noindent\textit{(iii) Stationarity.} We also note that in \Cref{t:characterize}, we do not assume that $Y_t$ is stationary. Rather, it is a consequence of the theorem that the poles of $Y_t$ converge to ALE$_\beta$, which is stationary, and hence $Y_t$ is stationary as well. Our proof of \Cref{t:characterize} in fact establishes a natural relaxation for the SDE \eqref{e:defYt}. Specifically, for a family of random particle-generated Nevanlinna functions $\{Y_t\}_{t\geq0}$, if there is a (deterministic) $\fd>0$ and a random number $C$ such that $Y_0$ is $(\fd, C)$-Airy-like, and $\{Y_t\}_{t\geq 0}$ satisfies \Cref{a:SDE}, then for $T\rightarrow \infty$, the poles of $\{Y_t\}_{t\geq T}$ converge to ALE$_\beta$, under the uniform in compact topology. 

\smallskip

\noindent\textit{(iv) Stieltjes transform and poles dynamics.} 
Stieltjes transforms and Nevanlinna functions have been widely used to investigate and characterize eigenvalue distributions of random matrix ensembles. See \cite{MR2871147} for studies on eigenvalue rigidity, \cite{erdHos2012bulk,aizenman2015ubiquity} for bulk limits, and \cite{MR3161313,MR3034787,BGS} for edge limits.

The concept of characterizing the evolution of interacting particle systems through the pole dynamics of meromorphic functions has been explored previously. In integrable systems, it has been demonstrated that the movement of poles in certain solutions of various nonlinear PDEs can be formally linked to the dynamics of particle systems interacting through simple two-body potentials. This discovery was initially made in \cite{choodnovsky1977pole} for equations such as the Korteweg-de Vries and Burgers-Hopf equations, and in \cite{moser1976three} for specific integrable Hamiltonian systems. Subsequently, these observations were extended to include elliptic solutions of equations such as the Kadomtsev-Petviashvili equation \cite{krichever1980elliptic}, the Korteweg–de Vries equation \cite{deconinck2000pole}, the Kadomtsev-Petviashvili hierarchy \cite{zabrodin2020kp}, and the Toda lattice hierarchy \cite{prokofev2021elliptic}.
Our results can be interpreted as a stochastic counterpart to these findings, wherein ALE$_\beta$ is characterized as the pole evolution of the SDE \eqref{e:defYt}.

{We also comment that other than pole dynamics, interacting particle systems such as the Gaussian $\beta$ ensemble, and their dynamics, have also be analyzed via the zeros of the corresponding characteristic polynomials, in e.g., \cite{lambert2020strong,ashbury2022random}.
Note that Stieltjes transforms are the derivative of the logarithm of characteristic polynomials, therefore $Y_t$ in \Cref{t:characterize} should be the derivative of the logarithm of the stochastic Airy function from \cite{lambert2020strong}. (Although the dynamics in \cite{lambert2020strong} differ from ALE$_\beta$, as already mentioned.) }

\subsection{Convergence framework}
Given the characterization presented in \Cref{t:characterize}, to prove convergence to ALE$_\beta$, it suffices to 
\begin{itemize}
    \item[(1)] establish the tightness of the Stieltjes transforms of the empirical particle density at the microscopic scale, and 
    \item[(2)]  verify that the scaling limit is Airy-like, and satisfies the SDE \eqref{e:defYt}.
\end{itemize}
As a demonstration of this approach, we prove the convergence to ALE$_\beta$ for several continuous interacting particle systems.
We next give the formal statement of our result.

We use a strong topology of uniform in compact convergence for line ensembles. More precisely, for a sequence of ordered families of functions $\{f_i^{(1)}(t)\}_{i\in\bN, t\in\bR}, \{f_i^{(2)}(t)\}_{i\in\bN, t\in\bR}, 
\ldots$, they converge to $\{f_i(t)\}_{i\in\bN, t\in\bR}$ under the uniform in compact topology, if for each $i\in\bN$, $\lim_{n\to\infty} f_i^{(n)} = f_i$ uniformly in any compact interval.
\begin{thm}\label{t:convergence_Airy}
ALE$_\beta$ is the edge scaling limit of stationary DBM with certain general potentials (satisfying \Cref{a:Vasump} below), stationary Laguerre process, and stationary Jacobi process, all with parameter $\beta$, under the uniform in compact topology. We refer to \Cref{t:converge_Airy_details} for a more detailed statement.
\end{thm}
The definitions and background of these processes, as well as the precise statement and proof, will be given in \Cref{s:process}. 
We emphasize that these convergence results are new even for the classical cases of $\beta=1, 4$ (except for the DBM with $\beta=1$), which can be viewed as the joint convergence of eigenvalues of time-evolved classical ensembles with real or quaternion entries. 

We remark that the developed framework can also be applied to prove convergence to ALE$_\beta$ for the other mentioned models. The main remaining task is to establish the desired tightness given in (1). 
While such tightness are not available from \cite{bourgade2022optimal}, a plausible way is to utilize the dynamical loop equation, as in \cite{huang2024edge} where local laws down to any mesoscopic scale have been proven for random tilings. 
We leave this for future works.

\subsection{Other properties}
In addition to proving convergence to ALE$_\beta$, our new characterization can be leveraged to further investigate its properties. 
First, we can study the regularity of ALE$_\beta$. The Brownian regularity for the ALE has been intensively studied in \cite{corwin2014brownian,MR4403929, MR3987302,dauvergne2023wiener}. For ALE$_\beta$ with $\beta\geq 1$, it has been established in \cite{landon2020edge} that the lines of ALE$_\beta$ are locally Brownian.
In \Cref{s:holder} we show that the lines of ALE$_\beta$ are H{\"o}lder continuous with an exponent $1/2$ for any $\beta>0$.
The second property we study is the collision of lines. For $\beta\geq 1$, it has been established in \cite{landon2020edge} that the lines of ALE$_\beta$ do not collide. Conversely, for $\beta\in (0,1)$, collisions among lines are anticipated. We prove in \Cref{s:noncolliding} that the occurrence of collisions is almost surely of measure zero.

\subsection{Proof ideas}
We give an outline of our proofs, highlighting the main difficulties and ideas.

To prove the uniqueness in law as stated in \Cref{t:characterize}, 
the overall strategy is to establish a certain sense of `mixing in time' of the dynamics (given by \Cref{a:SDE}).
More precisely, we take two families of random particle-generated Nevanlinna functions, both satisfying the two assumptions.
Using the SDE \eqref{e:defYt} we reconstruct the dynamics of the poles, which are `infinite dimensional DBM' in a certain sense.
We couple the two `infinite dimensional DBM' obtained from both functions, by coupling the driving Brownian motions. Then we show that the poles get closer in time under this coupling.
Thus since both dynamics start from time $-\infty$, necessarily they are the same.

For both the reconstruction of DBM and the coupling, an essential input is that the poles have Airy-zero approximation, uniformly in time. 
This is implied by the uniform in time Airy-like property, as explained in \Cref{rem:1}.
Then under \Cref{a:Infinite}, it remains to show that such an approximation propagates in time, for which we again resort to the SDE \eqref{e:defYt}.

In summary, three tasks are inline: propagation of Airy-zero approximation, reconstruction of DBM, and coupling.
We next explain each of them in more details.

\subsubsection{Propagation of Airy-zero approximation}

Our \Cref{a:Infinite} concerns specific times, implying that the $i$-th pole remains constant away from the $i$-th zero of the Airy function. Utilizing the SDE \eqref{e:defYt}, we manage to get refined estimates: the $i$-th pole approximates the $i$-th zero with a polynomially small error over arbitrarily long time intervals with high probability, as demanded for later steps.  To achieve this, in \Cref{s:rigidity}, we analyze \eqref{e:defYt} along certain characteristics which offset the singularity of the nonlinear term. This idea has previously been used (see e.g., \cite{MR4009708, adhikari2020dyson,bourgade2021extreme}) to study DBM down to any mesoscopic scale, where the distance from the spectral parameter $w$ to the particles is much bigger than particle fluctuations.
However, in our analysis of \eqref{e:defYt}, we operate at a microscopic scale, where the distance from $w$ to the poles is of the same order as their fluctuation size. While a straightforward union bound over characteristic flows from polynomially many points suffices at the mesoscopic scale, our case demands careful selection of characteristic flows and precise estimation of error terms in the SDE, tailored to their initial positions.

\subsubsection{DBM reconstruction}
As already mentioned, there are significant challenges in analyzing DBM due to the singular repulsive interaction and possible particle collisions, in particular when $\beta\in (0,1)$. Even for finite dimensional DBM, establishing the existence and uniqueness of a strong solution requires the theory of multivalued SDE, see \cite{cepa1997diffusing,cepa2006equations}. Our approach through pole evolutions circumvents these issues entirely. Notably, there are no singularities even when poles collide.  

On the other hand, a key challenge of our method lies in reconstructing the dynamics of each pole, which requires ruling out the possibility of poles adhering to each other for prolonged periods. To address these, in \Cref{s:holder} we first establish that the trajectory of each pole is H{\"o}lder continuous solely utilizing \eqref{e:defYt}. Together with the Airy-zero approximation, for any short time interval, we can identify a large index $k$ such that the $k$-th and $(k+1)$-th poles remain bounded away from each other.  This enables us to localize the system and study the evolution of the first $k$ poles, treating the remaining poles' influence as an additional potential. 
 For this $k$ poles system, in \Cref{s:noncolliding},
by employing classical It\^{o} calculus on certain elementary symmetric polynomials, we show that the time of collisions almost surely has measure zero. We note that similar ideas have been employed to show  the instant
diffraction of the particles for DBM \cite{graczyk2014strong}.
Subsequently, the evolution of each pole can be reconstructed by performing a contour integral of \eqref{e:defYt}.
As a result, the $k$ poles system can be interpreted as a $k$-dimensional DBM with a time-dependent random drift, which exhibits a monotonicity property. 

\subsubsection{Uniqueness via coupling}
In \Cref{s:unique}, we take two solutions to \eqref{e:defYt}, and design a coupling where the poles get closer in time.
Consider the $k$-dimensional DBMs with random drifts reconstructed in the previous step, for these two solutions respectively.
Our coupling is by using the same set of driving Brownian motions for both.
Note that such $k$-dimensional DBMs with random drifts are constructed with random $k$, and only for a short time interval; but we need a coupling for a long time (to let the poles get closer).
A trick here is to concatenate these short intervals, and allow for different $k$ in each of them, as long as $k$ is always large enough.

There is a monotonicity property: if the $i$-th pole of the initial data for the first solution dominates that of the second solution for each $i$, then at any time after the $i$-th pole of the first solution dominates that of the second solution for each $i$. 
Then we can sandwich one solution between affine shifts of the other, while keeping the error arising from the affine shifts arbitrarily small. Such sandwiching forces the poles of the two solutions to get closer in time. By taking long enough time intervals, they must coincide, establishing the uniqueness as desired. Such coupling and sandwiching strategies have been used to establish local statistics universality in random lozenge tilings \cite{ULS,aggarwal2021edge,MR4771179}.

\subsection*{Notations}
In the rest of this paper, for any $a\le b\in\bR$, we let $\llbracket a,b\rrbracket=[a,b]\cap \bZ$. For two numbers $a,b$, we denote $a\wedge b=\min\{a,b\}$ and $a\vee b=\max\{a,b\}$.
For any $w\in \bC$, we use $\OO(w)$ to denote some $w'\in \bC$, satisfying $|w'|<C|w|$ for some universal constant $C>0$.
We also write $w'\lesssim w$ for $w'=\OO(w)$.

\subsection*{Acknowledgement}
 The research of J.H.\ is supported by  NSF grant DMS-2331096 and DMS-2337795, and the Sloan research award.
 The research of L.Z.\ is supported by NSF grant DMS-2505625, the Miller Institute for Basic Research in Science, and the Sloan research award.
 Part of this project was done when L.Z.\ was visiting University of Pennsylvania in the spring of 2023, and he thanks them for their hospitality.
The authors would like to thank Paul Bourgade, Vadim Gorin, Benjamin Landon, and B\'alint Vir\'ag for helpful discussions, and the referees for their careful reading of the manuscript and constructive suggestions.

\section{Preliminaries and decomposition}
In this section, we set up some preliminaries of our arguments.

We start with an explicit expression for any Airy-like Nevanlinna $Y$ from \Cref{defn:nvali}.
Recall the Airy function $\Ai$ with zeros $0>\fa_1>\fa_2>\fa_3>\cdots$.
\begin{prop}\label{p:Ytexp}
For any particle-generated Nevanlinna function $Y:\bH\to\bH\cup\bR$ with infinitely many poles $x_1\ge x_2\ge \cdots$, if i) $\sup_{i\in\bN}|x_i-\fa_i|<\infty$; and ii) there exists a sequence of complex numbers $w_n\to\infty$ along any direction in $(0, 3\pi/4)$, such that $Y(w_n)-\sqrt{w_n}\rightarrow 0$, then
\begin{align}\label{eq:Ytwes}
    Y(w)=\sum_{i=1}^\infty \left( \frac{1}{x_i-w}-\frac{1}{\fa_i}\right)-\frac{\Ai'(0)}{\Ai(0)}.
\end{align}
Moreover, \eqref{eq:Ytwes} holds if $Y$ is Airy-like.
\end{prop}
The Nevanlinna representation \eqref{e:Nevanlinna_representation}, for any particle-generated Nevanlinna function $Y$ with poles $P$ (a multiset), can be written as
\begin{align}\label{e:Ytz}
    Y(w)=b+cw+\sum_{x\in P} \left( \frac{1}{x-w}-\frac{x}{1+x^2}\right),
\end{align}
where $b,c\in \bR$, $c\geq 0$. We remark that it is possible that $P$ contains only finitely many numbers, and the summation in \eqref{e:Ytz} is finite.
Then to prove \Cref{p:Ytexp}, it remains to determine $b$ and $c$ in \eqref{e:Ytz} for $Y_t$, and establish that 
 the sum $\sum_{i=1}^\infty \frac{1}{\fa_i}-\frac{x_i}{1+x_i^2}$ converges.

\subsection{Estimates on Nevanlinna functions}
We now present some estimates on (particle-generated) Nevanlinna functions, which will be used in the Airy-like function part of \Cref{p:Ytexp}. We note that some of them are also used repeatedly in subsequent sections.

For any particle-generated Nevanlinna function $Y:\bH\to\bH\cup\bR$, from \eqref{e:Ytz} we have
\begin{align}\label{e:ImY}
    \Im[Y(w)]=c\Im[w]+\sum_{x\in P} \frac{\Im[w]}{|x-w|^2}.
\end{align}

\begin{lem}\label{l:Yproperty}
For any particle-generated Nevanlinna function $Y:\bH\to\bH\cup\bR$, the quantity $\Im[w]\Im[Y(w)]$ is monotone in $\Im[w]$; the derivatives of $Y$ satisfy
\begin{align}\label{e:stbb}
|\del^k_w Y(w)|\leq \frac{k!\Im[Y(w)]}{\Im[w]^k}\leq \frac{k!|Y(w)|}{\Im[w]^k},
\end{align}
where $\del^k_wY (w)$ is the $k$-th derivative of $Y$, for any integer $k\ge 2$.
\end{lem}
\begin{proof}
We first consider $k=1$.
Denote $w=E+\ri\eta$, then \eqref{e:ImY} gives
\begin{align}
\Im[w]\Im[Y(w)]=c\eta^2+\sum_{x\in P}\frac{\eta^2}{|E-x|^2+\eta^2},
\end{align}
which is increasing in $\eta\geq 0$. Using \eqref{e:Ytz}, the derivative of $Y(w)$ satisfies
\begin{align}
    |\del_w Y(w)|=\left|c+\sum_{x\in P}\frac{1}{(x-w)^2}\right|
    \leq c+\sum_{x\in P}\frac{1}{|x-w|^2}=\frac{\Im[Y(w)]}{\Im[w]}.
\end{align} 
And by
\[
|\del^k_w Y(w)| \le \sum_{x\in P} \frac{k!}{|x-w|^{k+1}} \le \frac{k!|\del_w Y(w)|}{\Im[w]^{k-1}},
\]
the $k\ge 2$ case follows.
\end{proof}

As already alluded to, if a particle-generated Nevanlinna function $Y$ is close to the function $\sqrt{w}$, its poles would be close to the Airy function zeros.
More precisely, we have the following estimate.
\begin{lem}  \label{lem:parti-clo}
Take any parameters $K>100$ and $0\le \delta<1$. Suppose that a particle-generated Nevanlinna function $Y$ satisfies the following conditions:
\begin{itemize}
    \item there is no pole of $Y$ in $(K,\infty)$;
    \item for any $w=x+\ri y$ with $x\le K^2$ and $y\ge 4K^2/(1+|x|^{\delta/2})$, we have
    \[
            \left|Y(w)-\sqrt{w}\right|\leq \frac{\Im[\sqrt w]^{1-\delta}}{\Im[w]}.
    \]
\end{itemize}
Then $Y$ has infinitely many poles $x_1\ge x_2\ge \cdots$, and $|x_i-\fa_i|<CK^4i^{-\delta/6}$ for any $i\in \bN$, where $C>0$ is a universal constant.
\end{lem}
In particular, these conditions are satisfied by Airy-like Nevanlinna functions (with $K$ large and $\delta=\fd$).

We also note that the domain of $w$ in \Cref{lem:parti-clo} is different from that in \Cref{defn:nvali}.
The domain in \Cref{defn:nvali} is taken to be easily verifiable for the sub-sequential limit of various models, as will be evident in \Cref{s:process}. 
In fact, bounds on $|Y(w)-\sqrt{w}|$ for $w$ closer to $\bR_+$ can be readily deduced for Airy-like Nevanlinna functions (see \Cref{lem:closeR}).

\begin{cor}\label{cor:Ytdiff}
    For any $(\fd,C_*)$-Airy-like Nevanlinna function $Y$ (with poles $x_1\ge x_2\ge \cdots$), there exists $B>0$ depending only on $C_*$, such that $|x_i-\fa_i|\leq B$ for each $i\in \bN$.
\end{cor}
The proof of \Cref{lem:parti-clo} relies on the Helffer-Sj{\"o}strand formula, which has become standard in random matrix theory. Therefore, we defer it to \Cref{s:HS}.

\subsection{Proof of \Cref{p:Ytexp}}
Thanks to \eqref{e:aiasymp} and \eqref{eq:weairy}, we have
\begin{align}\label{e:dequation}
\sum_{i=1}^\infty \left( \frac{1}{\fa_i-w} -\frac{1}{\fa_i} \right) -\frac{\Ai'(0)}{\Ai(0)}-\sqrt w\rightarrow 0,
\end{align}
when $w\to\infty$ along any direction in $(-3\pi/4, 3\pi/4)$.
By $|\fa_i|\sim (3\pi i/2)^{2/3}$ from \eqref{e:aklocation}, we have 
\begin{align}
    \label{e:location_estimate}
    |\fa_i+(3\pi i/2)^{2/3}|, |x_i+(3\pi i/2)^{2/3}|\leq B
\end{align} for a large $B>0$; in particular $x_1\leq B$. If we take $w$ with $\arg(w)\in (-3\pi/4, 3\pi/4)$ and $|w|>2B$,
\begin{align}\begin{split}\label{e:replace}
&\phantom{{}={}}\left|\sum_{i=1}^\infty\left( \frac{1}{x_i-w} -\frac{1}{\fa_i} \right) - \sum_{i=1}^\infty \left( \frac{1}{\fa_i-w}-\frac{1}{\fa_i}\right)\right| \\&
 \leq \sum_{i=1}^\infty\frac{|\fa_i-x_i|}{|x_i-w||\fa_i-w|}\leq \sum_{i=1}^\infty\frac{B}{|x_i-w||\fa_i-w|}\\
 &= \sum_{i> |w|^{3/2}} \frac{B}{|x_i-w||\fa_i-w|}
 +\sum_{i=1}^{\lfloor |w|^{3/2}\rfloor}\frac{B}{|x_i-w||\fa_i-w|}\\
 &\lesssim \sum_{i> |w|^{3/2}}\frac{B}{i^{4/3}}+\sum_{i=1}^{\lfloor |w|^{3/2}\rfloor}\frac{B}{|w|^2}
 \lesssim \frac{1}{|w|^{1/2}},
\end{split}\end{align}
where in the last line we used that, when $i>|w|^{3/2}$, $|x_i-w||\fa_i-w|\gtrsim i^{4/3}$ for the first term,
and $|x_i-w||\fa_i-w|\gtrsim |w|^2$ for the second term (since $|w|> 2B$).

Therefore, \eqref{e:dequation} and \eqref{e:replace} together give that
\begin{align}\label{e:infinity}
\sum_{i=1}^\infty \left( \frac{1}{x_i-w}-\frac{1}{\fa_i}\right)-\frac{\Ai'(0)}{\Ai(0)} -\sqrt{w}\rightarrow 0,
\end{align}
for $w\to\infty$ along any direction in $(-3\pi/4, 3\pi/4)$.
Also note that, thanks to \eqref{e:location_estimate}, we have
\begin{align}\label{e:replace2}
\left|\sum_{i=1}^\infty \frac{1}{\fa_i}-\frac{x_i}{1+x_i^2}\right|
\le \sum_{i=1}^\infty \left|\frac{1+x_i(x_i-\fa_i )}{\fa_i(1+x_i^2)}\right| \le \sum_{i=1}^\infty \frac{1+|\fa_i| B+B^2}{|\fa_i| (1+ ((\fa_i+B)\wedge 0)^2   )} < \infty.
\end{align}
By plugging \eqref{e:replace2} into the  representation \eqref{e:Ytz} for $Y$, we can rewrite $Y$ as (for some $b'\in\bR$)
\begin{align}\label{e:Ytz_new}
    Y(w)=b+cw+\sum_{i=1}^\infty \left( \frac{1}{x_i-w}-\frac{x_i}{1+x_i^2} \right)
    =b'+cw+\sum_{i=1}^\infty \left( \frac{1}{x_i-w}-\frac{1}{\fa_i}\right)-\frac{\Ai'(0)}{\Ai(0)}.
\end{align}

By our assumption, $Y(w_n)-\sqrt w_n\rightarrow 0$ as $n\to\infty$. Taking $w=w_n$ in \eqref{e:Ytz_new}, comparing with \eqref{e:infinity}, we conclude that $b'+cw_n\rightarrow 0$ when $n\to\infty$. It follows that $b'=c=0$, and \eqref{eq:Ytwes} holds.

Finally, if $Y$ is Airy-like, then \Cref{cor:Ytdiff} implies that $\sup_{i\in\bN}|x_i-\fa_i|<\infty$, and $|Y(n\ri)-\sqrt{n\ri}|\to 0$ as $n\to\infty$. These verify the assumptions in \Cref{p:Ytexp}, and \eqref{eq:Ytwes} holds. 
\qed

\subsection{Domain extension}
As mentioned below \Cref{lem:parti-clo}, for an Airy-like Nevanlinna function $Y$, we also provide a bound of $|Y(w)-\sqrt{w}|$ for $w$ close to $\bR_+$, which will be useful later.
\begin{lem}  \label{lem:closeR}
For $Y:\bH\to\bH\cup\bR$ being any $(\fd, C_*)$-Airy-like Nevanlinna function, we have
\[
|Y(w)-\sqrt{w}| \le B|w|^{-1/2}, \quad \forall \arg(w)\in(0, 3\pi/4), |w|>B.
\]
for $B>0$ depending only on $\fd$ and $C_*$.
\end{lem}
\begin{proof}
By \Cref{cor:Ytdiff}, there is $B'>0$ with each $|x_i-\fa_i|\le B'$.
Then, for any $w\in \bH$ with $\arg(w)\in(0, 3\pi/4), |w|>2B'$, we have $|\fa_i-w|\le 2|x_i-w|$.

By \Cref{p:Ytexp} and \eqref{eq:weairy}, we have 
\[
\left|Y(w)+\frac{\Ai'(w)}{\Ai(w)}\right|\le \sum_{i=1}^\infty \left| \frac{1}{\fa_i-w} - \frac{1}{x_i-w} \right| = \sum_{i=1}^\infty \left| \frac{x_i-\fa_i}{(x_i-w)(\fa_i-w)} \right| \le \sum_{i=1}^\infty \frac{2B'}{|\fa_i-w|^2}.
\]
Using \eqref{e:aklocation}, we have that
\[
\sum_{i=1}^\infty \frac{1}{|\fa_i-w|^2} = \sum_{i=1}^{\lfloor |w|^{3/2} \rfloor} \frac{1}{|\fa_i-w|^2} + \sum_{i=\lfloor |w|^{3/2} \rfloor + 1}^\infty \frac{1}{|\fa_i-w|^2}
\lesssim \frac{|w|^{3/2}}{|w|^2} + \sum_{i=\lfloor |w|^{3/2} \rfloor + 1}^\infty \frac{1}{i^{4/3}} \lesssim |w|^{-1/2}.
\]
Thus with \eqref{e:aiasymp}, the conclusion follows.
\end{proof}

\subsection{Topological statements}
As we shall derive convergence to ALE$_\beta$ from Stieltjes transforms, we will need several statements on the functional spaces, which we provide here.
\begin{definition}
For any locally finite measures $\mu_1, \mu_2, \cdots$ on $\bR$, we say that they \emph{converge in the vague topology} to another locally finite $\mu$, if $\mu_n(f)\to \mu(f)$, for any $f:\bR\to\bR$ that is compactly supported and smooth.  
\end{definition}
Such vague topology arises naturally from Nevanlinna function convergence.
\begin{lem} \label{lem:StoMconv}
Take Nevanlinna functions $Y_1, Y_2, \cdots$ and $Y$ such that $Y_n\to Y$ as $n\to \infty$, uniformly in any compact subset of $\bH$.
Suppose the corresponding measures (in their Nevanlinna representation) are $\mu_1, \mu_2, \cdots$ and $\mu$, respectively,
then $\mu_n\to \mu$ in the vague topology.
\end{lem}
\begin{proof}
Take any $f$ that is compactly supported and smooth, and let $K$ be a large enough number such that $f=0$ outside $[-K,K]$.
Take a smooth function $\chi:\bR_+\to\bR$, such that $\chi=1$ on $(0,1)$, and $\chi=0$ on $(2,\infty)$.
By \Cref{lem:HSf}, we have
\[
\mu_n(f) = \frac{1}{\pi} \int_{x+\ri y\in \bH}   -\Re[Y_n(x+\ri y)]yf'(x)\chi'(y) - \Im[Y_n(x+\ri y)](yf''(x)\chi(y)+f(x)\chi'(y))  \rd x \rd y.
\]
We note that $yf'(x)\chi'(y)=f(x)\chi'(y)=0$ whenever $(x,y)\not\in[-K,K] \times [1,2]$.
Moreover, $y f''(x) \chi(y) = 0$ whenever $(x, y) \notin [-K, K] \times (0, 2]$.
For $y \le 2$, by the monotonicity in \Cref{l:Yproperty}, we have $\Im[Y_n(x + \ri y)] \le \frac{2}{y} \Im[Y_n(x + 2\ri)]$.
Therefore, we have that the integrand in the above integral is non-zero only in $[-K,K]\times (0, 2]$; and it is bounded by a constant, which is independent of $n$ by the uniform convergence of $Y_n$ in $[-K,K]\times [1,2]$.
Thus we can apply dominated convergence theorem to deduce that the above integral converges to
\[
\mu(f) = \frac{1}{\pi} \int_{x+\ri y\in \bH}   -\Re[Y(x+\ri y)]yf'(x)\chi'(y) - \Im[Y(x+\ri y)](yf''(x)\chi(y)+f(x)\chi'(y))  \rd x \rd y.
\]
So the conclusion follows.
\end{proof}

On the other hand, in the setting of particle-generated measures, vague topology convergence can imply pole convergence.

\begin{lem}  \label{lem:measuretoparticle}
For a sequence of particle-generated measures $\mu_1, \mu_2, \ldots$, such that $\mu_k\to \mu$ as $k\to\infty$ in the vague topology, the limit $\mu$ must also be particle-generated.
Moreover, if there is some $K>0$ such that $\mu_k([K,\infty))=0$ for each $k$, then the followings are true.
We denote by $x^k_i$ the $i$-th largest pole of $\mu_k$ (with the convention that $x^k_i=-\infty$ if there are less than $i$ poles).
For each $i\in\bN$, either $x^k_i\to-\infty$ as $k\to\infty$, or $\lim_{k\to\infty} x^k_i$ exists and is a pole of $\mu$.
Also, all the poles of $\mu$ are given by such limits.
\end{lem}
\begin{proof}
By vague topology convergence, for any $a<b$, we have that $\limsup_{k\to\infty} \mu_k([a,b])\le \mu([a,b])$, and $\liminf_{k\to\infty} \mu_k((a,b))\ge \mu((a,b))$.
Then for any $a<b$ with $\mu([a, b])<1$, we must have $\mu_k([a,b])=0$ for $k$ large enough, since each $\mu_k([a,b])$ must be an integer.
Therefore, $\mu_k((a,b))=0$ for $k$ large enough, and $\mu((a,b))=0$.
These imply that $\mu$ in any open interval is either $\ge 1$ or zero.
Therefore, in any compact interval, $\mu$ is supported on finitely many points.

Now take any $x$ where $\mu(\{x\})>0$. Take any $a<x<b$ such that $\mu((a,b))=\mu([a,b])=\mu(\{x\})$.
Then for $k$ large enough,  $\mu_k((a,b))=\mu_k([a,b])=\mu(\{x\})$; so $\mu(\{x\})\in\bN$.
These imply that $\mu$ is particle-generated.

Under the additional assumption, there is $\mu((K,\infty))=0$. So we can write the poles of $\mu$ as $x_1\ge x_2\ge \cdots$ (with the convention that $x_i=-\infty$ if there are less than $i$ poles). 
Then we can show $x^k_i\to x_i$ via induction in $i\in\bN$, using that for each $a<b$ with $a, b\not\in\{x_i\}_{i=1}^\infty$,  $\mu_k((a, b))=\mu((a,b))$ for any $k$ large enough.
\end{proof}

\section{Pole evolution: uniform rigidity in time}
\label{s:rigidity}

In this section, we prove a uniform in time estimate for the poles. More precisely, the following proposition states that for the SDE \eqref{e:defYt},
 with high probability, its pole evolution gives a line ensemble (i.e., all the poles are bounded from above, and the trajectories are continuous), and the poles are close to the zeros of the Airy function, uniformly in time.
\begin{prop}\label{p:xrigidity}
For any $\fd, C_*>0$, there exist small $\delta,c>0$ and large $C>0$, such that the following holds.
Take any particle-generated $\{Y_t\}_{t\in\bR}$ satisfying \Cref{a:SDE}, and large $T>0$. Conditional on the event that $Y_{0}$ is $(\fd, C_*)$-Airy-like,
with probability at least $1-e^{-c(\log T)^2}$, we have
\begin{enumerate}
    \item The poles of $\{Y_t\}_{t\in[T, 2T]}$ give a line ensemble $\{x_i(t)\}_{i\in \bN, t\in[T,2T]}$, and  for each $t\in[T, 2T]$ and $w\in\bH$,
    \begin{align}\label{eq:Ytwes_copy}
        Y_t(w)=\sum_{i=1}^\infty\left(\frac{1}{x_i(t)-w}-\frac{1}{\fa_i}\right)-\frac{\Ai'(0)}{\Ai(0)}.
    \end{align}
    \item For each $t\in [T, 2T]$, $i\in\bN$, $y\le C$, we have that  
\begin{equation}\label{e:rigidity}
|x_i(t)-\fa_i|\leq \frac{C(\log T)^{40}}{i^{\delta}},
\end{equation} 
and
\begin{equation}\label{e:wegner}
|\{i\in\bN:x_i(t)\in [y-1, y+1]\}|\leq C\sqrt{|y|+1}.
\end{equation}

\end{enumerate}
\end{prop}
This immediately implies the first part of \Cref{t:characterize}.
\begin{cor}  \label{cor:31}
For any $\{Y_t\}_{t\in\bR}$ satisfying \Cref{a:Infinite} and \Cref{a:SDE}, its poles give a line ensemble $\{x_i(t)\}_{i\in\bN, t\in\bR}$, and $Y_t(w)=\sum_{i=1}^\infty\left(\frac{1}{x_i(t)-w}-\frac{1}{\fa_i}\right)-\frac{\Ai'(0)}{\Ai(0)}$ for any $t\in\bR$.
Moreover, almost surely $\sup_{t\in [-T, T], i\in \bN} |x_i(t)-\fa_i|<\infty$ for any $T>0$.
\end{cor}

We note that \eqref{e:wegner} (which is a  Wegner estimate) follows easily from \eqref{e:rigidity}, plus
\[
|\{i\in\bN:\fa_i(t)\in [y-1, y+1]\}|\leq C\sqrt{|y|+1},
\]
which directly follows from \eqref{e:aklocation}.

Our general strategy is to obtain uniform in time estimates for $Y_t-\sqrt w$, and to apply \Cref{lem:parti-clo}.
The main tasks are (1) to estimate bulk pole densities, via bounding $|Y_t(w)-\sqrt{w}|$ for $w$ in a reasonable domain contained in $\bH$ (in particular, allowing for polynomially close to the real axis, when $\Re[w]\to-\infty$); (2) to bound the first pole $x_1$. 
Both these are to be achieved through analyzing the SDE \eqref{e:defYt}.

\subsection{Characteristic flow}
We consider the characteristic flow,
\begin{align}   \label{eq:flowwt}
\del_t w_t=-\sqrt{w_t},\quad w_0\in \bH,
\end{align}
which can be solved as
\[
\del_t \sqrt{w_t}=-\frac{1}{2},\quad w_t=(\sqrt{w_0}-t/2)^2.
\]
This flow is motivated by the equation \eqref{e:defYt}.
More precisely, by ignoring the martingale and the second order derivatives, one considers the following PDE:
\[
\partial_t y = y\partial_w y - \frac{1}{2},
\]
for some smooth $y: \bR\times \bH\to \bH\cup\bR$. As we expect $y(t, w)$ to be roughly $\sqrt{w}$, we can take $r(t, w)=y(t, w)-\sqrt{w}$, and the above equation becomes
\[
\partial_t r = \sqrt{w} \partial_w r + r\partial_w r + \frac{r}{2\sqrt{w}}.\]
Then the flow \eqref{eq:flowwt} natually arises: if $r=0$ at some $w\in\bH$, we would have $r=0$ along the flow \eqref{eq:flowwt} starting from $w$.

By plugging the characteristic flow \eqref{eq:flowwt} into \eqref{e:SDE_form}, It{\^o}'s formula gives the following semi-martingale decomposition for $Y_t(w_t)-\sqrt{w_t}$ and $M_t(w_t)$
\begin{align}\begin{split}\label{e:dYtwt}
& \rd  (Y_t(w_t)-\sqrt{w_t})
= (\rd M_t)(w_t)+\left(\frac{2-\beta}{2\beta }(\del^2_w Y_t)(w_t)+(Y_t(w_t)-\sqrt{w_t})(\del_w Y_t)(w_t)\right)\rd t,\\
&\rd (M_t(w_t))=(\rd M_t)(w_t)-\sqrt{w_t}(\del_w M_t)(w_t)\rd t.
\end{split}\end{align}
The quadratic variations of the martingale $\int_0^{\cdot} (\rd M_s)(w_s)$ are the same as those of $M_t(w_t)$, which  are given by (recall from \eqref{eq:qvmwwp}):
\begin{equation}  \label{eq:qvmww3}
     \frac{\rd}{\rd t}\left\langle \int_0^{\cdot} (\rd M_s)(w_s)\right\rangle_t=\frac{1}{3\beta }(\del_w^3Y_t)(w_t).
\end{equation}

\medskip

In the rest of this section, we fix $\fd, C_*>0$, and take particle-generated $\{Y_t\}_{t\in\bR}$ satisfying \Cref{a:SDE}, and (unless otherwise noted) conditional on  $Y_0$ which is $(\fd, C_*)$-Airy-like.
All the constants below (including those in $\lesssim$ and $\OO(\cdot)$) can depend on $\fd$ and $C_*$.
We take $T$ to be a large number, and set $K=(\log T)^8$.

\subsection{Estimates for the bulk}

We next prove the following bound of $|Y_t(w)-\sqrt{w}|$ for $w$ in a domain contained in $\bH$.
It will be used to bound the bulk pole densities.
\begin{prop}\label{p:rigidity}
There exist small $\delta,c>0$ such that the following holds. Define the spectral domain 
    \begin{align}\label{e:defcD}
        \cD=\{\kappa+\ri\eta: \eta\geq K,  \eta^{-1-\delta}K\leq \kappa\leq \sqrt{K^2+\eta^2}\}.
    \end{align}
    With probability at least $1-e^{-c(\log T)^2}$, for any $t\in [T, 2T]$, and $\sqrt w\in \cD$, it holds
    \begin{align}\label{e:Ytboundha}
        \left|Y_t(w)-\sqrt{w}\right|\leq \frac{\Im[\sqrt w]^{1-\delta}}{\Im[w]}.
    \end{align}
\end{prop}

Note that $Y_t$ is Lipschitz away from the real axis by \Cref{l:Yproperty}. Therefore, $Y_t(w)-\sqrt{w}$ is also Lipschitz, and we only need to prove \eqref{e:Ytboundha} for a set of carefully chosen mesh points.
Namely, we consider the following mesh of points in the upper half plane: 
\begin{align}
\cL=\left\{\kappa+\ri\eta: \eta^3\in\bZ, \eta\geq K, \kappa=\frac{\bZ}{\eta^2}, K\leq \kappa\leq T+K+\eta\right\}.
\end{align}

\begin{lem}\label{l:approx}
    For any $\kappa'+\ri\eta'\in \cD$ as defined in \eqref{e:defcD} and any $t\in [T,2T]$, there exists $\kappa+\ri\eta\in \cL$ such that  $t\leq 2\kappa-2\eta^{-1-\delta}K$, and 
    \begin{align}
    |\kappa'-\kappa+t/2|, |\eta'-\eta|\leq \frac{1}{\eta^2}.
    \end{align}
\end{lem}
\begin{proof}
Suppose that $\eta'\in [i^{1/3}, (i+1)^{1/3}]$ for some integer $i\geq K^3$, and $\kappa'+t/2\in [j\eta^{-2}, (j+1)\eta^{-2}]$ for some $j\in \bZ$, we can take $\eta=i^{1/3}$ and $\kappa=(j+1)\eta^{-2}$.

Then for $\eta$ we have 
\begin{align}
    |\eta-\eta'|\leq (i+1)^{1/3}-i^{1/3}\leq i^{-2/3}/3=\eta^{-2}/3.
\end{align}
For $\kappa$ we have $|\kappa'-\kappa+t/2|\leq \eta^{-2}$, and $\kappa\ge \kappa'+t/2 \ge \kappa'+T/2\geq T/2 >K$. Moreover,  
\begin{align}
\kappa \leq \kappa'+t/2+\eta^{-2}\leq \kappa'+T+\eta^{-2}\leq \sqrt{K^2+(\eta')^2}+T+\eta^{-2}\leq T+K+\eta,
\end{align} 
where we used $\kappa'\leq \sqrt{K^2+(\eta')^2}$ because $\kappa'+\ri \eta'\in \cD$. Also $\kappa-t/2\ge \kappa'\geq (\eta')^{-1-\delta}K\ge \eta^{-1-\delta}K$, thus $t\leq 2\kappa-2\eta^{-1-\delta}K$.
\end{proof}

Now \Cref{p:rigidity} follows from the following estimate on one point.
\begin{lem}\label{l:Ywtest}
The following holds true for small enough $\delta,c>0$.
    Take any $\kappa_0+\eta\ri\in \cL$, and for each $t<2\kappa_0$ let 
    \begin{align}
    \sqrt{w_t}=(\kappa_0+\eta\ri)-t/2=:\kappa_t+\eta\ri,\quad \kappa_t=\kappa_0-\frac{t}{2}.
    \end{align}
    Conditional on $Y_0$ with $|Y_0(w_0)-\sqrt{w_0}|\leq {\Im[\sqrt{w_0}]^{1-\fd}}/{\Im[w_0]}$,  with probability at least $1-e^{-c\sqrt\eta}$,
    \begin{align}
        |Y_t(w_t)-\sqrt{w_t}|\leq \frac{1}{\kappa_t \eta^\delta},\quad \forall\;0\leq t\leq 2T\wedge (2\kappa_0-2\eta^{-1-\delta}K).
    \end{align}
\end{lem}

\begin{proof}[Proof of \Cref{p:rigidity}]
    By a union bound over all the points in $\cL$, \Cref{l:Ywtest} implies that
    \begin{align}  \label{eq:Ytdiff}
           |Y_t(w_t)-\sqrt{w_t}|\leq \frac{1}{\kappa_t \eta^\delta},\quad \forall\;0\leq t\leq 2T\wedge (2\kappa_0-2\eta^{-1-\delta}K), \; \sqrt{w_0}\in \cL,
    \end{align}
    with probability at least 
    \begin{align}  \label{eq:Ksqetad}
        1-\sum_{\eta\geq K, \eta^3\in\bZ} e^{-c\sqrt{\eta}} \eta^2(2T+2\eta)\geq 1-e^{-c(\log T)^2},
    \end{align}
    where we used that $K=(\log T)^8$ and $T$ is large enough.

For any $t\in [T,2T]$ and $\sqrt{w'}=\kappa'+\ri\eta'\in \cD$, thanks to \Cref{l:approx}, there exists $\sqrt{w_0}=\kappa+\ri\eta\in \cL$ such that $|\kappa'-\kappa_t|, |\eta'-\eta|\leq \eta^{-2}$. 
Then we have $\eta^{-1-\delta}K/2<\kappa_t<2\eta$.
It also follows that 
\begin{equation}  \label{eq:sqrtdif}
|\sqrt{w_t}-\sqrt{w'}|\leq \sqrt{|\kappa'-\kappa_t|^2+ |\eta'-\eta|^2}\leq \sqrt 2\eta^{-2},    
\end{equation}
and
\begin{align}
    |w_t-w'|\leq |\sqrt{w_t}+\sqrt{w'}||\sqrt{w_t}-\sqrt{w'}|\leq (2\sqrt{\kappa'^2+\eta^2}+2 \eta^{-2})\sqrt 2\eta^{-2}\lesssim \eta^{-1}.
\end{align}
In particular, this implies that $|w_t-w'|$ is much smaller than $2\kappa_t\eta=\Im[w_t]$ (which is at least $\eta^{-\delta}K$).
It then follows from \Cref{l:Yproperty} that
\begin{align}
    |Y_t(w_t)-Y_t(w')|\lesssim \frac{\Im\sqrt{w_t}}{\Im[w_t]}|w_t-w'|=\frac{|w_t-w'|}{2\kappa_t}\lesssim \frac{1}{\kappa_t\eta}.
\end{align}
Combining this with \eqref{eq:Ksqetad} and \eqref{eq:sqrtdif}, and using that $\eta^{-2}$ is much smaller than $\frac{1}{\kappa_t\eta^{\delta}}$ (which is at least $\eta^{-1-\delta}/2$), we conclude that
\[
|Y_t(w')-\sqrt{w'}| \lesssim \frac{1}{\kappa_t\eta^\delta} \lesssim \frac{1}{2\kappa'{\eta'}^\delta}=\frac{\Im[\sqrt w]^{1-\delta}}{\Im[w]}.
\]
Then the proof finishes by taking a slightly smaller $\delta$ (to remove the constant factor).
\end{proof}
We now prove the one point estimate, using \eqref{e:dYtwt}.
\begin{proof}[Proof of \Cref{l:Ywtest}]
We introduce the following stopping time,
\begin{align}
\sigma=\inf\left\{t\geq0: |Y_t(w_t)-\sqrt{w_t}|\ge \frac{1}{\kappa_t\eta^{\delta}}, \text{ or }\kappa_t\le \eta^{-1-\delta}K\right\}\wedge 2T.
\end{align}
We now bound the terms in the RHS of \eqref{e:dYtwt}. 
For $0\le t\leq \sigma$, we have $\Im[Y_t(w_t)]\leq 2\Im[\sqrt{w_t}]=2\eta$, since
\begin{align}
    |\Im[Y_t(w_t)]-\eta|\leq \frac{1}{\kappa_t \eta^\delta}\leq \frac{\eta}{K}.
\end{align}
Therefore (using \Cref{l:Yproperty}) we get
\begin{align}  \label{eq:twoder}
    |(\del_w^2 Y_t)(w_t)|\leq \frac{2\Im[Y_t(w_t)]}{\Im[w_t]^2}
    \leq \frac{4\Im[\sqrt{w_t}]}{\Im[w_t]^2}=\frac{1}{\kappa_t^2\eta}.
\end{align}

The quadratic variation of the martingale term is given by \eqref{eq:qvmww3}; then (again using \Cref{l:Yproperty})
\[
    \left|\frac{\rd}{\rd t}\left\langle \int_0^{\cdot} (\rd M_s) (w_s) \right\rangle _t\right|\leq \frac{2\Im[Y_t(w_t)]}{\beta\Im[w_t]^3}\leq \frac{4\Im[\sqrt{w_t}]}{\beta\Im[w_t]^3}=\frac{1}{2\beta\kappa_t^3\eta^2}.
\]
By integrating in time, we have
\[
\left\langle \int_0^{\cdot} (\rd M_s) (w_s) \right\rangle_t \leq 
\int_0^{t\wedge \sigma} \frac{\rd s}{2\beta\kappa_t^3\eta^2} \le \frac{1}{2\beta \kappa_t^2\eta^2}=\frac{2}{\beta}\cdot \frac{1}{(2\kappa_t\eta)^2}.
\]
Take $\theta=\sqrt{\eta}$. By the Burkholder-Davis-Gundy inequality, for any $0\leq t\leq 2T$ the following holds with probability at least\footnote{Here and in the rest of this proof, $C$ is used to denote a large constant, whose value may change from line to line.} $\ge 1-Ce^{-c\theta}$:
\[
\sup_{0\le s\leq t\wedge \sigma}\left|\int_0^{s}(\rd M_u)(w_u)\right|
\leq \frac{\theta}{2\kappa_{t} \eta},
\]
where we have absorbed the extra constant factor $2/\beta$ into $c$ in the exponent.
Then we can take a union bound over times
\[
    t=\left(2-\frac{1}{2^{i-1}}\right)\kappa_0, \quad \kappa_t=\frac{\kappa_0}{2^i},
\]
for $i\in\llbracket 1, 3\log_2(T)\rrbracket$, and get that with probability at least $1-C\log(T)e^{-c\theta}$, 
\begin{align}\label{eq:2term}
\left|\int_0^{t}(\rd M_s)(w_s)\right|
\leq \frac{\theta}{\kappa_{t} \eta},\quad \forall\; 0\leq t\leq \sigma.
\end{align}

Now by \eqref{e:dYtwt}, and using \eqref{eq:twoder} and \eqref{eq:2term}, we have that for any $0\le t\le \sigma$,
\begin{align}\label{e:globalest0}
\left|Y_t(w_t)-\sqrt{w_t})\right|
\leq\int_0^t\left|Y_{s}(w_s)-\sqrt{w_s}\right|\left|(\partial_w Y_s)(w_s)\right|\rd s+
\frac{2\theta}{\kappa_t\eta} +\left|Y_0(w_0) - \sqrt{w_0} \right|.
\end{align}
Here we used that $\int_0^t\frac{2}{\kappa_s^2\eta} \rd s<\frac{4}{\kappa_t\eta}$, which is much smaller than $\frac{\theta}{\kappa_t\eta}$.

For $0\le t\leq \sigma$, (using \Cref{l:Yproperty}) we have
\[
    |(\del_w Y_t)(w_t)|
    \leq \frac{\Im[Y_t(w_t)]}{\Im[w_t]}\leq \frac{1}{2\kappa_t}+\frac{|Y_t(w_t)-\sqrt{w_t}|}{2\kappa_t\eta}
    \leq \frac{1}{2\kappa_t}+\frac{1}{2\kappa_t^2\eta^{1+\delta}}=:\gamma(t) \le \frac{1}{\kappa_t}.
\]
Then for any $0\le s<t\le \sigma$, 
\begin{align}  \label{eq:intlogkst}
\int_s^t \gamma(u)\rd u\leq\log(\kappa_s/\kappa_t)
+\frac{1}{\kappa_t \eta^{1+\delta}}\leq \log(\kappa_s/\kappa_t)+\frac{1}{K}.
\end{align}
By Gr\"onwall's inequality and \eqref{e:globalest0}, for any $ 0\leq t\leq \sigma$, we can bound $|Y_{t}(w_{t})-\sqrt{w_{t}}|$ by:
\begin{align}
\frac{2\theta}{\kappa_{t} \eta} +\left|Y_0(w_0) - \sqrt{w_0} \right|
+\int_0^{t}\gamma(s)\left(\frac{2\theta}{\kappa_{s} \eta} +\left|Y_0(w_0) - \sqrt{w_0} \right|\right)\exp\left(\int_s^{t} \gamma(u)\rd u\right) \rd s.
\end{align}
We note that by \eqref{eq:intlogkst}, $\exp(\int_s^{t} \gamma(u)\rd u) \le \frac{2\kappa_s}{\kappa_t}$. Also
\[
\int_0^t \gamma(s) \frac{\theta}{\kappa_{s} \eta} \cdot \frac{\kappa_s}{\kappa_t} \rd s \le \frac{\theta}{\kappa_t \eta} \int_0^t \frac{1}{\kappa_s} \rd s = \frac{2\log (\kappa_0/\kappa_{t})\theta}{\kappa_{t} \eta},
\]
and
\[
\int_0^t \gamma(s) \left|Y_0(w_0) - \sqrt{w_0} \right|  \frac{\kappa_s}{\kappa_t} \rd s \le 
\left|Y_0(w_0) - \sqrt{w_0} \right|
\int_0^t \frac{1}{\kappa_t} \rd s = \frac{t\left|Y_0(w_0) - \sqrt{w_0} \right|}{\kappa_{t}}.
\]
Therefore we have
\[
\left|Y_{t}(w_{t})-\sqrt{w_{t}}\right|\le \frac{2(1+4\log (\kappa_0/\kappa_{t}))\theta}{\kappa_{t} \eta}+\left(1+\frac{2t}{\kappa_{t}}\right)\left|Y_0(w_0) - \sqrt{w_0} \right|\leq \frac{1}{2\kappa_{t} \eta^\delta},
\]
where in the last inequality we used $\log(\kappa_0/\kappa_{t})\leq \log(T\eta^2)$, which is much smaller than $\eta^{1/2-\delta}$ provided $K$ is large enough; and 
\begin{align}
    \left(1+\frac{2t}{\kappa_{t}}\right)\left|Y_0(w_0) - \sqrt{w_0} \right|
    \leq 
     \left(1+\frac{2t}{\kappa_{t}}\right)\frac{\Im[\sqrt{w_0}]^{1-\fd}}{\Im[w_0]}
     = 
      \left(1+\frac{2t}{\kappa_{t}}\right)\frac{1}{(2\kappa_{t}+t)\eta^{\fd}}
      \leq \frac{1}{3\kappa_{t}\eta^{\delta}}.
\end{align}
Therefore, since $Y_t(w_t)-\sqrt{w_t}$ is continuous in $t$, we conclude that $\sigma=2T\wedge (2\kappa-2\eta^{-1-\delta}K)$ with probability at least $1-Ce^{-c\sqrt \eta}$.
\end{proof}

\subsection{Estimates for the first pole}

Under the same setup as the previous subsection (in particular, $T$ is taken to be any large number, and $K=(\log T)^8$), 
we next prove the following proposition, which states that with high probability all the poles are bounded by $K$.

\begin{prop}\label{p:x1bound}
There exists a small number $c>0$ such that the following holds. With probability at least $1-e^{-c(\log T)^4}$, for any $t\in [0, 2T]$, $Y_t$ has no pole in $(K, \infty)$.
\end{prop}

We introduce the following stopping time, which is the first time the largest pole exceeds $K$,
\begin{align}\label{e:defsigma_0}
\sigma_0=\inf\left\{t\geq 0: Y_t \text{ has a pole in }(K, \infty) \right\}\wedge 2T.
\end{align}

We now denote $K_j=jK$ for $j\in \bN$ (in particular $K_1=K$),
and consider a mesh of points:
\[
\cL=\left\{\kappa+\ri\eta:\eta=(400 K_j)^{-1/4}, \kappa=\bZ\eta , \sqrt{K_j+\eta^2}\leq \kappa\leq \sqrt{K_{j+1}+\eta^2}+2T, j\in \bN\right\}.
\]
Then for any $\sqrt w=\kappa+\ri\eta\in \cL$, it holds $\kappa \eta^2\geq 1/4$, and $\Im[\sqrt{w}]=\eta\geq 1/{\kappa \eta^{1-\delta}}$, provided $K$ is large enough.

Now \Cref{p:x1bound} follows from the following estimate on one point. 
\begin{lem}\label{l:Ywtest_edge}
For any $B>0$, the following holds true for small enough $\delta,c>0$.
    Take any $j\geq 0$, $u=\kappa_0+\eta\ri\in \cL$ with $\eta=(400 K_j)^{-1/4}$, and let $\sqrt{w_t}=u-t/2=:\kappa_t+\eta\ri$ for each $t$. Conditional on $Y_0$ with $|Y_0(w_0)-\sqrt{w_0}|\leq B|w_0|^{-1/2}$,  with probability at least $ 1-e^{-c\sqrt{K_j}}$,
    \begin{align}
        |Y_t(w_t)-\sqrt{w_t}|\leq \frac{1}{\kappa_t \eta^{1-\delta}},\quad \forall\; 0\leq t\leq \sigma_0\wedge(2\kappa_0-2\sqrt{K_j+\eta^2}).
    \end{align}
\end{lem}
\begin{proof}[Proof of \Cref{p:x1bound}]
As $Y_0$ is $(\fd,C_*)$-Airy-like,
\Cref{lem:closeR} implies that (for some $B>0$)
\begin{align}\label{e:Y0bound}
    |Y_0(w)-\sqrt w|\leq B|w|^{-1/2} \text{ for all }\arg(w)\in (0, 3\pi/4), |w|>B.
\end{align}
    By a union bound over all the points in $\cL$, \Cref{l:Ywtest_edge} implies that conditional on \eqref{e:Y0bound},  with probability at least (for some $c'>0$, and taking $c<c'$ and $T$ large)
    \begin{align}
        1-\sum_{j=1}^\infty 3T(400K_j)^{1/4}e^{-c'\sqrt{K_j}}
        =1-\sum_{j=1}^\infty 3T(400j K)^{1/4}e^{-c'\sqrt{jK}}\geq 1-e^{-c\sqrt{K}},
    \end{align}
    it holds that for any $j\in\bN$ and $\sqrt{w_0}=\kappa_0+(400 K_j)^{-1/4}\ri\in \cL$,
    \begin{align}\label{e:Ytdiff_out}
           |Y_t(w_t)-\sqrt{w_t}|\leq \frac{1}{\kappa_t\eta^{1-\delta} },\quad \forall\; 0\leq t\leq \sigma_0\wedge(2\kappa_0-2\sqrt{K_j+\eta^2}).
    \end{align}

    Next we show that \eqref{e:Ytdiff_out} implies that $\sigma_0=2T$. Take any $t\in [0,2T]$, and $x\geq K$. Then $K_j\leq x<K_{j+1}$ for some $j\in \bN$. 
    Take $\eta=(400 K_j)^{-1/4}$. Then $\sqrt{x+\eta^2}+t/2\in [i\eta, (i+1)\eta]$ for some $i\in \bN$. We can take $\kappa_0=(i+1)\eta$ and $\sqrt{w_0}=\kappa_0+\ri\eta\in \cL$, so that $\kappa_t\in [\sqrt{x+\eta^2},\sqrt{x+\eta^2}+\eta ]$. Then we have
\begin{align*}
    \Re[w_t]&
    =\kappa_t^2-\eta^2\in[x, x+3\sqrt{x}\eta],\\
    \Im[w_t]&=2\kappa_t\eta\geq 2\sqrt{x}\eta.
\end{align*}
If there is a pole of $Y_t$ at $x$, (by Nevanlinna representation \eqref{e:Ytz}) necessarily,
\[
    \Im[Y_t(w_t)]\geq \frac{\Im[w_t]}{|x-w_t|^2}\geq \frac{\Im[w_t]}{(3\sqrt x \eta)^2+\Im[w_t]^2}
    =\frac{2\kappa_t}{9x\eta+4\kappa_t^2\eta^2}\geq \frac{1}{7\sqrt x\eta},
\]
However, \eqref{e:Ytdiff_out} (note that $\kappa_t\ge \sqrt{K_j+\eta^2}$, so $t\le 2\kappa_0-2\sqrt{K_j+\eta^2}$) implies
\[
    \Im[Y_t(w_t)]\leq \Im[\sqrt{w_t}]+\frac{1}{\kappa_t\eta^{1-\delta} }=\eta+\frac{1}{\kappa_t \eta^{1-\delta}}\leq \eta+\frac{1}{\sqrt{x}\eta^{1-\delta}}< \frac{1}{7\sqrt{x}\eta}
\]
This leads to a contradiction. 
Therefore (with probability at least $ 1-e^{-c\sqrt{K}}$) there is no pole in $(K,\infty)$ at any time in $[0, \sigma_0]$. Recall the definition of $\sigma_0$ from \eqref{e:defsigma_0}. Using that $Y_t$ is continuous in $t$, and \Cref{lem:StoMconv} and \Cref{lem:measuretoparticle}, 
we conclude that $\sigma_0=2T$, and \Cref{p:x1bound} follows.
\end{proof}

\begin{proof}[Proof of \Cref{l:Ywtest_edge}]
We introduce the following stopping time
\begin{align}
\sigma=\inf\left\{t\leq \sigma_0: |Y_t(w_t)-\sqrt{w_t}|\ge \frac{1}{\kappa_t\eta^{1-\delta}},\text{ or } \kappa_t^2-\eta^2\le  K_j,\text{ or } t=\sigma_0\right\}.
\end{align}
For $t\leq \sigma$, it is necessary that $t\leq \sigma_0$, and $Y_t$ has no pole in $(K, \infty)$. Moreover, we also have that $\kappa_t\geq \sqrt{K_j+\eta^2}\geq \sqrt{K_j}$, and $\kappa_t\eta^{2-\delta}\geq1$, provided that $K$ is large enough.

We now consider the terms in the RHS of \eqref{e:dYtwt}. For $0\le t\leq \sigma$, using Nevanlinna representation \eqref{e:Ytz} we have
\[
|(\del_w^2 Y_t)(w_t)|
    \leq\sum_{x\in P}\frac{1}{|x - w_{t}|^3} 
 \leq \frac{1}{|w_t-K|}\sum_{x\in P}\frac{1}{|x - w_t|^2} \le \frac{ \Im[Y_t(w_t)]}{|w_t-K| \Im[w_t]},
\]
where the second inequality follows from that $\Re[w_t]=\kappa_t^2-\eta^2\ge K_j\geq K \ge x$ for any $x\in P$.
As $|\Im[Y_t(w_t)]-\Im[\sqrt{w_t}]| \leq 1/(\kappa_t\eta^{1-\delta})\leq\eta$, we have $\Im[Y_t(w_t)]\le 2\eta$, so
\begin{align*}
    |(\del_w^2 Y_t)(w_t)|
 \leq  \frac{2\eta}{|w_t-K|\Im[w_{t}]} 
 \leq \frac{\sqrt{2}}{(\kappa_t^2-\eta^2-K+2\kappa_t\eta) \kappa_t},
\end{align*}
where we used $\Im[w_t]=2\kappa_t\eta$, and $|w_t-K|=|\kappa_t^2-\eta^2-K+2\ri\kappa_t\eta|$ for the second inequality. 
By integrating in time, we have
\begin{align}\begin{split}  \label{eq:lasder}
 &\int_0^t |(\del_w^2 Y_s)(w_s)|\rd s
\leq \int_0^t \frac{\sqrt{2}}{(\kappa_t\kappa_s-\eta^2-K+2\kappa_t\eta) \kappa_t}\rd s\\
\le &
 \frac{2\sqrt{2}  (\log(\kappa_t\kappa_0-\eta^2-K+2\kappa_t\eta)- \log (\kappa_t^2-\eta^2-K+2\kappa_t\eta))  }{\kappa_t^2}
\le \frac{2\sqrt{2}  \log(1+\kappa_0/(2\eta))}{\kappa_t^2}.
\end{split}\end{align}
The quadratic variation of the martingale term is given by \eqref{eq:qvmww3}.
By using Nevanlinna representation \eqref{e:Ytz} for $Y_t$, we have 
\[
\left|\frac{\rd}{\rd t}\left\langle \int_0^{\cdot} (\rd M_s) (w_s) \right\rangle_t\right|
\leq  \sum_{x\in P}\frac{2}{\beta |x-w_t|^4}.
\]
Then using $\Re(w_t)\ge K_j\geq K \ge x$ for any $x\in P$, we can bound the above by
\[ \frac{2}{\beta |w_t-K|^2} \sum_{x\in P} \frac{1}{|x-w_t|^2} \le \frac{2\Im[Y_t(w_t)]}{\beta |w_t-K|^2 \Im[w_t]} \le \frac{4}{\beta(\kappa_t^2-\eta^2-K+2\kappa_t\eta)^2 \kappa_t},
\]
where we used $\Im[Y_t(w_t)]\le 2\eta$, $\Im[w_t]=2\kappa_t\eta$, and $|w_t-K|= |\kappa_t^2-\eta^2-K+2\ri\kappa_t\eta|$ for the last inequality.
By integrating in time, we have
\begin{multline*}
\int_0^t \frac{4\rd s}{\beta(\kappa_s^2-\eta^2-K+2\kappa_s\eta)^2 \kappa_s} \le \int_0^t \frac{4\rd s}{\beta(\kappa_t\kappa_s-\eta^2-K+2\kappa_t\eta)^2 \kappa_t} \\
\le \frac{8}{\beta(\kappa_t^2-\eta^2-K+2\kappa_t\eta) \kappa_t^2} \le \frac{4}{\beta \kappa_t^3\eta}.
\end{multline*}
Similarly to \eqref{eq:2term}, taking $\theta=K_j^{1/4}$, a union bound implies that with probability at least $1-\log(T)e^{-c\theta}$,
\begin{equation}  \label{eq:l2term}
\left|\int_0^t(\rd M_s)(w_s)\right|
\leq \frac{\theta}{\sqrt{\kappa_t^3 \eta}},\quad \forall\; 0\leq t\leq \sigma.
\end{equation}

Now by \eqref{e:dYtwt}, and using \eqref{eq:lasder} and \eqref{eq:l2term}, we have that for any $0\le t\le \sigma$,
\begin{align}\label{e:globalest}
\left|Y_t(w_t)-\sqrt{w_t})\right|
\leq\int_0^t\left|Y_{s}(w_s)-\sqrt{w_s}\right|\left|(\partial_w Y_s)(w_s)\right|\rd s+
\frac{2\theta}{\sqrt{\kappa_t^3\eta}} +\left|Y_0(w_0) - \sqrt{w_0} \right|.
\end{align}
Here we used that $\frac{2\sqrt{2}  \log(1+\kappa_0/(2\eta))}{\kappa_t^2} <\frac{\theta}{\sqrt{\kappa_t^3\eta}}$, provided $K$ is large enough. 

As for $(\del_w Y_t)(w_t)$, for $0\le t\le \sigma$ (using \Cref{l:Yproperty}) we have
\begin{align*}
    |(\del_w Y_t)(w_t)|
    \leq \frac{\Im[Y_t(w_t)]}{\Im[w_t]}\leq \frac{1}{2\kappa_t}+\frac{|Y_t(w_t)-\sqrt{w_t}|}{2\kappa_t\eta}
    \leq \frac{1}{2\kappa_t}+\frac{1}{2\kappa_t^2\eta^{2-\delta}}=:\gamma(t) \le \frac{1}{\kappa_t}.
\end{align*}
Then for any $0\le s<t\le \sigma$ (noticing that $\kappa_t \eta^2\geq 1/20$), 
\begin{align}  \label{eq:laintlogkst}
\int_s^t \gamma(u)\rd u\leq\log(\kappa_s/\kappa_t)
+\frac{1}{\kappa_t \eta^{2-\delta}}\leq \log(\kappa_s/\kappa_t)+20\eta^{\delta}.
\end{align}
By Gr\"onwall's inequality and \eqref{e:globalest}, for any $0\leq t\leq \sigma$, we can bound $|Y_{t}(w_{t})-\sqrt{w_{t}}|$ by:
\[
\frac{2\theta}{\sqrt{\kappa_{t}^3 \eta}} +\left|Y_0(w_0) - \sqrt{w_0} \right|
+\int_0^{t}\gamma(s)\left(\frac{2\theta}{\sqrt{\kappa_{s}^3 \eta}} +\left|Y_0(w_0) - \sqrt{w_0} \right|\right)\exp\left(\int_s^{t} \gamma(u)\rd u\right) \rd s.
\]
By \eqref{eq:laintlogkst}, $\exp(\int_s^{t} \gamma(u)\rd u) \le \frac{2\kappa_s}{\kappa_t}$. Also
\[
\int_0^t \gamma(s) \frac{\theta}{\sqrt{\kappa_{s}^3 \eta}} \cdot \frac{\kappa_s}{\kappa_t} \rd s \le \frac{\theta}{\kappa_t} \int_0^t \frac{1}{\sqrt{\kappa_s^3\eta}} \rd s \le \frac{\theta}{\sqrt{\kappa_{t}^3\eta}},
\]
and
\[
\int_0^t \gamma(s) \left|Y_0(w_0) - \sqrt{w_0} \right|  \frac{\kappa_s}{\kappa_t} \rd s \le 
\left|Y_0(w_0) - \sqrt{w_0} \right|
\int_0^t \frac{1}{\kappa_t} \rd s = \frac{t\left|Y_0(w_0) - \sqrt{w_0} \right|}{\kappa_{t}}.
\]
Therefore for $\theta=K_j^{1/4}$, it holds
\[
\left|Y_{t}(w_{t})-\sqrt{w_{t}}\right|\le \frac{6\theta}{\sqrt{\kappa_{t}^3 \eta}}+\left(1+\frac{2t}{\kappa_{t}}\right)\left|Y_0(w_0) - \sqrt{w_0} \right|\leq \frac{1}{2\kappa_{t} \eta^{1-\delta}},
\]
where in the last inequality we used that 
\[
    \left(1+\frac{2t}{\kappa_{t}}\right)\left|Y_0(w_0) - \sqrt{w_0} \right|
    \le 
    B \left(1+\frac{2t}{\kappa_{t}}\right)|w_0|^{-1/2}
     \le
     B \left(1+\frac{2t}{\kappa_{t}}\right)\frac{1}{\kappa_{t}+t/2}
      \leq \frac{1}{3\kappa_{t}\eta^{1-\delta}}.
\]
Therefore, since $Y_t(w_t)-\sqrt{w_t}$ is continuous in $t$, we conclude that $\sigma=\sigma_0\wedge(2\kappa_0-2\sqrt{K_j+\eta^2})$ with the desired probability.
\end{proof}

\subsection{Proof of \Cref{p:xrigidity}}
As before, take $T$ large enough and $K=(\log T)^8$. \Cref{p:x1bound} and \Cref{p:rigidity} verify the two assumptions in \Cref{lem:parti-clo}, respectively. Thus with probability at least $1-e^{-c(\log T)^2}$, for any $t\in [T,2T]$, $Y_t$ has infinitely many poles $x_1(t)\ge x_2(t)\ge \cdots$, satisfying
\begin{align}\label{e:eiglocation}
    |x_i(t)-\fa_i|<CK^4 i^{-\delta/6}.
\end{align}
This gives the second statement in \Cref{p:xrigidity}, after replacing $\delta/6$ by $\delta$. 

The statement \eqref{e:eiglocation} also verifies the first assumption in \Cref{p:Ytexp}. Moreover, in \eqref{e:Ytboundha} we can take a sequence of complex numbers $w_n=n\ri$, so that $|Y_t(w_n)-\sqrt{w_n}|\to 0$ as $n\to\infty$. This verifies the second assumption in \Cref{p:Ytexp}, from which \eqref{eq:Ytwes_copy} holds. 

Finally, for each $i\in\bN$, the continuity of $x_i(t)$ in $t\in [T, 2T]$ follows from the continuity of $Y_t$ in $t$, and \Cref{lem:StoMconv}, \Cref{lem:measuretoparticle}.
\qed

\section{H\"older Regularity}
\label{s:holder} 
In this section, we upgrade the trajectory continuity into H{\"o}lder regularity.

\begin{prop}\label{p:holder}
For any $B>0$, there exist large and small $C,c>0$ such that the following holds.
Take any particle-generated $\{Y_t\}_{t\in\bR}$ satisfying \Cref{a:SDE}, and that its poles are given by a line ensemble $\{x_i(t)\}_{i\in\bN, t\in\bR}$, and that $Y_0(w)=\sum_{i=1}^\infty\frac{1}{x_i(0)-w}-\frac{1}{\fa_i}-\frac{\Ai'(0)}{\Ai(0)}$ for any $w\in\bH$.
Take $k\in\bN$ large enough and any $0<\xi<ck^{-4/3}$.
Then conditional on the event that
\begin{equation}  \label{eq:xibdfh}
|x_i(0)-\fa_i| \le B, \quad \forall i \in \llbracket k/2, 2k\rrbracket,
\end{equation}
with probability at least $1-e^{-ck^{1/6}}$ we have
    \[ \max_{0\leq s\leq \xi} |x_k(s)-x_k(0)|\leq C k^{2/3}\xi^{1/2}.\]
\end{prop}

\begin{proof}
We first prove the (with conditional probability $>1-e^{-ck^{1/6}}$) upper bound
    \begin{align}\label{e:holder2}
        \max_{0\leq s\leq \xi} x_k(s)\leq x_k(0)+ C k^{2/3}\xi^{1/2}. 
    \end{align}
The lower bound $\min_{0\leq s\leq \xi}x_k(s)\geq x_k(0)-C k^{2/3}\xi^{1/2}$ can be proven in the same way. 

By \eqref{eq:xibdfh} with \eqref{e:aklocation}, we conclude that there exists a large constant $C_1$ such that 
\begin{align}\label{e:wegnerest}
    |\{i: x_i(0)\in [x_k(0)-1, x_k(0)+1]\}|\leq C_1k^{1/3}.
\end{align}
Take small $\delta\leq 1/(C_1k^{1/3})$. Then \eqref{e:wegnerest} implies that there exists some $1\leq \ell \leq C_1k^{1/3}$, such that $x_{k-\ell-1}(0)- x_{k-\ell}(0)\geq \delta$. We take the smallest such $\ell$, then $|x_{k-\ell}(0)-x_k(0)|\leq \delta \ell$. 

Let $E=x_{k-\ell}(0)+\delta/2$, and take a small $b>0$ and $w=E+\ri b$. Next we show that, provided $\delta\geq 8bC_1^{1/2}k^{1/6}$,
\begin{align}\label{e:ImYt}
   \Im[Y_0(w)]= \Im[Y_0(E+\ri b)]=\sum_{i=1}^\infty\frac{b}{|E+\ri b-x_i(0)|^2}\leq \frac{1}{4b}.
\end{align}
For this, from \eqref{eq:xibdfh} and \eqref{e:aklocation}, we have  $|\{i: x_i(0)\in [E-\delta/2-m, E-\delta/2-(m-1)]\}|\leq C_1m^{1/2}k^{1/3}$ for each $m\in\bN$.
It follows that
\begin{align*}
    \sum_{i: x_i(0)\leq E-\delta/2}\frac{b}{|E+\ri b-x_i(0)|^2}
    &\leq 
      \sum_{m=1}^\infty\sum_{i: x_i(0)\in [E-\delta/2-m, E-\delta/2-(m-1)]}\frac{b}{|E+\ri b-x_i(0)|^2}\\
      &\leq \sum_{m=1}^\infty\frac{C_1 bm^{1/2}k^{1/3}}{(\delta/2+m-1)^2}\leq \frac{8C_1bk^{1/3}}{\delta^2}\leq \frac{1}{8b},
\end{align*}
using that $\delta\geq 8bC_1^{1/2}k^{1/6}$.
By a similar argument, we can also upper bound the summation over $x_i(0)\geq E+\delta/2$, and \eqref{e:ImYt} follows.

We now introduce a stopping time $\tau$:
\[
    \tau=\inf\left\{s\geq 0: \Im[Y_s(w)]\geq \frac{1}{2b}\right\}\wedge \xi.
\]
For $0\le s\leq \tau$, it follows from \Cref{l:Yproperty}
\[
|\partial_w^2 Y_s(w)|\leq \frac{2\Im[Y_s(w)]}{b^2}\leq \frac{1}{b^3}, \quad |\del_w (Y_s(w)^2)|=2|\partial_wY_s(w)|\cdot|Y_s(w)|\leq \frac{|Y_s(w)|}{b^2}.
\]
Thus by \Cref{a:SDE},
\begin{align}\label{e:Ytdiff}
|Y_s(w)-Y_0(w)|\le \left|M_s(w)-M_0(w) \right|+\frac{\int_0^s|Y_u(w)|\rd u}{2b^2}+\OO\left(\frac{s}{b^3}\right).
\end{align}
For the martingale term, 
using \eqref{eq:qvmwwp} and \Cref{l:Yproperty}, it follows that (for $0\le s\leq \tau$) 
\[
    \left| \frac{\rd}{\rd s}\langle M (w)\rangle_s \right|\leq \frac{2\Im[Y_s(w)]}{\beta b^3}\leq \frac{1}{\beta b^4}.
\]
Therefore, by the Burkholder-Davis-Gundy inequality, there exists a small constant $c_1>0$,
\begin{align}\label{e:martingale}
\bP\left(\sup_{0\le s\leq \tau}\left| M_s(w)-M_0(w)\right|\geq \frac{k^{1/6}\xi^{1/2}}{b^2}\right)\leq e^{-c_1k^{1/6}}
\end{align}
By plugging \eqref{e:martingale} into \eqref{e:Ytdiff}, it follows that with probability at least $1-e^{-c_1k^{1/6}}$, 
\begin{align}\label{e:Ydiff}
\left|Y_s(w)-Y_0(w)\right|=\OO\left(\frac{k^{1/6} \xi^{1/2}}{b^2}+\frac{\xi}{b^3}\right),\quad \forall s\in [0, \tau],
\end{align}
provided that $\xi\leq b^2$.
Now we take a large $C_2>0$, and assume that $\xi\le b^2/(C_2 k^{1/3})$. Then for any $s\in [0, \tau]$, and using \eqref{e:ImYt}, we have
\[
\Im[Y_s(w)]\le |Y_s(w)-Y_0(w)| + \Im[Y_0(w)]<\frac{1}{2b},
\]
which, in particular, implies that $\tau=\xi$. Since $\Im[Y_s(w)]=\sum_{i=1}^\infty\frac{b}{|w-x_i(s)|^2}$, this further implies that $\{x_i(s)\}_{i\in \bN} \cap  [E-b, E+b]=\emptyset$, for any $s\in [0, \xi]$.

In summary, we conclude that with probability at least $1-e^{-c_1k^{1/6}}$, $\{x_i(s)\}_{i\in \bN} \cap  [E-b, E+b]=\emptyset$, for any $s\in [0, \xi]$.
Since $x_k(0)<E$, it follows that (with probability at least $1-e^{-c_1k^{1/6}}$) for any $s\in [0, \xi]$,
\[
    x_k(s)\leq E-b=x_{k-\ell}(0)+\delta/2-b\le x_k(0)+ \delta\ell + \delta/2 \le x_k(0)+2C_1k^{1/3}\delta.
\]

Finally, we choose the parameter $\delta$ and $b$, satisfying all the above constraints:
\[
\delta \le 1/(C_1k^{1/3}), \quad \delta \ge 8bC_1^{1/2}k^{1/6}, \quad \xi \le b^2/(C_2k^{1/3}).
\]
Then we can take $b=C_2^{1/2}k^{1/6}\xi^{1/2}$ and $\delta = 8C_1^{1/2}C_2^{1/2}k^{1/3}\xi^{1/2}$. 
By taking $C$ large enough and $c$ small enough (depending on $C_1$, $C_2$, and $c_1$), the constraint $\delta \le 1/(C_1k^{1/3})$ is also satisfied since $\xi<ck^{-4/3}$; and \eqref{e:holder2} follows. 
\end{proof}

\section{Recover Dyson Brownian Motion}
\label{s:noncolliding}
In this section, we localize any line ensemble given by the poles of the SDE \eqref{e:defYt}, by deriving another SDE satisfied by the first finitely many poles (in the sense of weak solution).

More precisely, for the line ensemble $\{x_i(t)\}_{i\in\bN, t\in\bR}$ of poles,
we prove that, if for some large $k\in\bN$, $x_k(t)$ and $x_{k+1}(t)$ are bounded away from each other for certain amount of time, the evolution of $x_1(t)\geq x_2(t)\geq \cdots \ge x_k(t)$ is then described by DBM plus a drift term, describing the effect of $\{x_i(t)\}_{i=k+1}^\infty$.

\begin{prop}\label{p:noncolliding}
For any $C>0$ the following is true.
Take any particle-generated $\{Y_t\}_{t\in\bR}$ satisfying \Cref{a:SDE}, and that its poles are given by a line ensemble $\{x_i(t)\}_{i\in\bN, t\in\bR}$.
Fix a large $k\in \bN$, and denote the stopping time 
\[
    \tau=\inf\{t\geq 0: x_k(t)-x_{k+1}(t)\leq 1/(Ck^{1/3})\}\wedge 1.
\]
We consider the line ensemble $\{x_i(t)\}_{i\in\bN, t\in\bR}$ conditional on $x_k(0)-x_{k+1}(0)>1/(Ck^{1/3})$ (i.e., $\tau>0$).
For this conditional process,
there exist independent Brownian motions $B_1, \cdots, B_k$ adapted to the filtration $\cF_t=\sigma(\{\bmx(u)\}_{u\leq t})$, satisfying (again in the sense of the differential form of the semimartingale decomposition)
        \begin{equation}\label{e:dbmsde}
\rd x_i(t)=\sqrt{\frac{2}{\beta}}\rd B_i(t) +\sum_{1\leq j\leq k\atop j\neq i}
\frac{\rd t}{x_i(t)-x_j(t)} + W_t(x_i(t))\rd t,\quad \forall i\in\llbracket 1, k\rrbracket, \; t\in [0, \tau],
\end{equation}
where $W_t$ is a random meromorphic function, defined as
\begin{align}\label{e:defWt2}
W_t(w):=-Y_t(w)+\sum_{i=1}^k \frac{1}{x_i(t)-w}.
\end{align}
Moreover, almost surely the following holds: 
\begin{align}\label{e:nocollision}
    \int_0^\tau\mathds{1}(\exists 1\leq i<j\leq k: x_i(t)=x_j(t)) \rd t=0.
\end{align}
\end{prop}

\begin{remark}
    We remark that for $t\leq \tau$, $x_k(t)-x_{k+1}(t)\leq 1/Ck^{1/3}$. In this regime the simple poles of the sum in \eqref{e:defWt2} at $w=x_i(t)$, $1\le i\le k$, cancel with the corresponding singular parts of $Y_t(w)$; hence these singularities are removable. Consequently, $W_t$ is analytic at $w=x_i(t)$ for $1\le i\le k$, and the values $W_t\big(x_i(t)\big)$ are well defined (as limits $w\to x_i(t)$).
\end{remark}
The rest of this section is devoted to proving \Cref{p:noncolliding}.
Using that $\{x_i(t)\}_{i\in\bN}$ are poles of $Y_t$, we derive \eqref{e:dbmsde} from \Cref{a:SDE}, using a contour integral.
For this, we need first establish that the poles do not collide at almost every time (i.e., \eqref{e:nocollision}).
The idea to establish the collision time estimate is to consider the process $(x_i(t)-x_j(t))^2$ for some $i<j$, showing that its level-$0$ local time equals $0$. We mainly follow the standard argument used to study the Bessel process, see \cite[Chapter XI, Section 1]{revuz2013continuous}.
To analyze such processes we again resort to contour integrals, and therefore an induction will be used.

We next give the semi-martingale decomposition of a process, which is the sum of $(x_i(t)-x_j(t))^2$ for $i, j$ in an interval.

For any particle-generated $\{Y_t\}_{t\in\bR}$ satisfying \Cref{a:SDE}, with poles given by the line ensemble $\{\bmx(t)\}_{t\in\bR}$, and any $a, \al\in \bN$ with $\al\ge 2$, denote
\begin{equation}  \label{e:defwt}
    W^{a,\al}_t(w)=-Y_t(w)+\sum_{i=a}^{a+\al-1}\frac{1}{x_i(t)-w},
\end{equation}
and
\[
Z^{a,\al}(t)=\sum_{a\le i<j\le a+\al-1}(x_i(t)-x_j(t))^2 = 
\al \sum_{i=1}^{a+\al-1} x_i(t)^2 - \left(\sum_{i=1}^{a+\al-1} x_i(t)\right)^2.
\]
\begin{lem}  \label{lem:SDEVM}
In the above setup, take any $t_0\in \bR$, and denote the stopping time
\[
\sigma=\inf\{t\ge t_0: x_{a-1}(t)=x_a(t)\text{ or } x_{a+\al-1}(t)=x_{a+\al}(t)\text{ or } t=t_0+1\},
\]
where we use the convention that $x_0(t)=\infty$, when $a=1$.
Then $Z^{a,\al}$ in $[t_0, \sigma)$ satisfies $
\rd Z^{a,\al}(t)= \rd M^{a,\al}(t) + V^{a,\al}(t) \rd t$,
where 
\[
    V^{a,\al}(t)=2\sum_{1\leq i<j\leq \al} (x_i(t)-x_j(t))(W^{a,\al}_t(x_i)-W^{a,\al}_t(x_j))+\alpha^2(\alpha-2)+\frac{2\alpha(\al-1)}{\beta},
\]
and $\rd M^{a,\al}(t)$ is the martingale term, with quadratic variation
\begin{equation}\label{e:dMtt}
\frac{\rd}{\rd t}\langle M^{a,\al}\rangle_t = \frac{8\alpha}{\beta}Z^{a,\al}(t).
\end{equation}
\end{lem}
\begin{proof}
For simplicity of notations, we fix $a, \al$, and write $Z(t)=Z(t)^{a,\al}$, $V(t)=V(t)^{a,\al}$, an $\rd M(t)=\rd M(t)^{a,\al}$ within this proof.

For $t$ with $x_{a-1}(t)>x_a(t)$ and  $x_{a+\al-1}(t)>x_{a+\al}(t)$, we take a contour $\cC=\cC_t$ enclosing $x_a(t), \cdots, x_{a+\al-1}(t)$, but not any $x_i(t)$ for $i<a$ or $i\ge a+\al$.
Since each $x_i(t)$ is continuous in $t$, we have that $\cC$ encloses $x_a(t'), \cdots, x_{a+\al-1}(t')$, but not any $x_i(t')$ for $i<a$ or $i\ge a+\al$, for each $t'$ in a small open neighborhood of $t$.
Then by \Cref{a:SDE}, we have
\begin{align}\begin{split}\label{e:dx}
    \rd \sum_{i=a}^{a+\al-1} x_i(t) 
    &=-\frac{1}{2\pi \ri}\oint_\cC  w \rd Y_t(w)\rd w \\
    &=-\frac{1}{2\pi\ri}\oint_\cC w\rd w\left(\rd M_t(w)+\left(\frac{2-\beta}{2\beta }\del^2_w Y_t(w)+\frac{1}{2}\del_w (Y_t(w)^2)-\frac{1}{2}\right)\rd t\right).
\end{split}\end{align}
We remark that for this expression, the corresponding integral semimartingale indentity is only for a small time interval containing $t$, with $\cC$ being fixed and enclosing $x_a(t'), \cdots, x_{a+\al-1}(t')$, but not any $x_i(t')$ for $i<a$ or $i\ge a+\al$, for each $t'$ in the interval.

Note that
\[
\oint_\cC w\rd w =0, \quad \oint_\cC w \partial_w^2 Y_t(w) \rd w = \oint_\cC \sum_{i=1}^\infty \frac{2w \rd w}{(x_i(t)-w)^3} = 0, 
\]
and
\[
-\frac{1}{2\pi\ri} \oint_\cC \frac{w}{2}\partial_w (Y_t(w)^2) \rd w
=
\frac{1}{2\pi \ri} \oint_\cC \frac{Y_t(w)^2}{2} \rd w = \sum_{i=a}^{a+\al-1} W_t^{i,1}(x_i(t)).
\]
Note that the poles of $W_t^{i,1}$ are $x_1(t), x_2(t), \ldots$, except for $x_i(t)$.
Then we have
\begin{equation}  \label{eq:dxc2}
\rd \sum_{i=a}^{a+\al-1} x_i(t) = -\frac{1}{2\pi\ri}\oint_\cC w\rd M_t(w)\rd w+ \sum_{i=a}^{a+\al-1} W^{a,\al}_t(x_i(t))\rd t,
\end{equation}
using that $\sum_{i=a}^{a+\al-1} W_t^{i,1}(x_i(t))=\sum_{i=a}^{a+\al-1} W^{a,\al}_t(x_i(t))$.

Similarly, we have
\begin{align}\begin{split}\label{e:dx2}
    \rd \sum_{i=a}^{a+\al-1} x^2_i(t) 
    &=-\frac{1}{2\pi \ri}\oint_\cC  w^2 \rd Y_t(w)\rd w\\
    &=-\frac{1}{2\pi\ri}\oint_\cC w^2\rd w\left(\rd M_t(w)+\left(\frac{2-\beta}{2\beta }\del^2_w Y_t(w)+\frac{1}{2}\del_w (Y_t(w)^2)-\frac{1}{2}\right)\rd t\right)\\
    &=-\frac{1}{2\pi\ri}\oint_\cC w^2\rd M_t(w)\rd w+ \left(2\sum_{i=a}^{a+\al-1} W_t^{a,\al}(x_i(t)) x_i(t) +\alpha(\alpha-2)+\frac{2\alpha}{\beta}\right)\rd t.
\end{split}\end{align}
Here for the last equality, we used that
\[
\oint_\cC w^2\rd w =0, \quad -\frac{1}{2\pi\ri}\oint_\cC w^2 \partial_w^2 Y_t(w) \rd w = -\frac{1}{2\pi\ri}\oint_\cC \sum_{i=1}^\infty \frac{2w^2 \rd w}{(x_i(t)-w)^3} = 2\alpha, 
\]
\[
-\frac{1}{2\pi\ri} \oint_\cC \frac{w^2}{2}\partial_w (Y_t(w)^2) \rd w
=
\frac{1}{2\pi \ri} \oint_\cC wY_t(w)^2 \rd w = 2\sum_{i=a}^{a+\al-1} W_t^{i,1}(x_i(t))x_i(t),
\]
and that $\sum_{i=a}^{a+\al-1} W_t^{i,1}(x_i(t)) x_i(t)=\sum_{i=a}^{a+\al-1} W^{a,\al}_t(x_i(t))x_i(t) + \frac{\al(\al-1)}{2}$.

Now using \eqref{e:dx} and \eqref{e:dx2}, and It{\^o}'s formula, we can write $\rd Z(t)=\rd M(t)+V(t)\rd t$, with
\begin{align*}\begin{split}
    V(t) 
    =&2\alpha\sum_{i=a}^{a+\al-1} W_t^{a,\al}(x_i(t)) x_i(t) +\alpha^2(\alpha-2)+\frac{2\alpha^2}{\beta}\\
    &- 2\left(\sum_{i=a}^{a+\al-1} x_i(t)\right)\left(\sum_{i=a}^{a+\al-1} W_t^{i,1}(x_i(t))\right) -  \frac{2}{(2\pi\ri)^2\beta} \oiint_{\cC^2} 
ww'\del_w\del_{w'}\frac{Y_t(w)-Y_t(w')}{w-w'}
\rd w\rd w' \\
    =&2\sum_{1\leq i<j\leq \al} (x_i(t)-x_j(t))(W^{a,\al}_t(x_i)-W^{a,\al}_t(x_j))+\alpha^2(\alpha-2)+\frac{2\alpha(\al-1)}{\beta}.
\end{split}\end{align*}
Here we used \eqref{eq:qvmwwp} in the first equality.
For the second equality, it is by evaluating the contour integral in $w$ and $w'$, via integration by parts.

As for  $\rd M(t)$, we have
\[
    \rd M(t)=\frac{1}{2\pi\ri}\oint_\cC \left(2w\sum_{i=a}^{a+\al-1} x_i(t)-\alpha w^2\right)\rd M_t(w)\rd w.
\]
By \eqref{eq:qvmwwp}, the quadratic variation $\rd\langle M\rangle_t/\rd t$ therefore equals
\begin{align*}
&\frac{2}{(2\pi\ri)^2\beta} \oiint_{\cC^2} 
\left(2w\sum_{i=a}^{a+\al-1} x_i(t)-\alpha w^2\right)
\left(2w'\sum_{i=a}^{a+\al-1} x_i(t)-\alpha w'^2\right)\del_w\del_{w'}\frac{Y_t(w)-Y_t(w')}{w-w'}
\rd w\rd w'\\
=&
\frac{8}{(2\pi\ri)^2\beta} \oiint_{\cC^2} 
\left(\sum_{i=a}^{a+\al-1} x_i(t)-\alpha w\right)
\left(\sum_{i=a}^{a+\al-1} x_i(t)-\alpha w'\right)\frac{Y_t(w)-Y_t(w')}{w-w'}
\rd w\rd w'.
\end{align*}
By taking the $w'$ residues at $x_a(t),\ldots, x_{a+\al-1}(t)$, we get
\[
\frac{8}{2\pi\ri\beta} \oint_\cC 
\left(\sum_{i=a}^{a+\al-1} x_i(t)-\alpha w\right)
\left(\sum_{j=a}^{a+\al-1}\frac{\sum_{i=a}^{a+\al-1} x_i(t)-\alpha x_j(t)}{w-x_j(t)}\right)
\rd w.
\]
By further taking the $w$ residues at $x_1(t),\ldots, x_\al(t)$, this equals
\[
    \frac{8\alpha}{\beta} \left(\alpha  \sum_{i=a}^{a+\al-1} x_i^2(t)- \left(\sum_{i=a}^{a+\al-1} x_i(t)\right)^2\right)=\frac{8\alpha}{\beta} Z(t),
\]
and the conclusion follows.
\end{proof}
We next establish the collision time estimate, for poles whose indices are in an interval.
\begin{lem}  \label{lem:ZLocTime}
Under the same setup as \Cref{lem:SDEVM}, almost surely
\[
\int_{t_0}^\sigma \mathds{1}[x_a(t)=x_{a+1}(t)=\cdots =x_{a+\al-1}(t)] \rd t = 0.
\]
\end{lem}
\begin{proof}
Again, we write $Z(t)=Z(t)^{a,\al}$ and $V(t)=V(t)^{a,\al}$ in this proof.

We use the local time of $Z(t)$ to analyze its boundary behavior at zero. 
According to \Cref{lem:SDEVM}, $Z(t)$ for $t\in [t_0, \sigma)$ is a semi-martingale. We let $L_t^h$ be the level $h$ local time in $[t_0, \sigma)$ (with $L_{t_0}^h = 0$ for each $h\in\bR$). Then by \cite[Chapter VI, Theorem 1.7]{revuz2013continuous}, almost surely $L_t^h$ is continuous in $t$ and cadlag (right continuous) in $h$, and $L_t^0=L_t^0-L_t^{0-}$ satisfies
\begin{equation}\label{e:Lt0}
    L_\sigma^0
    =2\int_{t_0}^\sigma \mathds{1}(Z(t)=0) V(t)\rd t=\left(\alpha^2(\alpha-2)+\frac{2\alpha(\al-1)}{\beta}\right)\int_{t_0}^\sigma \mathds{1}(Z(t)=0)\rd t,
\end{equation}
where for the second equality, we used that if $Z(t)=0$, then $x_a(t)=\cdots=x_{a+\al-1}(t)$ and $V(t)=\alpha^2(\alpha-2)+\frac{2\alpha(\al-1)}{\beta}$.

Thanks to the occupation time formula \cite[Chapter VI, Corollary 1.6]{revuz2013continuous}, we have
\begin{align*}
    \int_0^\infty h^{-1} L_\sigma^h \rd h
    =\int_{t_0}^\sigma Z(t)^{-1}\rd \langle  Z\rangle_t 
    \le\frac{8\alpha}{\beta}<\infty,
\end{align*}
where we used  \eqref{e:dMtt} which gives
$\rd\langle  Z\rangle_t=8\al Z(t) \rd t/\beta$, and $\sigma \le t_0+1$.
Since $L^h_\sigma$ is right continuous in $h$, it follows that $L^0_\sigma=0$, and hence \eqref{e:Lt0} implies that almost surely $\int_{t_0}^\sigma \mathds{1}(Z(t)=0)\rd t=0$.
Thus the conclusion follows.
\end{proof}

\begin{proof}[Proof of \Cref{p:noncolliding}]
We take the following two steps.

\smallskip

\noindent\textbf{Step 1: Non-collision.}
We will first show \eqref{e:nocollision}, i.e., $\rd t$ almost everywhere poles do not collide. More precisely, we will prove inductively on $\ell=1,2,3,\cdots,k-1$
\begin{align}\label{e:hypo}
    \int_0^\tau\mathds{1}(x_1(t), x_2(t),\cdots, x_k(t) \text{ take at most $\ell$ distinct values} ) \rd t=0.
\end{align}
The claim of \Cref{p:noncolliding} follows from the case of $\ell=k-1$ in \eqref{e:hypo}.

For the base case where $\ell=1$, it follows from \Cref{lem:ZLocTime} with $a=1$ and $\al=k$, and $t_0=0$.
(Note that in this case, we always have $\sigma\ge \tau$)

We next give the induction step: if \eqref{e:hypo} holds for some $1\le\ell<k-1$, then it holds for $\ell+1$. 

Under the induction hypothesis, $\rd t$ almost everywhere, there exist $0< \alpha_1<\alpha_2<\cdots<\alpha_\ell<k$ such that 
\begin{align}\label{e:xgap}
   x_{\alpha_1}(t)> x_{\alpha_1+1}(t),\quad  x_{\alpha_2}(t)> x_{\alpha_2+1}(t),\quad 
   \cdots,
   \quad x_{\alpha_\ell}(t)> x_{\alpha_\ell+1}(t),
\end{align}
Fix the indices $0<\alpha_1<\alpha_2<\cdots<\alpha_\ell<k$, the set of time $t\in [0,\tau]$ such that \eqref{e:xgap} holds is a random open set, and we denote it by $I\subset [0,\tau]$.

 Next we show that almost surely, for almost every $t\in I$, $x_1(t), x_2(t),\cdots, x_k(t)$ take at least $\ell+2$ distinct values. This implies that \eqref{e:hypo} holds for $\ell+1$.
For the convenience of notations, we denote $\alpha_0=0$ and $\alpha_{\ell+1}=k$.
Since $\ell\leq k-2$, there exists some $\nu\in\llbracket 0, \ell\rrbracket$ such that $\alpha_{\nu+1}-\alpha_\nu\geq 2$. 
We then apply \Cref{lem:ZLocTime} with $a=\alpha_{\nu}+1$ and $\alpha=\alpha_{\nu+1}-\alpha_\nu$, and $t_0$ taking any rational numbers. 
We note that the union of all such $[t_0, \sigma)$ would cover $I$, therefore
\[
    \int_I \mathds{1}(x_a(t)=\cdots=x_{a+\al-1}(t))\rd t=0.
\]
Thus for almost every $t\in I$, $x_a(t),\cdots,x_{a+\al-1}(t)$ would take at least two distinct values, so we finish the induction step.

Then by induction principle, we finish the proof of \eqref{e:nocollision}.

\medskip

\noindent\textbf{Step 2: Dyson Brownian motion.}
We next prove \eqref{e:dbmsde}. For that we need to construct the Brownian motions $B_i(t)$ in \eqref{e:dbmsde}. 
By \eqref{e:nocollision} and that each $x_i(t)$ is continuous, for any $t\in [0, \tau]$ outside a closed measure zero set (i.e., in a countable union of open intervals, whose closure is $[0, \tau]$), we have $x_i(t)>x_{i+1}(t)$ for each $i\in\llbracket 1, k\rrbracket$.
Then we can take a small contour $\cC_i=\cC_{i,t}$ enclosing $x_i(t)$ but not any other poles.
From \eqref{eq:dxc2} in the proof of \Cref{lem:SDEVM}, we have
\begin{equation}  \label{e:dbmsdepf}
\rd x_i(t)= \sqrt{\frac{2}{\beta}} \rd B_i(t) + W_t^{i,1}(x_i(t)) \rd t, 
\end{equation}
where
\[
    \rd B_i(t)=-\sqrt{\frac{\beta}{2}}\frac{1}{2\pi\ri} \oint_{\cC_i}w\rd M_t(w)\rd w.
\]
We remark that, like \eqref{e:dx}, here the corresponding integral semimartingale indentity is for a small time interval containing $t$, such that $\cC_i$ is fixed and encloses $x_i(t')$ but not any other pole, for each $t'$ in the interval.
We can then further extend $B_i(t)$ to all of $[0, \tau]$ as a continuous process. The quadratic variations are given by
\[
    \frac{\rd}{\rd t}\langle B_i,  B_j\rangle_t
    =\frac{1}{(2\pi\ri)^2}\oiint_{\cC_i\times \cC_j}ww'\del_w \del_{w'}\frac{Y_t(w)-Y_t(w')}{w-w'}
    =\frac{1}{(2\pi\ri)^2}\oiint_{\cC_i\times \cC_j}\frac{Y_t(w)-Y_t(w')}{w-w'},
\]
which equals $\mathds{1}(i=j)$.
Then it follows that $\{B_i(t)\}_{i=1}^k$ are independent Brownian motions.

Noting that $W_t=W_t^{1,k}$, we get \eqref{e:dbmsde} from \eqref{e:dbmsdepf}. 
\end{proof}

\begin{remark}
    Another approach to study DBM developed in \cite{graczyk2014strong} is based on applying It{\^o}'s formula to the elementary symmetric functions $\sum_{1\leq j_1<j_2\cdots<j_n}x_{j_1}x_{j_2}\cdots x_{j_n}$. There a large family of Dyson type interacting particle systems are considered.
    For $\beta\ge 1$, they show that if the initial data has some particles at the same location, they will separate instantly. The method there could potentially be adapted and derive non-collision in the above proof as well.
\end{remark}
\section{Coupling and Uniqueness}\label{s:unique}
In this section we prove the uniqueness part of \Cref{t:characterize}.

Take any $\{Y_t\}_{t\in\bR}$ satisfying \Cref{a:Infinite} and \Cref{a:SDE}. Let $\{\bmx(t)\}_{t\in\bR}=\{x_i(t)\}_{i\in\bN, t\in\bR}$ be the line ensemble given by its poles (from \Cref{cor:31}).
We also take another line ensemble $\{\bmy(t)\}_{t\in\bR}=\{y_i(t)\}_{i\in\bN, t\in\bR}$ through the same way.
\begin{prop}\label{p:unique}
    The two line ensembles $\{\bmx(t)\}_{t\in\bR}$ and $\{\bmy(t)\}_{t\in\bR}$ have the same law. 
\end{prop}
Our general strategy is to construct a coupling of the dynamics in $t$, where $\bmx(t)$ and $\bmy(t)$ would get closer as $t$ increases. 
Then by sending the starting time of the dynamics to $-\infty$, one concludes that these two line ensembles must equal in law.
\medskip

\noindent\textbf{The coupling.}
There are four parameters $\delta, C, T, n$ in the definition of this coupling. 
Here $\delta, C>0$ are small and large real numbers;
 $-T$ is among the sequence $t_1, t_2, \cdots \to -\infty$ in \Cref{a:Infinite}, and $T$ is large enough depending on $\delta, C$; and we let $n=\lfloor T \rfloor$.
We shall mainly consider the dynamics of the first order $n$ many paths, for $t\in [-T, T]$. 

For each $t\in [-T, T]$, let $\cE[t]$ be the event where
\[
    |x_i(t)-\fa_i|, |y_i(t)-\fa_i|\leq \frac{C(\log T)^{40}}{i^\delta}, 
\]
for each $i\in \bN$.

For each $\ell \in \llbracket 0, 2Tn^3\rrbracket$, denote $t_\ell = -T+\ell n^{-3}$.
Under $\cE[t_\ell]$, by \eqref{e:aklocation}, we let $k_\ell$ and $k'_\ell$ be the smallest numbers in $\qq{n, 2n-1}$, such that 
\[
    |x_{k_\ell}(t_\ell)-x_{k_\ell+1}(t_\ell)|\geq \frac{2}{C'n^{1/3}},\quad |y_{k_\ell'}(t_\ell)-y_{k_\ell'+1}(t_\ell)|\geq \frac{2}{C'n^{1/3}},
\]
where $C'$ is a large enough universal constant. Note that when $n$ is large enough depending on $C'$, such $k_\ell$ and $k'_\ell$ exist.

We introduce a stopping time $\tau$ (with respect to the filtration $\cF_t=\sigma(\{\bmx(u)\}_{u\leq t}, \{\bmy(u)\}_{u\leq t})$), as follows.
If there exists any $t\in [-T, T]$ such that $\cE[t]$ does not hold, or if there is any $\ell \in \llbracket 0, 2Tn^3-1\rrbracket$ and $t\in [t_\ell, t_{\ell+1}]$, such that 
\[
    |x_{k_\ell}(t)-x_{k_\ell+1}(t)|\wedge |y_{k_\ell'}(t)-y_{k_\ell'+1}(t)|\leq \frac{1}{C'n^{1/3}},
\]
we let $\tau$ be the smallest such $t$.
Otherwise, we let $\tau=\infty$.

\begin{lem}  \label{lem:probtauinf}
For any $\varepsilon>0$, there exist $\delta, C>0$, such that for $T$ and $n=\lfloor T\rfloor$ large enough, under the above coupling we have $\bP[\tau=\infty]\ge 1-\varepsilon$.
\end{lem}
\begin{proof}
By \Cref{p:xrigidity}, for small enough $\delta$, large enough $C$ and $n$, we have $\bP\left[\bigcap_{t\in [-T, T]} \cE[t] \right]\ge 1-\varepsilon/2$.
Then by the H\"older continuity estimate \Cref{p:holder},
\[\bP[\tau=\infty]\ge 1-\varepsilon/2-2Tn^4e^{-cn^{1/6}} \ge 1-\varepsilon,\]
where $c>0$ is small enough depending on $C$, and the second inequality is by taking $n$ large.
\end{proof}

By \Cref{p:noncolliding}, we can find a family of independent Brownian motions $\{B_i\}_{i=1}^{2n}$, such that for each $\ell \in \llbracket 0, 2Tn^3-1\rrbracket$ and $t\in [t_\ell \wedge \tau, t_{\ell+1} \wedge \tau]$, $i\in \llbracket 1, k_\ell\rrbracket$, 
\begin{align}\begin{split}\label{e:xeq}
\rd x_i(t)&=\sqrt{\frac{2}{\beta}}\rd B_i(t) +\left(\sum_{1\leq j\leq k_\ell\atop j\neq i}
\frac{1}{x_i(t)-x_j(t)} + W_t(x_i(t)) \right)\rd t,\\
W_t(w)&= \frac{\Ai'(0)}{\Ai(0)} +\sum_{i=1}^{k_\ell}\frac{1}{\fa_i} +\sum_{i=k_\ell+1}^\infty\left( \frac{1}{w-x_i(t)}+\frac{1}{\fa_i}\right),
\end{split}\end{align}
where we used \Cref{cor:31} for the expression of $W_t$.
We can similarly  find a family of independent Brownian motions $\{\overline B_i\}_{i=1}^{2n}$, such that for each $\ell \in \llbracket 0, 2Tn^3-1\rrbracket$ and $t\in [t_\ell \wedge \tau, t_{\ell+1} \wedge \tau]$, $i\in \llbracket 1, k_\ell'\rrbracket$, 
\begin{align}\begin{split}\label{e:yeq}
\rd y_i(t)
&=\sqrt{\frac{2}{\beta}}\rd \overline B_i(t) +\left(\sum_{1\leq j\leq k_\ell'\atop j\neq i}
\frac{1}{y_i(t)- y_j(t)} +  \overline W_t(y_i(t)) \right)\rd t,\\
\overline W_t(w)
&= \frac{\Ai'(0)}{\Ai(0)} +\sum_{i=1}^{k_\ell'}\frac{1}{\fa_i} +\sum_{i=k_\ell'+1}^\infty\left( \frac{1}{w-y_i(t)}+\frac{1}{\fa_i}\right).
\end{split}\end{align}

We now couple $\{B_i\}_{i=1}^{2n}$ and $\{\overline B_i\}_{i=1}^{2n}$ in a common probablity space so that they equal almost surely. 
Thereby, we get a coupling between $\{\bmx(t)\}_{t\in\bR}$ and $\{\bmy(t)\}_{t\in\bR}$ in a common probability space.

The following proposition states that under this coupling, these two line ensembles are close to each other with high probability. 

\begin{prop}\label{p:x-ydelta}
  Fix any $\varepsilon, \theta>0$ and $S>0$. Then there exist $\delta, C>0$, such that for $T$ and $n=\lfloor T\rfloor$ large enough, under the above coupling with probability at least $1-\varepsilon$, 
  \[
      |x_i(t)-y_i(t)|\leq \theta, \quad \forall t\in [-S, S],\; i\in\llbracket 1, n\rrbracket.  
  \]
\end{prop}

In the following, we will prove that with probability at least $1-\varepsilon$, 
\begin{align}\label{e:upperb}
    x_i(t)\leq y_i(t)+ \theta,\quad \forall t\in [-S, S],\; i\in\llbracket 1, n\rrbracket.
\end{align} 
The lower bound that $x_i(t)\geq y_i(t)- \theta$ can be proven in the same way. 

Our strategy is to consider a shifted version of $\{\bmy(t)\}_{t\in\bR}$, which at $t=-T$ is much larger than $\bmx(-T)$;
then we show that it is larger than $\{\bmy(t)\}_{t\in\bR}$ for $t\in [-S, S]$ (under the coupling), while the amount of shift is $\le \theta$ in $[-S, S]$.

We now define the shifted version of $\{\bmy(t)\}_{t\in\bR}$. For any $t\in\bR$ and $i\in\bN$, we let
\[
\widetilde y_i(t)=y_i(t)+M -\kappa (t+T),
\]
where $M$ and $\kappa$ are taken as follows.
By \Cref{a:Infinite} and \Cref{cor:Ytdiff}, we take $M$ taken large enough (depending only on $\varepsilon$) such that with probability at least $1-\varepsilon/2$, 
\begin{align}\label{e:choseM}
    y_i(-T)+M>x_i(-T),\quad \forall i\in \bN.
\end{align}
We then take $\kappa$ such that
\[
M-\kappa (T-S)=\theta.
\]
Then for $n=\lfloor T\rfloor$ large enough (depending on $M, S, \theta$), the above choice of parameters imply
\begin{align}\label{e:halfdelta}
   \kappa < \frac{2M}{n}, \quad M-\kappa (S+T)=\theta -2S\kappa\geq \frac{\theta}{2}.
\end{align}

We can now rewrite \eqref{e:yeq} in terms of $\{\widetilde\bmy(t)\}_{t\in\bR}$. For each $\ell \in \llbracket 0, 2Tn^3-1\rrbracket$ and $t\in [t_\ell \wedge \tau, t_{\ell+1} \wedge \tau]$, $i\in \llbracket 1, k_\ell'\rrbracket$, we have
\begin{align}\begin{split}\label{e:wyeq}
\rd \widetilde y_i(t)
&=\sqrt{\frac{2}{\beta}}\rd B_i(t) +\left(\sum_{1\leq j\leq k_\ell'\atop j\neq i}
\frac{1}{\widetilde y_i(t)-\widetilde y_j(t)} + \widetilde W_t(\widetilde y_i(t))\right)\rd t,\\
\widetilde W_t(w)&= \frac{\Ai'(0)}{\Ai(0)}-\kappa+\sum_{i=1}^{k_\ell'}\frac{1}{\fa_i} +\sum_{i=k_\ell'+1}^{\infty}\left(\frac{1}{w-\widetilde y_i(t)}+\frac{1}{\fa_i}\right) .
\end{split}\end{align}

\begin{lem}\label{l:xycompare}
There exist $\delta, C$, such for any $T$ and $n=\lfloor T\rfloor$ large enough, under the above coupling the following holds. Take any $\ell\in \llbracket 0, (S+T)n^3\rrbracket$. Assuming that
\[
     \widetilde y_i(t_\ell\wedge \tau)> x_i(t_\ell\wedge \tau),\quad \forall i\in \llbracket 1, n\rrbracket,
\]
then 
\begin{align}\label{e:xydiffda}
     \widetilde y_i(t)>x_i(t),\quad  \forall t \in [t_\ell\wedge \tau, t_{\ell+1}\wedge \tau], \; i\in \llbracket 1, n\rrbracket.
\end{align}
\end{lem}
Assuming this lemma, we can now finish proving the uniqueness in law of line ensembles.
\begin{proof}[Proof of \Cref{p:x-ydelta}]
As already alluded to, it suffices to prove \eqref{e:upperb}.
    From our choice of $M$ (see \eqref{e:choseM}), we have that $\widetilde y_i(-T)>x_i(-T)$ for all $i\geq 1$. Then by repeatedly applying \Cref{l:xycompare} for $\ell\in \llbracket 0, (S+T)n^3\rrbracket$, and \Cref{lem:probtauinf}, we conclude that with probability at least $1-\varepsilon$, we have $\tau=\infty$ and $\widetilde y_i(t)>x_i(t)$ for all $t\in [-T, S]$ and $i\in\llbracket 1, n\rrbracket$. In particular for $t\in [-S,S]$, this gives
    \begin{align*}
        y_i(t)+\theta= y_i(t)+(M-\kappa(T-S))\geq y_i(t)+(M-\kappa (t+T))>x_i(t).
    \end{align*}
    This finishes the proof of \eqref{e:upperb}. 
\end{proof}
\begin{proof}[Proof of \Cref{p:unique}]
    The conclusion follows from taking $\varepsilon, \theta$ to zero and $S$ to infinity in \Cref{p:x-ydelta}.
\end{proof}

The rest of this section is devoted to proving \Cref{l:xycompare}.
The idea is straightforward: from the coupling we take the difference between \eqref{e:xeq} and \eqref{e:wyeq}, to cancel out the Brownian motions; and the rest are deterministic arguments.
\begin{proof}[Proof of \Cref{l:xycompare}]
For simplicity of notation, in this proof we fix $\ell$, and write $k=k_\ell$ and $k'=k_\ell'$. Recall that $k,k'\in \qq{n,2n-1}$.
We can take the difference  between \eqref{e:xeq} and \eqref{e:wyeq}, so that for any $i\in \llbracket 1, n\rrbracket$,
\begin{align}\begin{split}\label{e:yxdiff}
    &\phantom{{}={}}\rd(\widetilde y_i(t)-x_i(t))=\sum_{j\in \llbracket 1,n\rrbracket, j\neq i}\frac{(x_i(t)-\widetilde y_i(t))-(x_j(t)-\widetilde y_j(t))}{(\widetilde y_i(t)-\widetilde y_j(t))(x_i(t)-x_j(t))} \rd t\\
    &+\left(\widetilde W_t(\widetilde y_i(t))+\sum_{j=n+1}^{k'}\frac{1}{\widetilde y_i(t)-\widetilde y_j(t)}-W_t(x_i(t))-\sum_{j=n+1}^{k}\frac{1}{x_i(t)-x_j(t)}\right) \rd t.
\end{split}\end{align}

Denote the stopping time $\sigma$ to be the first time after $t_\ell\wedge\tau$, such that there exists at least one index $i_*\in \llbracket 1,n\rrbracket$ with $x_{i_*}(\sigma)=\widetilde y_{i_*}(\sigma)$ (if there were multiple such indices, take $i_*$ to be the smallest one). We will prove that $\sigma\geq t_{\ell+1}\wedge\tau$ then \eqref{e:xydiffda} holds. 

We prove by contradiction, and assume that $\sigma<t_{\ell+1}\wedge\tau$. By the definition of the stopping time $\tau$, for each $i\in\bN$, $i\ge n/3$, and $t\in [-T, S\wedge \tau]$, we have
\begin{equation}   \label{eq:xiyidif}
    x_i(t)\leq \fa_i+\frac{C(\log T)^{40}}{i^\delta} \leq y_i(t)+\frac{2C(\log T)^{40}}{i^\delta}< y_i(t)+M-\kappa (t+T)=\widetilde y_i(t).    
\end{equation}
We let $a$ (resp. $b$) be the smallest (resp. largest) index with $x_a(\sigma)=x_{i_*}(\sigma)$ (resp. $x_b(\sigma)=x_{i_*}(\sigma)$);
and we let $a', b'\in\bN$ be the corresponding indices for $y_{i_*}(\sigma)$.
By \eqref{eq:xiyidif}, and that $\widetilde y_i(\sigma)\geq x_i(\sigma)$ for each $i\in\llbracket 1, n\rrbracket$,  
necessarily $1\le a'\leq a\leq b'\leq b<n/2$.
Now for \eqref{e:yxdiff}, by summing  over $i \in \llbracket a, b'\rrbracket$, and integrating from $\sigma-\iota$ to $\sigma$ for a sufficiently small $\iota$, we have
\begin{align}\begin{split}\label{e:y-x}
     0>&\sum_{i=a}^{b'}(\widetilde y_i(t)-x_i(t))\Big|^{\sigma}_{\sigma-\iota}
     =\int_{\sigma-\iota}^\sigma \sum_{i=a}^{b'}\sum_{j\in \llbracket 1,n\rrbracket \setminus \llbracket a,b'\rrbracket}\frac{(x_i(t)-\widetilde y_i(t))-(x_j(t)-\widetilde y_j(t))}{(\widetilde y_i(t)-\widetilde y_j(t))(x_i(t)-x_j(t))}\rd t\\
     &+\int_{\sigma-\iota}^\sigma\sum_{i=a}^{b'}\bigg(\widetilde W_t(\widetilde y_i(t))+\sum_{j=n+1}^{k'}\frac{1}{\widetilde y_i(t)-\widetilde y_j(t)}-W_t(x_i(t))-\sum_{j=n+1}^{k}\frac{1}{x_i(t)-x_j(t)} \bigg)\rd t.
\end{split}\end{align}
Consider  the first term on the RHS of \eqref{e:y-x}.
Since $x_i(t), \widetilde y_i(t)$ are continuous, $\lim_{t\rightarrow \sigma} \widetilde y_i(t)-x_i(t)=0$ for $i\in \qq{a, b'}$, and $\lim_{t\rightarrow \sigma} \widetilde y_i(t)-x_i(t)>0$ for at least one $i\notin \qq{a, b'}$. Also, we have that $\lim_{t\rightarrow \sigma}(\widetilde y_i(t)-\widetilde y_j(t))(x_i(t)-x_j(t))\ge 0$, for each $i\in \llbracket a, b'\rrbracket$ and $j\in\llbracket 1,n\rrbracket\setminus \llbracket a,b'\rrbracket$. Thus
\begin{align}\label{e:ftm}
    \liminf_{\iota \rightarrow 0}\frac{1}{\iota}\int_{\sigma-\iota}^\sigma \sum_{i=a}^{b'}\sum_{j\in \llbracket 1,n\rrbracket \setminus\llbracket a,b'\rrbracket}\frac{(x_i(t)-\widetilde y_i(t))-(x_j(t)-\widetilde y_j(t))}{(\widetilde y_i(t)-\widetilde y_j(t))(x_i(t)-x_j(t))}\rd t>0.
\end{align}

We next consider the second term in the RHS of \eqref{e:y-x}. From the definition of $\tau$, $\cE[t]$ holds for any $t\in [t_\ell\wedge \tau, t_{\ell+1}\wedge \tau]$.
Thus when $n, C$ are large enough, for any $i \in \llbracket a, b'\rrbracket$ and $j>n$ it holds
\begin{align}\label{e:x-x}
    |x_i(t)-x_j(t)|,|\widetilde y_i(t)-\widetilde y_j(t)|\leq Cj^{2/3}.
\end{align}
We can rewrite the integrand as
\begin{align}\label{e:y-xbound}
    \sum_{i=a}^{b'}\sum_{j=n+1}^\infty\frac{x_i(t)-\widetilde y_i(t)}{(\widetilde y_i(t)-\widetilde y_j(t))(x_i(t)-x_j(t))}+\sum_{i=a}^{b'}-\kappa+\sum_{j=n+1}^\infty\frac{\widetilde y_j(t)-x_j(t)}{(\widetilde y_i(t)-\widetilde y_j(t))(x_i(t)-x_j(t))} .
\end{align}
Since $\lim_{t\rightarrow \sigma} \widetilde y_i(t)-x_i(t)=0$ for $i\in \qq{a, b'}$, and $\lim_{t\rightarrow \sigma}(\widetilde y_i(t)-\widetilde y_j(t))(x_i(t)-x_j(t)) > 0$ for $i\in \qq{a, b'}$ and $j>n$, the first term in \eqref{e:y-xbound} converges to zero as $t\rightarrow \sigma$. For the second term in \eqref{e:y-xbound}, using \eqref{e:x-x}, it is at least
\[
 \sum_{j=n+1}^\infty \frac{\widetilde y_j(t)-x_j(t)}{C^2j^{4/3}} -\kappa \geq \sum_{j=n+1}^\infty\frac{M-\kappa (t+T)-\frac{2C(\log T)^{40}}{ j^{\delta}}}{C^2j^{4/3}}-\kappa\geq \frac{\theta}{10C^2n^{1/3}}-\kappa\geq\kappa,
\]
where the first inequality is by $\cE[t]$; and in the second and last inequalities we used that
\[
M-\kappa (t+T)\geq M-\kappa (S+T)\geq \theta/2\geq (20C^2n^{1/3}\kappa)\vee(4C(\log T)^{40}/n^\delta),
\]
which is by \eqref{e:halfdelta}  and taking $n$ large enough (depending on $\varepsilon, \theta, S$ and $\delta, C$).
Then we conclude that 
\begin{multline}\label{e:stm}
\liminf_{\iota\rightarrow 0}\frac{1}{\iota}\int_{\sigma-\iota}^\sigma\sum_{i=a}^{b'}\bigg(\widetilde W_t(\widetilde y_i(t))+\sum_{j=n+1}^{k'}\frac{1}{\widetilde y_i(t)-\widetilde y_j(t)}-W_t(x_i(t))-\sum_{j=n+1}^{k}\frac{1}{x_i(t)-x_j(t)}\bigg)\rd t\\ > 0.
\end{multline}
Combining \eqref{e:ftm} and \eqref{e:stm}, we conclude that for $\iota>0$ small enough, the RHS of \eqref{e:y-x} is positive, which leads to a contradiction. 
\end{proof}

\section{Convergence to the Airy$_\beta$ line ensemble}
\label{s:process}

The dynamical versions of the three classical ensembles as in \eqref{e:classicalensemble} correspond to Dyson Brownian motion (DBM), the Laguerre process, and the Jacobi process, all of which have been intensively studied in the literature, as seen in \cite{MR2760897,bru1991wishart,demni2010beta,konig2001eigenvalues}. 
In this section we prove that their edge limit is ALE$_\beta$, i.e., \Cref{t:convergence_Airy}.
In particular, for DBM our method covers more general potentials (beyond quadratic ones). 
Our strategy is to prove tightness, and verify that any subsequential limit satisfies \Cref{a:Infinite} and \Cref{a:SDE}, and then apply \Cref{t:characterize}.

We now formally introduce these processes, from a random matrix theory perspective.

\smallskip
\noindent\textbf{DBM (and with general potential).}
Let $B_t=(B_{ij}(t))$ be an $n\times n$ real/complex Brownian matrix (with each entry given by Brownian motion $B(t)$ for the real case; and given by $(B(t)+\hat B(t)\ri)/\sqrt{2}$ with $B(t)$ and $\hat B(t)$ being independent Brownian motions for the complex case), and define $X_t=(B_t+B_t^*)/\sqrt{2}$ (where $B_t^*$ is the complex conjugate of $B_t$). Then the eigenvalues of $X_t$ (denoted by $\{\lambda^{(n)}_i(t)\}_{i=1}^n$) satisfy
\[
		\rd \lambda^{(n)}_i (t) = \sqrt{\frac{2}{\beta}} \rd B_i (t)+
  \displaystyle\sum_{\substack{1 \le j \le n \\ j \ne i}}	\displaystyle\frac{\rd t}{\lambda^{(n)}_i (t) - \lambda^{(n)}_j (t)},
	\]
with $\beta=1$ (real case) or $2$ (complex case).
At time $t=1$ the law is given by the Hermite/Gaussian ensemble. 

More generally, one can consider DBM with potential $V$ and any $\beta>0$, 
\begin{align}
\label{e:DBMV}
  \rd \lambda^{(n)}_i (t) = \sqrt{\frac{2}{\beta}} \rd B_i (t)+
  \displaystyle\sum_{\substack{1 \le j \le n \\ j \ne i}}	\displaystyle\frac{\rd t}{\lambda^{(n)}_i (t) - \lambda^{(n)}_j (t)}-\frac{\sqrt{n}}{2}V'\left(\frac{\lambda^{(n)}_i(t)}{\sqrt n}\right)\rd t.
\end{align}

Under specific conditions for the potential $V$ (refer to \Cref{a:Vasump}), the rescaled particle configurations $\lambda_i^{(n)}(t)/\sqrt n$ (adopting the notations from \cite{bourgade2022optimal}) have a $\beta$-ensemble as the stationary measure, with the following probability density proportional to:
\begin{align}\label{e:beta}
    \frac{1}{Z_{n,\beta,V}}\prod_{1\leq i<j\leq n}|x_i-x_j|^\beta \prod_{i=1}^n e^{-\frac{\beta n}{2}\sum_{i=1}^n V(x_i)},
\end{align}
where $Z_{n,\beta,V}$ is a renormalization constant. Under mild conditions of $V(x)$ (see e.g., \cite[Theorem 1]{de1995statistical}), there exists a unique equilibrium measure $\mu_V$ characterized by the following variational principle 
\begin{align}\label{e:variational_principle}
    \mu_V:=\argmin_{\mu}\left\{-\int_{\bR^2} \log|x-y|\rd \mu(x)\rd \mu(y) +\int_\bR V(x)\rd \mu(x)\right\}
\end{align}
where the minimization is taken over all probability measures on $\bR$.

We shall work on DBM with potentials under the following technical assumptions.
\begin{assumption}\label{a:Vasump}
    The potential function $V(x)$ satisfies 
    \begin{itemize}
        \item It is analytic on $\bR$. 
        \item There exist constants $M_0, C,c>0$ such that $V'(x)\geq c$ and $\sup_{y\in [M_0, x]}V'(y)/y\leq CV(x)$ for all $x\geq M_0$, and similar estimates apply for $x\leq -M_0$.
    \end{itemize}
    Under the previous assumptions, it is known that there exists a unique equilibrium measure $\mu_V$ on $\bR$ characterized by the variational principle \eqref{e:variational_principle}. We further assume $V(x)$ satisfies
    \begin{itemize}
        \item The measure $\mu_V$ has a density $\varrho_V$, which is positive and supported on a single interval $[A,B]$, with square root singularities at $A$ and $B$. 
        More precisely, there exists $R_A>0$ and $R=R_B>0$, such that $\lim_{x\to 0_+}x^{-1/2}\varrho_V(A+x)=R_A/\pi$ and $\lim_{x\to 0_+}x^{-1/2}\varrho_V(B-x)=R/\pi$.
        \item The function $x\mapsto V(x)/2-\int \log|x-y|\rd \mu_V(y)$ achieve its minimum value only in the interval $[A,B]$. 
    \end{itemize}
\end{assumption}
In particular, \Cref{a:Vasump} is satisfied by any $V$ that is analytic and strongly convex (see \cite{de1995statistical}).
Next we recall some estimates of the equilibrium density $\varrho_V$ and its Stieltjes transform 
\begin{equation}   \label{e:defmv}
    m_V(z)=\int_A^B \frac{\varrho_V(x)\rd x}{x-z}, \quad z\in \bC\setminus[A,B],
\end{equation}
from \cite[Section 2.1]{bourgade2022optimal}.

By \Cref{a:Vasump}, $V(x)$ is analytic on $\bR$, so it can be extended analytically to a simply connected open set $\Omega$ of the complex plane, which contains $[A,B]$. The equilibrium density $\varrho_V(x)$ is supported on $[A,B]$, given explicitly by 
\begin{align}\label{e:rhoV}
    \varrho_V(x)=\frac{r(x)}{\pi}\sqrt{(x-A)(B-x)}=\frac{R}{\pi}\sqrt{B-x} +\OO(|B-x|^{3/2}), \quad x\in[A,B],
\end{align}
where 
\[
    r(z)=\frac{1}{2\pi}\int_A^B \frac{V'(z)-V'(x)}{z-x}\cdot\frac{\rd x}{\sqrt{(x-A)(B-x)}},
\]
is analytic in $\Omega$, with $R_A=r(A)\sqrt{B-A}$ and $R=R_B=r(B)\sqrt{B-A}$. And
\begin{align}\label{e:mV}
    m_V(z)=\frac{-V'(z)+2r(z)\sqrt{(z-A)(z-B)}}{2}=-\frac{V'(B)}{2}+R\sqrt{z-B}+\OO(|z-B|),
\end{align}
for $z\in \Omega\setminus [A,B]$. 

\noindent\textbf{Laguerre process.}
Let $B_t=(B_{ij}(t))$ be an $n\times m$ real/complex Brownian matrix, and define $X_t=B_tB_t^*$. Assume $n\leq m$, then the evolution of eigenvalues of $X_t$ (denoted by $\{\lambda^{(n)}_i(t)\}_{i=1}^n$) is given by the Laguerre process
\begin{flalign}
		\label{e:Laguerre}
		\rd \lambda^{(n)}_i (t) = \frac{2}{\sqrt \beta}\sqrt{\lambda^{(n)}_i(t)} \rd B_i (t)+
  \left(m+\displaystyle\sum_{\substack{1 \le j \le n \\ j \ne i}}	\displaystyle\frac{\lambda^{(n)}_i(t)+\lambda^{(n)}_j(t)}{\lambda^{(n)}_i (t) - \lambda^{(n)}_j (t)}\right)\rd t,
	\end{flalign}
where $\beta=1$ (real case) or $2$ (complex case); while the Laguerre process, i.e., solution to \eqref{e:Laguerre}, is also considered for any $\beta>0$. At time $t=1$ the law is given by the Laguerre ensemble. 
There is also a stationary version:
\begin{equation}
		\label{e:stationary_Laguerre}
		\rd \lambda^{(n)}_i (t) = \frac{2}{\sqrt \beta}\sqrt{\lambda^{(n)}_i(t)} \rd B_i (t)+
  \left(m+\displaystyle\sum_{\substack{1 \le j \le n \\ j \ne i}}	\displaystyle\frac{\lambda^{(n)}_i(t)+\lambda^{(n)}_j(t)}{\lambda^{(n)}_i (t) - \lambda^{(n)}_j (t)}\right)\rd t-\lambda_i^{(n)}\rd t,
	\end{equation}
whose equilibrium measure is given by the Laguerre ensemble in \eqref{e:classicalensemble}. 
\smallskip

\noindent\textbf{Jacobi process.}
Let $\Theta(t)$ be the Brownian motion on the $m\times m$ orthogonal/unitary group. Take $p+q=m$ and $p\geq n+1, q\geq n+1$ and denote $C(t)$ the left corner of $\Theta(t)$ with size $n\times p$. Then the evolution of the eigenvalues of $C(t) C(t)^*$ (denoted by $\{\lambda^{(n)}_i(t)\}_{i=1}^n$) is given by the Jacobi process
\begin{equation}
\begin{split}
\label{e:Jacobi}
    \rd \lambda^{(n)}_i (t) 
    &= \frac{2}{\sqrt{\beta}}\sqrt{\lambda^{(n)}_i(t)(1-\lambda^{(n)}_i(t))} \rd B_i (t)\\
    &+
  \left(p-m\lambda^{(n)}_i(t)+\displaystyle\sum_{\substack{1 \le j \le n \\ j \ne i}}	\displaystyle\frac{\lambda^{(n)}_i(t)(1-\lambda^{(n)}_i(t))+\lambda^{(n)}_j(t)(1-\lambda^{(n)}_j(t))}{\lambda^{(n)}_i (t) - \lambda^{(n)}_j (t)}\right)\rd t,
\end{split}
\end{equation}
where $\beta=1$ and $2$ correspond to the orthogonal and unitary cases, respectively. Again, the Jacobi process, i.e., solution to \eqref{e:Jacobi}, is also considered for any $\beta>0$.
We refer to \cite[Chapter 9]{doumerc2005matrices} and \cite[Section 1.2]{demni2010beta} for detailed discussions of this matrix Jacobi process and the derivation of \eqref{e:Jacobi}\footnote{Note that compared to the definition of the Jacobi process in \cite{demni2010beta}, here we rescale time by a factor of $\beta$.}.
The equilibrium measure of \eqref{e:Jacobi} is given by the Jacobi ensemble in \eqref{e:classicalensemble}. 

For the special case with $\beta=1,2,4$, the Jacobi ensemble also describes the eigenvalues of MANOVA (multivariate analysis of variance) matrices. Let $W=(W_1, W_2)$ be an $n\times m$ real/complex Gaussian matrix, and $W_1$ consists of its first $p$ columns and $W_1$ consists of its last $q$ columns. The eigenvalues of the matrix $(W_1 W_1^*+W_2W_2^*)^{-1/2}W_1 W_1^*(W_1 W_1^*+W_2W_2^*)^{-1/2}$ are given by the Jacobi ensemble.

Given the above setup, we now state the more precise version of \Cref{t:convergence_Airy}.
\begin{thm}  \label{t:converge_Airy_details}
For each $n\in\bN$, let $\{\lambda^{(n)}_i(t)\}_{i=1}^n$ be either the stationary DBM \eqref{e:DBMV} with fixed general potential $V$ satisfying \Cref{a:Vasump}, Laguerre process \eqref{e:stationary_Laguerre}, or Jacobi process \eqref{e:Jacobi}.
Take $n\to\infty$, with
\begin{itemize}
    \item $\limsup m/n<\infty$ and $\liminf m/n >1$ in the Laguerre case;
    \item $\limsup p/n, \limsup q/n <\infty$, $\liminf p/n>1$, $\liminf q/n>0$ in the Jacobi case.
\end{itemize}
Then we have that
$\{(\lambda^{(n)}_i(\zeta t) - E)/\chi \}_{i=1}^\infty$
converges to ALE$_\beta$, under the uniform in compact topology. Here we take the convention of $\lambda^{(n)}_i=-\infty$ for $i>n$, and $E, \zeta, \chi$ are as follows:
\begin{itemize}
    \item[DBM:] 
    \begin{equation}   \label{eq:DBMc}
        E=Bn^{1/2}, \zeta=R^{-4/3}n^{-1/3}, \chi=R^{-2/3}n^{-1/6},
    \end{equation}
    where $B$, $R$, and $\varrho_V$ from \Cref{a:Vasump};
    \item[Laguerre:]
    \begin{equation}   \label{eq:Lagc}
        E=(\sqrt{m}+\sqrt{n})^2, \zeta=2^{-1}(\sqrt{m}+\sqrt{n})^{2/3}(mn)^{-1/3}, \chi=(\sqrt{m}+\sqrt{n})^{4/3}(mn)^{-1/6};
    \end{equation}
    \item[Jacobi:] 
    \begin{equation}   \label{eq:Jacc}
        E=\left(\frac{\sqrt{p(m-n)}+\sqrt{qn}}{m}\right)^2, \zeta=\frac{(E(1-E))^{1/3}}{2(pq(m-n))^{1/3}}n^{-1/3}, \chi=\frac{(E(1-E))^{2/3}}{(pq(m-n))^{1/6}}n^{-1/6}.
    \end{equation}
\end{itemize}
\end{thm}

\smallskip
The rest of this section is devoted to the proof of this result.
We note that, effectively, our proof also gives a self-contained contruction of ALE$_\beta$, which is defined and shown to be the edge limit of quadratic potential DBM in \cite{GXZ}.

\subsection{Scaling of particles and Stieltjes transform}
Our proof of convergence to ALE$_\beta$ consists of the following tasks: 
\begin{enumerate}
    \item[(1)] Write out the SDE satisfied by the (rescaled) Stieltjes transform; 
    \item[(2)] Establish tightness of  the (rescaled) Stieltjes transform and particles; 
    \item[(3)] Verify \Cref{a:Infinite} and \Cref{a:SDE} for any subsequential limit.
\end{enumerate}
This subsection is for task (1). 

For each one of the three processes, denote the Stieltjes transform
\begin{equation}  \label{eq:mntz}
 m^{(n)}_t(z)=\sum_{i=1}^n \frac{1}{\lambda^{(n)}_i(t)-z},\quad z\in \bC.
\end{equation}
And we let  $\wt \la^{(n)}_i(t)=(\la^{(n)}_i(\zeta t)-E)/\chi$.
We shall define a certain (time-evolving) particle-generated Nevanlinna function $Y^{(n)}_t$ with poles $\{\widetilde \lambda^{(n)}_i(t)\}_{i=1}^n$, through a rescaling from $m^{(n)}_t$; and as $n\to\infty$, such $Y^{(n)}_t$ should converge to $Y_t$ in \Cref{a:Infinite} and \Cref{a:SDE}. 

In light of the Airy-like property, we will instead work with $\Delta^{(n)}_t(w)=Y^{(n)}_t(w)-\sqrt{w}$, which we next define for the three cases, respectively.

\subsubsection{Scaling of particles and Stieltjes transform} 
\phantom{}
\\
\noindent\textbf{DBM.}
We consider the DBM \eqref{e:DBM} with general potential $V$ starting from the stationary distribution. 
Then the law of $\{\lambda^{(n)}_i(t)/\sqrt n\}_{i=1}^n$ for every fixed $t\in\bR$ is the $\beta$ ensemble \eqref{e:beta}. In light of the measure $\varrho_V$, $\{\lambda^{(n)}_i(t)\}_{i=1}^n$ for a fixed $t$ should fill in the interval of $[A\sqrt{n}, B\sqrt{n}]$, with square root behavior of density near the edges. 
Take $m_V$ as in \eqref{e:defmv}, and $w\in n^{1/2}\chi^{-1}(\Omega\setminus [A, B])$ for some simple connected open set $\Omega\subset \bC$ containing $[A, B]$.
By \eqref{eq:DBMc}, $\chi m^{(n)}_t(E+\chi w)$ is approximately $\chi \sqrt n m_V(B+n^{-1/2} \chi w)$, which equals $-\frac{\chi\sqrt{n}V'(B)}{2} + \sqrt{w}+\OO(n^{-1/3} |w|)$\footnote{Here the constant in upper bounding the $\OO$ term depends on $V$ and the set $\Omega$.} according to \eqref{e:mV}.
Therefore, we let
\begin{align}\begin{split}\label{e:deltasw}
    \Delta^{(n)}_{t}(w)
    &=\chi\left( m^{(n)}_{\zeta t}(E+\chi w)+ \frac{\sqrt nV'(B)}{2} \right)-\sqrt w\\
    &=\chi \left(m^{(n)}_{\zeta t}(E+\chi w)- \sqrt n m_V(B+n^{-1/2} \chi w)\right)+\OO(n^{-1/3} |w|).
\end{split}\end{align}
Then since $\wt \la^{(n)}_i(t)=(\la^{(n)}_i(\zeta t)-E)/\chi= R^{2/3}n^{1/6}\lambda^{(n)}_i(tR^{-4/3} n^{-1/3})-R^{2/3}B n^{2/3}$, we also have
\[
		\Delta^{(n)}_t (w) +\sqrt w= \sum_{i=1}^n	\displaystyle\frac{1}{\widetilde\lambda^{(n)}_i (t) - w} +\frac{R^{-2/3}V'(B)}{2}n^{1/3}.
\]

\noindent\textbf{Laguerre.}
It is known (see e.g., \cite{dumitriu2006global}) that the density of $\{\lambda^{(n)}_i(t)\}_{i=1}^n$ for a fixed $t$ is approximated by the (rescaled) Marchenko-Pastur law
\begin{equation}  \label{eq:defmpl}
    \varrho^{(n)}_{\rm mp}(x)=\frac{\sqrt{(x-E_-)(E_+-x)}}{2\pi x},
\end{equation}
supported on the interval $[E_-, E_+]$, where $E_- = (\sqrt{m}-\sqrt{n})^2$ and $E_+=E=(\sqrt{m}-\sqrt{n})^2$.
Taking its Stieltjes transform we get
\[
     m^{(n)}_{\rm mp}(z)=\int \frac{\varrho^{(n)}_{\rm mp}(x)\rd x}{x-z} = \frac{-(z-m+n )+\sqrt{(z-(m+n))^2-4mn }}{2z},
\]
which should approximate $m^{(n)}(z)$.
In particular, we have $m^{(n)}_{\rm mp}(E)=-\frac{\sqrt{n}}{\sqrt{m}+\sqrt{n}}$.
In light of this, we rescale $m^{(n)}_t(z)$ in time by $\zeta$ and in space by $\chi$ (from \eqref{eq:Lagc}), and denote
\[
    \Delta^{(n)}_t(w)
    =\chi\left( m^{(n)}_{\zeta t}(E+\chi w)+\frac{\sqrt n}{\sqrt m+\sqrt n} \right)-\sqrt w.
\]
Then we have
\begin{equation}  \label{eq:Deltmeq}
\Delta^{(n)}_t(w) +\sqrt w=\sum_{i=1}^n \frac{1}{\wt\lambda^{(n)}_i(t)-w} + \frac{\chi\sqrt n}{\sqrt m+\sqrt n}.
\end{equation}

\noindent\textbf{Jacobi.}
By e.g., \cite{dumitriu2012global}, the density of $\{\lambda^{(n)}_i(t)\}_{i=1}^n$ for a fixed $t$ is approximated by the following law:
\begin{equation}  \label{eq:defJar}
\varrho^{(n)}_{\rm Ja}(x)=
\frac{m\sqrt{(x-E_-)(E_+-x)}}{2\pi x(1-x)},    
\end{equation}
supported on the interval $[E_-, E_+]$, where $E_- = \left(\frac{\sqrt{p(m-n)}-\sqrt{qn}}{m}\right)^2$ and $E_+=E=\left(\frac{\sqrt{p(m-n)}+\sqrt{qn}}{m}\right)^2$.
Taking its Stieltjes transform we get
\[
     m^{(n)}_{\rm Ja}(z)=\int \frac{\varrho^{(n)}_{\rm Ja}(x)\rd x}{x-z} = \frac{-(mz-2nz+n-p)  +\sqrt{(mz+n-p)^2-4znq }}{2z(1-z)},
\]
which should approximate $m^{(n)}(z)$.
Recall the scaling from \eqref{eq:Jacc}. In particular, we have $m^{(n)}_{\rm Ja}(E)=-\frac{mE-2nE+n-p}{2E(1-E)}$.
Thus we denote 
\[
    \Delta^{(n)}_t(w)
    =\chi\left( m^{(n)}_{\zeta t}(E+\chi w)+\frac{mE-2nE+n-p}{2E(1-E)} \right)-\sqrt w,
\]
and we have
\[
\Delta^{(n)}_t(w) +\sqrt w=\sum_{i=1}^n \frac{1}{\wt\lambda^{(n)}_i(t)-w} + \frac{\chi(mE-2nE+n-p)}{2E(1-E)}.
\]

\bigskip

Having defined the rescaled particles and the Stieltjes transform, we now fix notation. Unless stated otherwise, the symbols introduced above—for example \(\{\widetilde{\lambda}^{(n)}_i(t)\}_{i\in\mathbb{N},\,t\in\mathbb{R}}\) and \(\{\Delta^{(n)}_{t}(w)+\sqrt{w}\}_{w\in\mathbb{H},\,t\in\mathbb{R}}\)—refer to any one of the three above cases.

\subsubsection{Rescaled SDE and error terms}

We now present the SDE satisfied by $\Delta^{(n)}_t$.
\begin{prop}\label{prop:Dteq}
The following SDE is satisfied:
\begin{multline}\label{e:Dteq}
\rd \Delta^{(n)}_t(w)=\rd M^{(n)}_t(w)
\\+\left(\frac{2-\beta}{2\beta}\del_w^2(\Delta^{(n)}_t(w)+\sqrt w)
+\frac{1}{2}\del_w(\Delta^{(n)}_t(w)+\sqrt w)^2-\frac{1}{2}+\cE^{(n)}_t(w)\right)\rd t.
\end{multline}
Here $\cE^{(n)}_t(w)$ is some error term,  
and $M^{(n)}_t(w)$ is the martingale term, with quadratic variation given by
\begin{equation}  \label{e:Mnvardu}
    \frac{\rd}{\rd t} \langle  M^{(n)}(w), M^{(n)}(w')\rangle_t
    =\frac{2}{\beta}\sum_{i=1}^n \frac{1}{(\widetilde \lambda^{(n)}_i(t)-w)^2(\widetilde \lambda^{(n)}_i(t)-w')^2} + \hcE^{(n)}_t(w,w'),
\end{equation}
for any $w, w'\in \bC\setminus \bR$, where $\hcE^{(n)}_t(w,w')$ is some other error term.

These error terms satisfy the following estimates.
Take any compact $\cK\subset \bH$.
For any $w, w' \in \cK\cup \overline{\cK}$, there is
\begin{equation}  \label{e:herr}
|\cE^{(n)}_t(w)|,\; |\partial_w\cE^{(n)}_t(w)| \lesssim n^{-1/3}(1+|\Delta^{(n)}_t(w)+\sqrt w|)^2,
\end{equation}
and
\begin{equation}\label{e:err}
|\hcE^{(n)}_t(w, w')|,\; |\partial_w\partial_{w'}\hcE^{(n)}_t(w, w')| \lesssim n^{-2/3}(1+|\Delta^{(n)}_t(w)+\sqrt w|),
\end{equation}
where all the constants behind $\lesssim$ can depend on $\cK$.
\end{prop}
To draw connection with \Cref{a:SDE}, we note that the quadratic variation of the martingale term can also be written as
\begin{equation}  \label{e:Mnvard}
    \frac{\rd}{\rd t}\langle M^{(n)}(w)\rangle_t
    =\frac{1}{3\beta}\del_w^3 (\Delta^{(n)}_t(w)+\sqrt w)+ \hcE^{(n)}_t(w,w),
\end{equation}
and
\begin{equation}  \label{e:Mnvar}
    \frac{\rd }{\rd t}\langle M^{(n)}(w),  M^{(n)}(w')\rangle_t
    =\frac{2}{\beta}\del_w \del_{w'}\left(\frac{(\Delta^{(n)}_t(w)+\sqrt w)-(\Delta^{(n)}_t(w')+\sqrt w')}{w-w'}\right)+ \hcE^{(n)}_t(w,w'),
\end{equation}
for $w\neq w'$. 

\begin{proof}[Proof of \Cref{prop:Dteq}]
We prove this for the cases respectively. Constants in all $\lesssim$  in this proof are allowed to depend on $\cK$.
\medskip

\noindent\textbf{DBM.}
For simplicity, below, we only consider the stationary DBM in \eqref{e:DBMV} with $V(x)=x^2/4$. The same arguments apply to any potential satisfying \Cref{a:Vasump}, essentially verbatim. More precisely, we now consider $\{\lambda^{(n)}_i(t)\}_{i=1}^n$ satisfying
		\begin{flalign}
		\label{e:DBM}
		\rd \lambda^{(n)}_i (t) = \sqrt{\frac{2}{\beta } } \rd B_i (t)+
\displaystyle\sum_{\substack{1 \le j \le n \\ j \ne i}}	\displaystyle\frac{\rd t}{\lambda^{(n)}_i (t) - \lambda^{(n)}_j (t)} -\frac{1}{2}\lambda^{(n)}_i(t)\rd t,
	\end{flalign}
with the law of $\{\lambda^{(n)}_i(t)\}_{i=1}^n$ for every fixed $t\in\bR$ being the Hermite/Gaussian $\beta$ ensemble. In this setting, $ \Delta^{(n)}_{t}(w)$ from \eqref{e:deltasw} simplifies as 
\[
    \Delta^{(n)}_{t}(w)
    =n^{-1/6}(  m^{(n)}_{tn^{-1/3}}(2\sqrt n+wn^{-1/6})+\sqrt n)-\sqrt w.
\]

By It\^{o}'s formula, $m^{(n)}_t(z)$ satisfies a Burgers type SDE on $\bH$:
\begin{align}\begin{split}\label{eq:dm}
& \rd  m^{(n)}_t(z)
=-\sum_{i=1}^n \frac{\rd \la^{(n)}_i(t)}{(\la^{(n)}_i(t)-z)^2}
+\sum_{i=1}^n \frac{\rd\langle  \la^{(n)}_i\rangle_t}{(\la^{(n)}_i(t)-z)^3}\\
&=-\sqrt{\frac{2}{\beta }}\sum_{i=1}^n \frac{{\rm d} B_i(t)}{(\la^{(n)}_i(t)-z)^2}+\frac{1}{2}\del_z( (m^{(n)}_t)^2(z)+z m^{(n)}_t(z))\rd t +\frac{2-\beta}{2\beta }\del^2_z m^{(n)}_t(z)\rd t,
\end{split}\end{align}
where the last line follows from plugging in \eqref{e:DBM}, and the following computations:
\[
\del_z m_t^{(n)}(z) = \sum_{i=1}^n \frac{1}{(\la^{(n)}_i(t)-z)^2}, \quad
\del_z^2 m_t^{(n)}(z) = \sum_{i=1}^n \frac{2}{(\la^{(n)}_i(t)-z)^3}, \quad
\]
\[
\del_z (m_t^{(n)})^2(z) = \sum_{i\neq j}^n \frac{2}{(\la^{(n)}_i(t)-z)(\la^{(n)}_j(t)-z)^2} + \sum_{i=1}^n \frac{2}{(\la^{(n)}_i(t)-z)^3},
\]
and then
\begin{align*}
& \sum_{i=1}^n \frac{1}{(\la^{(n)}_i(t)-z)^2}\left(\sum_{1\leq j\leq n,\atop j\neq i}\frac{1}{\lambda^{(n)}_i(t)-\lambda^{(n)}_j(t)}-\frac{\lambda^{(n)}_i(t)}{2}\right) \\
& =\frac{1}{2}\sum_{i\neq j}^n \bigg( \frac{1}{(\la^{(n)}_i(t)-z)^2}\cdot\frac{1}{\lambda^{(n)}_i(t)-\lambda^{(n)}_j(t)}+\frac{1}{(\la_j^{(n)}(t)-z)^2}\cdot\frac{1}{\lambda^{(n)}_j(t)-\lambda^{(n)}_i(t)} \bigg) -\frac{1}{2} z \del_z m_t^{(n)}(z)-\frac{1}{2}m_t^{(n)}(z)\\
&= -\frac{1}{2}\sum_{i\neq j}^n\frac{\la^{(n)}_i(t)-z+ \la_j(t)-z}{(\la^{(n)}_i(t)-z)^2(\lambda^{(n)}_j(t)-z)^2}-\frac{1}{2}  \del_z (z m_t^{(n)}(z))) \\
& =-\sum_{i\neq j}^n\frac{1}{(\la^{(n)}_i(t)-z)(\lambda^{(n)}_j(t)-z)^2}-\frac{1}{2}  \del_z (z m_t^{(n)}(z))) \\
&= -\frac{1}{2}\del_z( (m^{(n)}_t)^2(z)+z m^{(n)}_t(z))+\frac{1}{2} \del_z^2 m^{(n)}_t(z).
\end{align*}
We can rewrite \eqref{eq:dm} as \eqref{e:Dteq}, with
\[\cE^{(n)}_t(w)=\frac{1}{2n^{1/3}}\partial_w \left( w(\Delta^{(n)}_t(w)+\sqrt w)\right), \quad \hcE^{(n)}_t(w,w')=0.\]
Using \Cref{l:Yproperty} and noticing that $\overline{\cE_t^{(n)}(w)}=\cE_t^{(n)}(\overline{w})$, we can bound the error term $\cE^{(n)}_t(w)$ and its derivative by
\begin{equation*}
\begin{split}
         &| \cE^{(n)}_t(w)|\lesssim \frac{|w||\Im[\Delta^{(n)}_t(w)+\sqrt w]|}{\Im[w]n^{1/3}}+\frac{|\Delta^{(n)}_t(w)+\sqrt w|}{n^{1/3}}, \\
         &| \partial_w\cE^{(n)}_t(w)|\lesssim \frac{|w||\Im[\Delta^{(n)}_t(w)+\sqrt w]|}{\Im[w]^2n^{1/3}}+\frac{|\Im[\Delta^{(n)}_t(w)+\sqrt w]|}{\Im[w]n^{1/3}}.
\end{split}
\end{equation*}

\noindent\textbf{Laguerre.}
We will first use It\^o's formula applied to \eqref{eq:mntz} to obtain an expression for $\rd  m^{(n)}_t(z)$ (\eqref{e:dmt} below). 
We will then apply the scaling from \eqref{eq:Lagc}, and use \eqref{eq:Deltmeq} to rewrite the martingale and drift terms via $\Delta^{(n)}_t(w) +\sqrt w$, plus error terms.

By It\^o's
formula and \eqref{e:Laguerre}, and noting that $\rd\langle \lambda^{(n)}_i, \lambda^{(n)}_i\rangle_t=\frac{4}{\beta}\lambda^{(n)}_i(t)\rd t$,
\begin{align*}
\rd\frac{1}{\lambda^{(n)}_i(t)-z}
&=-\frac{2}{\sqrt\beta}\frac{\sqrt{\lambda^{(n)}_i(t)}}{(\lambda^{(n)}_i(t)-z)^{2}}\,\rd B_i(t)
-\frac{1}{(\lambda^{(n)}_i(t)-z)^{2}}
\Bigg(m+\sum^n_{j\neq i}\frac{\lambda^{(n)}_i(t)+\lambda^{(n)}_j(t)}{\lambda^{(n)}_i(t)-\lambda^{(n)}_j(t)}\Bigg)\rd t\\
&\qquad +\frac{4}{\beta}\frac{\lambda^{(n)}_i(t)}{(\lambda^{(n)}_i(t)-z)^{3}}\rd t.
\end{align*}
Recall that $\partial^k_z
m^{(n)}_t(z)=k!\sum_{i=1}^n
\frac{1}{(\lambda^{(n)}_i(t)-z)^{k+1}}$ 
for each $k\in\bN$. 
Thus summing over
$i=1,\dots,n$ yields
\begin{equation}\label{eq:dm_pre3}
\rd m^{(n)}_t(z)= -\frac{2}{\sqrt\beta}\sum_{i=1}^n
\frac{\sqrt{\lambda^{(n)}_i(t)}}{(\lambda^{(n)}_i(t)-z)^{2}}\,\rd B_i(t)
- (m\partial_z
m^{(n)}_t(z)+\mathcal I_t)\rd t
+\frac{4}{\beta}\sum_{i=1}^n\frac{\lambda^{(n)}_i(t)}{(\lambda^{(n)}_i(t)-z)^{3}}\,\rd t,
\end{equation}
where
\[
\mathcal I_t:=\sum_{i\neq j}^n
\frac{\lambda^{(n)}_i(t)+\lambda^{(n)}_j(t)}{\lambda^{(n)}_i(t)-\lambda^{(n)}_j(t)}\cdot
\frac{1}{(\lambda^{(n)}_i(t)-z)^{2}}.
\]
We symmetrize $\mathcal I_t$ over pairs $(i,j)$:
\begin{align*}
\mathcal I_t
=&\sum_{i<j}^n \frac{\lambda^{(n)}_i(t)+\lambda^{(n)}_j(t)}{\lambda^{(n)}_i(t)-\lambda^{(n)}_j(t)}
\left(\frac{1}{(\lambda^{(n)}_i(t)-z)^2}-\frac{1}{(\lambda^{(n)}_j(t)-z)^2}\right) \\
=&-\sum_{i<j}^n\frac{\big(\lambda^{(n)}_i(t)+\lambda^{(n)}_j(t)\big)\big(\lambda^{(n)}_i(t)+\lambda^{(n)}_j(t)-2z\big)}
{(\lambda^{(n)}_i(t)-z)^2(\lambda^{(n)}_j(t)-z)^2}\\
=&-\sum_{i\neq j}^n \left( \frac{1}{(\lambda^{(n)}_i(t)-z)^2} +\frac{1}{(\lambda^{(n)}_i(t)-z)(\lambda^{(n)}_j(t)-z)}\right)\\
&-2z\sum_{i\neq j}^n \frac{1}{(\lambda^{(n)}_i(t)-z)(\lambda^{(n)}_j(t)-z)^2}.
\end{align*}
Thus we get
\[
\mathcal I_t
=-(n-1)\partial_z m^{(n)}_t(z)-\big((m^{(n)}_t(z))^2-\partial_z m^{(n)}_t(z)\big)
-2z\Big(m^{(n)}_t(z)\partial_z m^{(n)}_t(z)-\frac12\partial_z^2 m^{(n)}_t(z)\Big).
\]
This can also be written as
\begin{equation}\label{eq:I_id2}
-(m\,\partial_z m^{(n)}_t(z)+\mathcal I_t)
=\partial_z\!\Big(z\,(m^{(n)}_t(z))^2+(z-m+n)m^{(n)}_t(z)\Big)
-2\sum_{i=1}^n\frac{\lambda^{(n)}_i(t)}{(\lambda^{(n)}_i(t)-z)^3},
\end{equation}
where we used
\[
\sum_{i=1}^n\frac{\lambda^{(n)}_i(t)}{(\lambda^{(n)}_i(t)-z)^3}
=\sum_{i=1}^n\left(\frac{1}{(\lambda^{(n)}_i(t)-z)^2}+\frac{z}{(\lambda^{(n)}_i(t)-z)^3}\right)
=\partial_z m^{(n)}_t(z)+\frac{z}{2}\partial_z^2 m^{(n)}_t(z).
\]
Substituting \eqref{eq:I_id2} into \eqref{eq:dm_pre3} yields
\begin{multline}\label{e:dmt}
\rd  m^{(n)}_t(z)= -\frac{2}{\sqrt{\beta }}\sum_{i=1}^n\frac{\sqrt{\lambda^{(n)}_i(t)}{\rm d} B_i(t)}{(\la^{(n)}_i(t)-z)^2}\\ + \del_z\left( z (m^{(n)}_t)^2(z) +(z-m+n) m^{(n)}_t(z)\right)\rd t+\frac{4-2\beta}{\beta }\sum_{i=1}^n\frac{\lambda^{(n)}_i(t)\rd t}{(\lambda^{(n)}_i(t)-z)^3}.
\end{multline}
We next do the scaling substitution. Recall the edge parameters from \eqref{eq:Lagc}:
\[
        E=(\sqrt{m}+\sqrt{n})^2,\qquad
        \zeta=\frac12(\sqrt{m}+\sqrt{n})^{2/3}(mn)^{-1/3},\qquad
        \chi=(\sqrt{m}+\sqrt{n})^{4/3}(mn)^{-1/6},
\]
and the rescaled eigenvalues
\[    \widetilde \lambda^{(n)}_i(t):=\frac{\lambda^{(n)}_i(\zeta t)-E}{\chi},\qquad i=1,\dots,n.
\]

For \(w\in\mathbb C\setminus\mathbb R\) set
$z=E+\chi w$.
After the time change \(t\mapsto \zeta t\), we again denote by \(\{B_i(t)\}_{i\le n}\) independent standard Brownian motions
(absorbing the usual \(\sqrt{\zeta}\)-factor into the notation). With this convention,
\eqref{e:dmt} rewrites as
\begin{align}
\begin{split}\label{eq:dms2}
    \chi\,\rd & m^{(n)}_{\zeta t}(E+\chi w)
    = -\frac{\sqrt{2}}{\sqrt{\beta }(\sqrt m+\sqrt n)}\sum_{i=1}^n
    \frac{(E +\chi \widetilde \lambda^{(n)}_i(t))^{1/2}\,\rd B_i(t)}{(\widetilde \lambda^{(n)}_i(t)-w)^2}\\
    &\quad + \zeta\,\partial_w\Bigl( (E +\chi w)\,(m^{(n)}_{\zeta t})^2(E +\chi w)
      +(\chi w+2\sqrt n(\sqrt m+\sqrt n))\, m^{(n)}_{\zeta t}(E +\chi w)\Bigr)\,\rd t\\
    &\quad +\frac{2-\beta}{\beta E}\sum_{i=1}^{n}
      \frac{(E +\chi \widetilde \lambda^{(n)}_i(t))\,\rd t}{(\widetilde \lambda^{(n)}_i(t)-w)^3}.
\end{split}
\end{align}
Let \(M^{(n)}_t(w)\) denote the martingale term in \eqref{eq:dms2}:
\[
M^{(n)}_t(w):= -\frac{\sqrt{2}}{\sqrt{\beta }(\sqrt m+\sqrt n)}\sum_{i=1}^n
    \int_0^t \frac{(E +\chi \widetilde \lambda^{(n)}_i(s))^{1/2}}{(\widetilde \lambda^{(n)}_i(s)-w)^2}\,\rd B_i(s).
\]
In the following we show that the margintale term and the drift term satisfy the estimates as in \Cref{prop:Dteq}.

For any \(w,w'\in\mathbb C\setminus \bR\), the quadratic variation of the martingale $M_t^{(n)}$ is given by
\begin{align*}
\frac{\rd}{\rd t}\Big\langle  M^{(n)}(w),  M^{(n)}(w')\Big\rangle_t
&=\frac{2}{\beta(\sqrt{m}+\sqrt n)^2}\sum_{i=1}^n
\frac{E +\chi \widetilde \lambda^{(n)}_i(t)}{(\widetilde \lambda^{(n)}_i(t)-w)^2(\widetilde \lambda^{(n)}_i(t)-w')^2}\nonumber\\
&=\frac{2}{\beta}\sum_{i=1}^n
\frac{1}{(\widetilde \lambda^{(n)}_i(t)-w)^2(\widetilde \lambda^{(n)}_i(t)-w')^2}
+\widehat{\mathcal E}^{(n)}_t(w,w'),
\end{align*}
where the covariance error term is
\begin{equation}\label{eq:def-Ehat}
\widehat{\mathcal E}^{(n)}_t(w,w')
:=\frac{2}{\beta(\sqrt{m}+\sqrt n)^2}\sum_{i=1}^n
\frac{\chi \widetilde \lambda^{(n)}_i(t)}{(\widetilde \lambda^{(n)}_i(t)-w)^2(\widetilde \lambda^{(n)}_i(t)-w')^2}.
\end{equation}

\begin{lem}[Bound on the covariance error]\label{lem:Ehat-bound}
For \(w,w'\in\mathbb C\setminus \bR\),
\begin{equation}\label{eq:Ehat-bound}
\big|\widehat{\mathcal E}^{(n)}_t(w,w')\big|
\;\lesssim\;
\frac{\chi}{(\sqrt m+\sqrt n)^2}\,
\frac{|w'|}{\Im[w]\Im[w']^2}\,
\Im [\Delta^{(n)}_t(w)+\sqrt w].
\end{equation}
In particular, since \(\chi/(\sqrt m+\sqrt n)^2=\chi/E\asymp n^{-2/3}\) under the soft-edge scaling,
\(\widehat{\mathcal E}^{(n)}_t(w,w')\) is of order \(n^{-2/3}\) (up to resolvent factors).
Moreover, the same type of bound holds for \(\partial_w\partial_{w'}\widehat{\mathcal E}^{(n)}_t(w,w')\),
with extra factors of \(|\Im[w]|^{-1}\) and \(|\Im[w']|^{-1}\) coming from differentiating the resolvents.
\end{lem}

\begin{proof}
We have
\begin{multline}\label{eq:sum-square}
\sum_{i=1}^n\frac{1}{|\widetilde\lambda_i^{(n)}(t)-w|^2}
=\frac{1}{\Im[w]}\sum_{i=1}^n \frac{\Im[\widetilde\lambda_i^{(n)}(t)-\overline{w}]}{(\widetilde\lambda_i^{(n)}(t)-w)(\widetilde\lambda_i^{(n)}(t)-\overline{w})}
\\=
\frac{1}{\Im[w]}\sum_{i=1}^n \Im\bigg[\frac{1}{\widetilde\lambda_i^{(n)}(t)-w}\bigg]
=\frac{\Im[ \Delta^{(n)}_t(w)+\sqrt w]}{\Im[w]},
\end{multline}
where the last equality is from \eqref{eq:Deltmeq}.
And then
\begin{equation}\label{eq:sum-cube}
\sum_{i=1}^n\frac{1}{|\widetilde\lambda_i^{(n)}(t)-w|^3}
\le \frac{1}{|\Im[w]|}\sum_{i=1}^n\frac{1}{|\widetilde\lambda_i^{(n)}(t)-w|^2}
=\frac{|\Im [\Delta^{(n)}_t(w)+\sqrt w]|}{\Im[w]^2}.
\end{equation}
Using \(|\widetilde\lambda^{(n)}_i(t)|\le |\widetilde\lambda^{(n)}_i(t)-w'|+|w'|\) and
\(|\widetilde\lambda^{(n)}_i(t)-w'|\ge |\Im[w']|\), we get
\begin{align*}
\frac{|\widetilde\lambda^{(n)}_i(t)|}{|\widetilde\lambda^{(n)}_i(t)-w|^2\,|\widetilde\lambda^{(n)}_i(t)-w'|^2}
&\le
\frac{1}{|\widetilde\lambda^{(n)}_i(t)-w|^2\,|\widetilde\lambda^{(n)}_i(t)-w'|}
+\frac{|w'|}{|\widetilde\lambda^{(n)}_i(t)-w|^2\,|\widetilde\lambda^{(n)}_i(t)-w'|^2}\\
&\le
\frac{1}{|\Im[w']|}\frac{1}{|\widetilde\lambda^{(n)}_i-w|^2}
+\frac{|w'|}{\Im[w']^2}\frac{1}{|\widetilde\lambda^{(n)}_i-w|^2}.
\end{align*}
Summing over $i$ and applying \eqref{eq:sum-square} yields
\[
\sum_{i=1}^n \frac{|\widetilde\lambda^{(n)}_i(t)|}{|\widetilde\lambda^{(n)}_i(t)-w|^2\,|\widetilde\lambda^{(n)}_i(t)-w'|^2}
\lesssim \frac{|w'|}{\Im [w']^2}\sum_{i=1}^n\frac{1}{|\widetilde\lambda^{(n)}_i-w|^2}
= \frac{|w'|}{\Im [w']^2}\frac{\Im [\Delta^{(n)}_t(w)+\sqrt w]}{\Im[w]}.
\]
Multiplying by the prefactor \(\chi/(\sqrt m+\sqrt n)^2\) from \eqref{eq:def-Ehat} gives \eqref{eq:Ehat-bound}.
Differentiating in \(w,w'\) only increases the power of \(|\widetilde\lambda^{(n)}_i-w|^{-1}\), \(|\widetilde\lambda^{(n)}_i-w'|^{-1}\),
hence produces additional \(|\Im [w]|^{-1},|\Im [w']|^{-1}\) factors in the same way.
\end{proof}

\begin{lem}[Bound on the drift term]
For any $w\in \bC\setminus \bR$, the drift term in \eqref{eq:dms2}
\begin{align}\begin{split}\label{e:drift}
    & \zeta\,\partial_w\Bigl( (E +\chi w)\,(m^{(n)}_{\zeta t})^2(E +\chi w)
      +(\chi w+2\sqrt n(\sqrt m+\sqrt n))\, m^{(n)}_{\zeta t}(E +\chi w)\Bigr)\,\rd t\\
    &\quad +\frac{2-\beta}{\beta E}\sum_{i=1}^{n}
      \frac{(E +\chi \widetilde \lambda^{(n)}_i(t))\,\rd t}{(\widetilde \lambda^{(n)}_i(t)-w)^3}.
\end{split}\end{align}
equals
\[
\frac12\,\partial_w\bigl(\Delta^{(n)}_t(w)+\sqrt w\bigr)^2-\frac12
+
\frac{2-\beta}{\beta}\sum_{i=1}^n\frac{1}{(\widetilde\lambda^{(n)}_i(t)-w)^3}+\cE^{(n)}_t(w),
\]
where 
\[
|\cE^{(n)}_t(w)|,\; |\partial_w\cE^{(n)}_t(w)| \lesssim n^{-1/3}(1+|\Delta^{(n)}_t(w)+\sqrt w|)^2.
\]
\end{lem}
\begin{proof}
We now rewrite the drift terms in \eqref{eq:dms2} in terms of \([\Delta^{(n)}_t(w)+\sqrt w]\).
First note the constant identities (directly from \eqref{eq:Lagc}):
\begin{equation}\label{eq:const-identities}
\frac{\zeta E}{\chi^2}=\frac12,
\quad
\zeta\Bigl(\frac{\chi n}{E}-\frac{\chi\sqrt n}{\sqrt m+\sqrt n}\Bigr)=-\frac12,\quad \frac{\zeta}{\chi}=\frac{1}{2(\sqrt{m}+\sqrt{n})^{2/3}(mn)^{1/6}}
\end{equation}

We decompose the drift term in \eqref{e:drift} as $\mathrm{I}+\mathrm{II}+\mathrm{III}$.
Using \eqref{eq:Deltmeq}, we rewrite \eqref{e:drift} in terms of
$\Delta^{(n)}_t(w)+\sqrt w$.
We then define $\mathrm{I}$ to collect the terms that are quadratic in $\Delta^{(n)}_t(w)+\sqrt w$ together with the constant terms,
$\mathrm{II}$ to collect the terms that are linear in $\Delta^{(n)}_t(w)+\sqrt w$,
and $\mathrm{III}$ to be the final term in \eqref{e:drift}.

\paragraph{Term \(\mathrm{I}\).}
Define
\begin{align}
\mathrm{I}
&:=\zeta\,\partial_w\left(
\frac{E+\chi w}{\chi^2}\,(\Delta^{(n)}_t(w)+\sqrt w)^2
+\Bigl(\frac{\chi\sqrt n}{\sqrt m+\sqrt n}\Bigr)^2\frac{w}{\chi}
-\frac{\chi\sqrt n}{\sqrt m+\sqrt n}\,w
\right)\nonumber\\
&=\zeta\,\partial_w\left(
\frac{E+\chi w}{\chi^2}\,(\Delta^{(n)}_t(w)+\sqrt w)^2\right)
+\zeta\left(\frac{\chi n}{E}-\frac{\chi\sqrt n}{\sqrt m+\sqrt n}\right)\nonumber\\
&=\frac12\,\partial_w\Bigl(\Delta^{(n)}_t(w)+\sqrt w\Bigr)^2-\frac12+\mathcal E^{(n),\mathrm{I}}_t(w),
\label{eq:I-main-plus-error}
\end{align}
where we used $E=(\sqrt{m}+\sqrt{n})^2$ from \eqref{eq:Lagc}, \eqref{eq:const-identities}, and  set
\begin{equation}\label{eq:def-EI}
\mathcal E^{(n),\mathrm{I}}_t(w)
:=\partial_w\left(\frac{\zeta w}{\chi}\,(\Delta^{(n)}_t(w)+\sqrt w)^2\right).
\end{equation}
By the product rule,
\begin{equation}\label{eq:EI-bound}
|\mathcal E^{(n),\mathrm{I}}_t(w)|
\lesssim \frac{\zeta}{\chi}\left(|\Delta^{(n)}_t(w)+\sqrt w|^2 + |w(\Delta^{(n)}_t(w)+\sqrt w)\partial_w(\Delta^{(n)}_t(w)+\sqrt w)|\right).
\end{equation}
Moreover, since
\[
|\partial_w (\Delta^{(n)}_t(w)+\sqrt w)|=\Bigg|\sum_{i=1}^n\frac{1}{(\widetilde\lambda^{(n)}_i-w)^2}\Bigg|
\le \sum_{i=1}^n\frac{1}{|\widetilde\lambda^{(n)}_i-w|^2}
=\frac{\Im[ \Delta^{(n)}_t(w)+\sqrt w]}{\Im[w]},
\]
we obtain the convenient bound
\begin{align}\begin{split}\label{eq:EI-bound-im}
|\mathcal E^{(n),\mathrm{I}}_t(w)|
&\lesssim \frac{\zeta}{\chi}\left(|\Delta^{(n)}_t(w)+\sqrt w|^2
+ |w(\Delta^{(n)}_t(w)+\sqrt w)|\,\frac{\Im [\Delta^{(n)}_t(w)+\sqrt w]}{\Im[w]}\right)\\
&\lesssim \frac{1}{n^{2/3}}\left(|\Delta^{(n)}_t(w)+\sqrt w|^2+|w(\Delta^{(n)}_t(w)+\sqrt w)|\frac{\Im[\Delta^{(n)}_t(w)+\sqrt w]}{\Im[w]}\right),
\end{split}\end{align}
and similarly for \(|\partial_w \mathcal E^{(n),\mathrm{I}}_t(w)|\) (with additional \(|\Im[w]|^{-1}\) factors).

\paragraph{Term \(\mathrm{II}\).}
Define
\begin{align}
\mathrm{II}
&:=\zeta\,\partial_w\left(
\left(-\frac{2\sqrt n (E +\chi w)}{\chi(\sqrt m+\sqrt n)}
+\frac{2\sqrt n(\sqrt m+\sqrt n)}{\chi}+w\right)(\Delta^{(n)}_t(w)+\sqrt w)\right)\nonumber\\
&=\zeta\,\partial_w\left(\frac{\sqrt m-\sqrt n}{\sqrt m+\sqrt n}\,w\,(\Delta^{(n)}_t(w)+\sqrt w)\right)
=: \mathcal E^{(n),\mathrm{II}}_t(w),
\label{eq:def-EII}
\end{align}
where we used \(E=(\sqrt m+\sqrt n)^2\) to cancel the constant terms.
Thus
\begin{align}\begin{split}\label{eq:EII-bound}
|\mathcal E^{(n),\mathrm{II}}_t(w)|
&\lesssim \zeta\left(|\Delta^{(n)}_t(w)+\sqrt w| + |w\partial_w (\Delta^{(n)}_t(w)+\sqrt w)|\right)\\
&\lesssim \zeta\left(|\Delta^{(n)}_t(w)+\sqrt w|
+|w|\frac{\Im [\Delta^{(n)}_t(w)+\sqrt w]}{\Im[w]}\right)\\
&\lesssim\frac{|\Delta^{(n)}_t(w)+\sqrt w|}{n^{1/3}}+\frac{|w|\Im[\Delta^{(n)}_t(w)+\sqrt w]}{n^{1/3}\Im[w]}
\end{split}\end{align}
and again the same type of bound holds for \(|\partial_w\mathcal E^{(n),\mathrm{II}}_t(w)|\).

\paragraph{Term \(\mathrm{III}\).}
Finally,
\begin{align}
\mathrm{III}
&:=\frac{2-\beta}{\beta E}\sum_{i=1}^n\frac{(E +\chi \widetilde \lambda^{(n)}_i(t))\,\rd t}{(\widetilde \lambda^{(n)}_i(t)-w)^3}\nonumber\\
&=\frac{2-\beta}{\beta}\sum_{i=1}^n\frac{\rd t}{(\widetilde \lambda^{(n)}_i(t)-w)^3}
+\mathcal E^{(n),\mathrm{III}}_t(w)\,\rd t,
\label{eq:def-III-split}
\end{align}
where
\[
\mathcal E^{(n),\mathrm{III}}_t(w)
:=\frac{2-\beta}{\beta E}\sum_{i=1}^{n}\frac{\chi \widetilde \lambda^{(n)}_i(t)}{(\widetilde \lambda^{(n)}_i(t)-w)^3}.
\]
Using \(|\widetilde\lambda^{(n)}_i|\le |\widetilde\lambda^{(n)}_i-w|+|w|\) and \eqref{eq:sum-cube},
\begin{align}
|\mathcal E^{(n),\mathrm{III}}_t(w)|
&\lesssim \frac{\chi}{E}\sum_{i=1}^n\frac{|\widetilde\lambda^{(n)}_i-w|+|w|}{|\widetilde\lambda^{(n)}_i-w|^3}
\lesssim \frac{\chi}{E}\left(\sum_{i=1}^n\frac{1}{|\widetilde\lambda^{(n)}_i-w|^2}
+\frac{|w|}{|\Im [w]|}\sum_{i=1}^n\frac{1}{|\widetilde\lambda^{(n)}_i-w|^2}\right)\nonumber\\
&\lesssim \frac{\chi}{E}\,\frac{|w|}{\Im [w]^2}\,|\Im [\Delta^{(n)}_t(w)+\sqrt w]|\lesssim \frac{|\Im[\Delta^{(n)}_t(w)+\sqrt w]||w|}{n^{2/3}\Im[w]^2 }.
\label{eq:EIII-bound}
\end{align}
Similarly, \(|\partial_w\mathcal E^{(n),\mathrm{III}}_t(w)|\) satisfies the same bound up to an extra \(|\Im [w]|^{-1}\) factor.

\paragraph{Collecting the drift errors.}
Define
\[
\mathcal E^{(n)}_t(w):=\mathcal E^{(n),\mathrm{I}}_t(w)+\mathcal E^{(n),\mathrm{II}}_t(w)+\mathcal E^{(n),\mathrm{III}}_t(w).
\]
Then \eqref{eq:I-main-plus-error}--\eqref{eq:def-III-split} give a drift decomposition of the rescaled dynamics into
the main term
\[
\frac12\,\partial_w\bigl(\Delta^{(n)}_t(w)+\sqrt w\bigr)^2-\frac12
+\frac{2-\beta}{\beta}\sum_{i=1}^n\frac{1}{(\widetilde\lambda^{(n)}_i(t)-w)^3},
\]
plus the error \(\mathcal E^{(n)}_t(w)\).
Moreover, the bounds \eqref{eq:EI-bound-im}, \eqref{eq:EII-bound}, \eqref{eq:EIII-bound}
(and their \(w\)-derivatives) yield the desired control of
\(|\mathcal E^{(n)}_t(w)|\) and \(|\partial_w\mathcal E^{(n)}_t(w)|\) in terms of
\(|\Delta^{(n)}_t(w)+\sqrt w|\) and \(|\Im [\Delta^{(n)}_t(w)+\sqrt w]|\).
\end{proof}

\noindent\textbf{Jacobi.}
We follow the same steps as in the Laguerre case.
By It\^o's formula and \eqref{e:Jacobi}, we have
\[
\rd\frac{1}{\lambda^{(n)}_i(t)-z}
= -\frac{1}{(\lambda^{(n)}_i(t)-z)^2}\,\rd\lambda^{(n)}_i(t)
+
\frac{4}{\beta}\cdot\frac{\lambda^{(n)}_i(t)\bigl(1-\lambda^{(n)}_i(t)}{(\lambda^{(n)}_i(t)-z)^3}\bigr)\,\rd t.
\]
Again using \eqref{e:Jacobi} and 
summing over $i=1,\dots,n$ yields
\begin{equation}\label{eq:pre_jacobi}
\begin{aligned}
\rd m^{(n)}_t(z)
&= -\frac{2}{\sqrt{\beta }}\sum_{i=1}^n
\frac{\sqrt{\lambda^{(n)}_i(t)(1-\lambda^{(n)}_i(t))}}{(\lambda^{(n)}_i(t)-z)^2}\,\rd B_i(t)\\
&\quad
-\sum_{i=1}^n\frac{p-m\lambda^{(n)}_i(t)}{(\lambda^{(n)}_i(t)-z)^2}\,\rd t
-\mathcal I_t(z)\,\rd t
+\frac{4}{\beta}\sum_{i=1}^n\frac{\lambda^{(n)}_i(t)(1-\lambda^{(n)}_i(t))}{(\lambda^{(n)}_i(t)-z)^3}\,\rd t,
\end{aligned}
\end{equation}
where
\[
\mathcal I_t(z):=
\sum_{i\neq j}^n
\frac{\lambda^{(n)}_i(t)(1-\lambda^{(n)}_i(t))+\lambda^{(n)}_j(t)(1-\lambda^{(n)}_j(t))}
{\lambda^{(n)}_i(t)-\lambda^{(n)}_j(t)}\,
\frac{1}{(\lambda^{(n)}_i(t)-z)^2}.
\]
We can write
\begin{equation}\label{eq:lin_drift}
-\sum_{i=1}^n\frac{p-m\lambda^{(n)}_i(t)}{(\lambda^{(n)}_i(t)-z)^2}
= -p\,\partial_z m^{(n)}_t(z)+m\,m^{(n)}_t(z)+mz\,\partial_z m^{(n)}_t(z).
\end{equation}

We next consider $\mathcal I_t(z)$. Set $q(x):=x(1-x)$. Symmetrizing over pairs $(i,j)$ gives
\begin{align}
\mathcal I_t(z)
&=\sum_{i<j}^n\frac{q(\lambda^{(n)}_i(t))+q(\lambda^{(n)}_j(t))}{\lambda^{(n)}_i(t)-\lambda^{(n)}_j(t)}
\left(\frac{1}{(\lambda^{(n)}_i(t)-z)^2}-\frac{1}{(\lambda^{(n)}_j(t)-z)^2}\right)\notag\\
&=-\sum_{i<j}^n
\frac{\bigl(q(\lambda^{(n)}_i(t))+q(\lambda^{(n)}_j(t))\bigr)\bigl(\lambda^{(n)}_i(t)+\lambda^{(n)}_j(t)-2z\bigr)}
{(\lambda^{(n)}_i(t)-z)^2(\lambda^{(n)}_j(t)-z)^2},\label{eq:symmI}
\end{align}
using
\[
\frac{1}{(\lambda^{(n)}_i-z)^2}-\frac{1}{(\lambda^{(n)}_j-z)^2}
=-(\lambda^{(n)}_i-\lambda^{(n)}_j)\frac{\lambda^{(n)}_i+\lambda^{(n)}_j-2z}{(\lambda^{(n)}_i-z)^2(\lambda^{(n)}_j-z)^2}.
\]
Next use the decomposition (valid for any $\lambda$)
\[
q(\lambda)=\lambda(1-\lambda)=z(1-z)+(1-2z)(\lambda-z)-(\lambda-z)^2,
\qquad
\lambda^{(n)}_i+\lambda^{(n)}_j-2z=(\lambda^{(n)}_i-z)+(\lambda^{(n)}_j-z).
\]
Plugging these into \eqref{eq:symmI} and expanding reduces $\mathcal I_t(z)$ to linear
combinations of the standard symmetric sums
\[
\sum_{i<j}^n\frac{1}{(\lambda^{(n)}_i-z)(\lambda^{(n)}_j-z)},\quad
\sum_{i<j}^n\Big(\frac{1}{(\lambda^{(n)}_i-z)^2}+\frac{1}{(\lambda^{(n)}_j-z)^2}\Big),\quad
\sum_{i\neq j}^n\frac{1}{(\lambda^{(n)}_i-z)^2(\lambda^{(n)}_j-z)}.
\]
Using the identities
\[
2\sum_{i<j}^n\frac{1}{(\lambda^{(n)}_i-z)(\lambda^{(n)}_j-z)}
=\bigl(m^{(n)}_t(z)\bigr)^2-\partial_z m^{(n)}_t(z),
\]
\[
\sum_{i<j}^n\Big(\frac{1}{(\lambda^{(n)}_i-z)^2}+\frac{1}{(\lambda^{(n)}_j-z)^2}\Big)
=(n-1)\partial_z m^{(n)}_t(z),
\]
\[
\sum_{i\neq j}^n\frac{1}{(\lambda^{(n)}_i-z)^2(\lambda^{(n)}_j-z)}
=m^{(n)}_t(z)\,\partial_z m^{(n)}_t(z)-\frac12\,\partial_z^2 m^{(n)}_t(z),
\]
a short bookkeeping computation yields
\begin{equation}\label{eq:I_identity_jacobi}
-\mathcal I_t(z)
=\partial_z\!\Big(z(1-z)\bigl(m^{(n)}_t(z)\bigr)^2+(-2nz+n)\,m^{(n)}_t(z)\Big)
-2\sum_{i=1}^n\frac{\lambda^{(n)}_i(t)\bigl(1-\lambda^{(n)}_i(t)\bigr)}{(\lambda^{(n)}_i(t)-z)^3}.
\end{equation}

Finally, combining \eqref{eq:pre_jacobi}, \eqref{eq:lin_drift}, and \eqref{eq:I_identity_jacobi},
 we get
\begin{multline}\label{e:dmtj}
\rd  m^{(n)}_t(z)= -\frac{2}{\sqrt{\beta }}\sum_{i=1}^n \frac{\sqrt{\lambda^{(n)}_i(t)(1-\lambda^{(n)}_i(t))}{\rm d} B_i(t)}{(\la^{(n)}_i(t)-z)^2} \\
+ \del_z\left( z(1-z) (m^{(n)}_t)^2(z) +(mz-2nz+n-p) m^{(n)}_t(z)\right)\rd t
+\frac{4-2\beta}{\beta }\sum_{i=1}^n\frac{\lambda^{(n)}_i(t)(1-\lambda^{(n)}_i(t))\rd t}{(\lambda^{(n)}_i(t)-z)^3}.
\end{multline}
The remaining arguments follow from computations analogous to the Laguerre case, and we omit the details.
\end{proof}

\subsection{Tightness}
This subsection is for the remaining tasks (2) and (3), i.e., tightness and verifying assumptions. They can be summarized as the following statement
\begin{prop}\label{p:tight}
For any sequence of integers $\to\infty$, we can take a subsequence $n_1<n_2<\cdots$, such that as $k\to\infty$, $\{\Delta^{(n_k)}_{t}(w)+\sqrt{w}\}_{w\in\bH, t\in \bR}$ and $\{\widetilde \lambda^{(n_k)}_i(t)\}_{i\in\bN, t\in\bN}$ converge jointly under the uniform in compact topology.
Moreover, the limit of $\{\Delta^{(n_k)}_{t}(w)+\sqrt{w}\}_{w\in\bH, t\in \bR}$ is particle-generated, satisfying \Cref{a:Infinite} and \Cref{a:SDE}, with the limit of $\{\widetilde \lambda^{(n_k)}_i(t)\}_{i\in\bN, t\in\bN}$ being the poles.
\end{prop}
This proposition, together with \Cref{t:characterize}, implies \Cref{t:converge_Airy_details}.

The proof of this tightness result has two components.
First, since all the three processes in \Cref{t:converge_Airy_details} are stationary, a single time slice is described by the corresponding $\beta$-ensemble. The tightness at one time is established, thanks to the results from \cite{bourgade2022optimal}, which also verify \Cref{a:Infinite}. 
Second, we check the tightness of the Stieltjes transform over a time interval, and verify \Cref{a:SDE}. These are achieved using the SDE \eqref{e:Dteq}.

\subsubsection{Tightness at one time}\label{s:one_time_tight}
We now show that for fixed time $t$, $\Delta_t^{(n)}(w)+\sqrt{w}$ converges as $n\to\infty$ to a particle-generated Nevanlinna function along subsequences, and verify \Cref{a:Infinite}.
We note that along this procedure we also get the tightness of $\{\wt\la_i^{(n)}(t)\}_{i\in\bN}$ at fixed $t$, which has already been proven to converge to the eigenvalues of the $\beta$ stochastic Airy operator, in e.g., \cite{MR2813333,MR3433632,holcomb2012edge}.
\begin{prop}\label{p:one_time_tight}
For any fixed $t\in\bR$, and any sequence of integers $\to\infty$, we can take a subsequence $n_1<n_2<\cdots$, such that $\{\Delta^{(n_k)}_{t}(w)+\sqrt{w}\}_{w\in\bH}$ (under the uniform in compact topology) and $\{\widetilde \lambda^{(n_k)}_i(t)\}_{i\in\bN}$ converge jointly as $k\to\infty$.
Besides, the limit of $\{\Delta^{(n_k)}_{t}(w)+\sqrt{w}\}_{w\in\bH}$ is $(\fd, C_*)$-Airy-like, with the limit of $\{\widetilde \lambda^{(n_k)}_i(t)\}_{i\in\bN}$ being its poles. Here $\fd$ is a universal constant, and $C_*$ is random with law independent of the subsequence or $t$.
\end{prop}

The following lemma states that $\Delta^{(n)}_t$ satisfies similar statements as being Airy-like.
\begin{lem}\label{l:onepointDelta}
    For any $\varepsilon>0$, there exists a large constant $C>0$, such that for any $t\in\bR$, the following holds with probability at least $1-\varepsilon$: 
    \begin{itemize}
        \item the particles are bounded above, i.e., $\widetilde \lambda^{(n)}_1(t)\leq C$;
        \item for any $w=a+\ri b$ with $|w|\leq n^{1/6}$ and $b\geq C\sqrt{a\vee 0+1}$, it holds
    \begin{align}\label{e:Dtw}
        |\Delta^{(n)}_t(w)|\leq \frac{C\Im[\sqrt{w}]^{1/2}}{\Im[w]}.
    \end{align}
    \end{itemize}
\end{lem}
Before providing its proof, we derive \Cref{p:one_time_tight} from it.
\begin{proof}[Proof of \Cref{p:one_time_tight}]
    First we show tightness of $\Delta^{(n)}_{t}(w)+\sqrt{w}$.
For this, note that $\Delta^{(n)}_{t}(w)+\sqrt{w}$ is holomorphic in $\bH$, so
it suffices to show that for any fixed compact subset $\cK\subset \bH$,  
    $|\Delta^{(n)}_t(w)+\sqrt w|$ is uniformly bounded in $\cK$. 
    
    Take any $\epsilon>0$. Let $\delta>0$ small enough depending on $\epsilon$, and that $\cK\subset \{w=a+b\ri: |a|\leq 1/\delta, \delta\leq b\leq 1/\delta\}$. Then \eqref{e:Dtw} implies that (with probability at least $1-\varepsilon$),
    \begin{align}\label{e:top_line}
        |\Delta_t^{(n)}(a+\ri/\delta)+\sqrt{a+\ri/\delta}|\leq 
        |a+\ri/\delta|^{1/4}+|a+\ri/\delta|^{1/2}< \frac{2}{\delta^{1/2}},
    \end{align}
    for $|a|\leq 1/\delta$. Here in the first inequality, we bound the constant $C$ from \eqref{e:Dtw} by $1/\delta$, since $\delta>0$ is small enough depending on $\epsilon$.
    Using \Cref{l:Yproperty}, we have
    \[
    -\del_w\log |\Delta^{(n)}_t(w)+\sqrt w|\leq \frac{1}{\Im[w]}.
    \]
    By integrating this expression from $w=a+\ri/\delta$ to $w=a+b\ri$, we conclude that
    \begin{align}\label{e:uniformbound}
     |\Delta^{(n)}_t(a+b\ri)+\sqrt{a+b\ri}|
     \leq \frac{1}{b\delta}|\Delta_t^{(n)}(a+\ri/\delta)+\sqrt{a+\ri/\delta}|\leq \frac{2}{b\delta^{3/2}}. 
    \end{align}
    This implies that with probability at least $1-\varepsilon$, 
    $|\Delta^{(n)}_t(w)+\sqrt{w}|$ is uniformly bounded in $\cK$.

Now by taking a subsequence, we have the convergence of $\Delta^{(n_k)}_t(w)+\sqrt{w}$.
By \Cref{lem:StoMconv} and \Cref{lem:measuretoparticle}, the limit is particle-generated.
By \Cref{l:onepointDelta}, and taking a further subsequence, the limit is $(\fd, C_*)$-Airy-like as asserted, and in particular, has infinitely many poles.
Then by \Cref{lem:StoMconv} and \Cref{lem:measuretoparticle} again, we have that for each $i\in\bN$, $\widetilde \lambda^{(n_k)}_i(t)$ converges as $k\to\infty$, and these give all the poles.
\end{proof}

In the rest of this subsection, we derive \Cref{l:onepointDelta} from an optimal local law for $\beta$-ensembles, proved in \cite{bourgade2022optimal}. 

\begin{thm}{\cite[Corollary 1.6, Proposition 2.5 and Proposition 3.5]{bourgade2022optimal}}\label{t:Dtmoment}
We consider $\beta$-ensemble $x_1\ge \cdots \ge x_n$ whose distribution density is given by \eqref{e:beta}, with potential $V(x)$ satisfying \Cref{a:Vasump}. Denote
\[
    s_n(z)=\frac{1}{n}\sum_{i=1}^n\frac{1}{x_i-z}, \quad m_V(z)=\int\frac{\rd \mu_V(x)}{x-z}.
\]
There exist $\eta>0$, and $C>0$ such that for any $q\geq 1$ and $n\geq 1$, the following holds
\begin{enumerate}
\item For any $0\leq u\leq n^{2/3}$, it holds that
\[
    \bP(\exists k\in \llbracket 1,n\rrbracket,  x_k\not\in [A-u n^{-2/3}, B+u n^{-2/3}])\leq Ce^{-u^{3/4}/C}.
\]
\item For any $z\in \bH$, with $A-\eta\leq \Re[z]\leq B+\eta$ and $0<\Im[z]\leq \eta$,
\[
    \bE[|s_n(z)-m_V(z)|^q]\leq \frac{(Cq)^{2q}}{(n\Im[z])^q}. 
\]
For  $z\in \bH$ with $\kappa=|\Re[z]-A|\wedge |\Re[z]-B|$, $0<\Im[z]\leq \eta$, and $Cq^{1/2}/(n\sqrt{\kappa}) \le \Im[z] \le \kappa$, 
\[
    \bE[|s_n(z)-m_V(z)|^{2q}]\leq \frac{(Cq)^{2q}}{(n\Im[z])^{4q}\kappa^{q}} + \frac{(Cq)^q}{n^{2q}(\kappa \Im[z])^q } +\frac{(Cq)^{2q}}{ n^{2q}\kappa^q}
\]
\end{enumerate}
\end{thm}

\begin{remark}\label{rem:extbour}
According to the proofs in \cite{bourgade2022optimal} (see \cite[Section 2.1 and Remark 2.4]{bourgade2022optimal}), for \Cref{t:Dtmoment} to hold, \Cref{a:Vasump} on the potential $V(x)$ can be relaxed to the following statements:
\begin{itemize}
    \item $V(x)$ is analytic in $\Omega$, which is a simply connected open subset of $\bC$, containing an interval $[A,B]$.
    \item There is a unique $\mu_V$ given by \eqref{e:variational_principle}, where the minimization is taken over all probability measures on $\bR\cap\Omega$.
    Moreover, $\mu_V$ has density $\varrho_V$ whose support is $[A, B]$, and has square root singularities at $A$ and $B$ (with coefficients $R_A/\pi>0$ and $R/\pi=R_B/\pi>0$).
    \item  \eqref{e:rhoV} and \eqref{e:mV} hold, and the function $r(z)$ there is analytic and nonzero in $\Omega$.
\end{itemize}
Besides, \Cref{t:Dtmoment} remains valid if $V=V_n$ depends on $n$, provided that the above conditions are satisfied quantitatively and uniformly for each $V_n$. Specifically, there exists a constant $C>0$ (independent of $n$) such that
\begin{itemize}
    \item $|A|, |B|< C, |B-A|> C^{-1}$;
    \item for any $z\in \Omega$, $C^{-1}< |r(z)|< C$ (note that this implies uniform bounds of $R_A$ and $R_B$);
    \item $\Omega$ (which may depend on $n$) contains the $C^{-1}$ neighborhood of $[A, B]$;
    \item the constants in $\OO$ in \eqref{e:rhoV} and \eqref{e:mV} are $<C$.
\end{itemize}
Under these assumptions, the original proof in \cite{bourgade2022optimal}  applies without modification, ensuring that \Cref{t:Dtmoment} remains valid.

With these extensions, we can now address both the Laguerre and Jacobi cases.
More precisely, in the Laguerre case, (upon a rescaling of particles to constant order) we have 
\[
V(x) =V_n(x) = -\frac{m-n+1-2/\beta}{n} \log(x) +x,
\]
and
\[A=(\sqrt{(m+1-2/\beta)/n}-1)^2, \quad B=(\sqrt{(m+1-2/\beta)/n}+1)^2,
\]
\[\varrho_V(x)=\varrho_{V_n}(x) = \frac{\sqrt{4(m+1-2/\beta)n - (xn-(m+n+1-2/\beta))^2}}{2\pi xn}.\]
See e.g., \cite{dumitriu2006global}. Note that this $\varrho_V$ is approximately a rescaling of $\varrho^{(n)}_{\rm mp}$ from \eqref{eq:defmpl}.
Moreover, we have 
\[
r(z)=\frac{2m_V(z)+V'(z)}{2\sqrt{(z-A)(z-B)}} =\frac{1}{2z}.
\]
In the Jacobi case, we have
\[
V(x)=V_n(x)=-\frac{p-n+1-2/\beta}{n}\log(x)-\frac{q-n+1-2/\beta}{n}\log(1-x), 
\]
and
\[
A=\left(\frac{\sqrt{(p+1-2/\beta)(m-n+2-4/\beta)}-\sqrt{(q+1-2/\beta)n}}{m+2-4/\beta}\right)^2,\] 
\[B=\left(\frac{\sqrt{(p+1-2/\beta)(m-n+2-4/\beta)}-\sqrt{(q+1-2/\beta)n}}{m+2-4/\beta}\right)^2,
\]
\[
\varrho_V(x)=\varrho_{V_n}(x) = \frac{(m+2-4/\beta)\sqrt{(x-A)(B-x)}}{2\pi x(1-x)n}.
\]
See e.g., \cite{dumitriu2012global}. Note that this $\varrho_V$ is approximately a rescaling of $\varrho^{(n)}_{\rm Ja}$ from \eqref{eq:defJar}.
Moreover, we have
\[
r(z)=\frac{2m_V(z)+V'(z)}{2\sqrt{(z-A)(z-B)}} =\frac{(m+2-4/\beta)}{2z(1-z)n}.
\]
In both cases, it is evident that the above conditions are satisfied, under the limiting scheme specified in \eqref{eq:Lagc} and \eqref{eq:Jacc}.
\end{remark}

\Cref{t:Dtmoment} can be translated into the following estimates of $\widetilde \lambda^{(n)}_1(t)$ and $\Delta^{(n)}_t(w)$.

\begin{lem}\label{l:Dtmoment2}
There exists $C>0$ such that (for any $n\in \bN$ and $t\in\bR$) the following holds:
\begin{enumerate}
\item For any $0<u=\OO(n^{2/3})$, it holds that \begin{align}\label{e:la1bound}
    \bP(\widetilde \lambda^{(n)}_1(t)> u)\leq Ce^{-u^{3/4}/C}.
\end{align}
\item 
For any $q\geq 1$ and $w=a+\ri b$ with $|w|=\OO(n^{2/3})$,
\begin{align}\label{e:moment2}
    \bE[|\Delta^{(n)}_t(w)+\OO(n^{-1/3}|w|)|^q]\leq \frac{(Cq)^{2q}}{b^q}, 
\end{align}
 and when $Cq^{1/2}/\sqrt{a} \le b \le a$, 
\begin{align}\label{e:moment3}
    \bE[|\Delta^{(n)}_t(w)+\OO(n^{-1/3}|w|)|^{2q}]\leq \frac{(Cq)^{2q}}{b^{4q}a^{q}} + \frac{(Cq)^q}{(a b)^q } +\frac{(Cq)^{2q}}{ n^{2q/3}a^q}.
\end{align}
\end{enumerate}
\end{lem}
\begin{proof} We prove the statement for the DBM with general potential, and the other two cases can be proven in the same way (since \Cref{t:Dtmoment} also applies, as discussed in \Cref{rem:extbour}). 

    For the stationary DBM of \eqref{e:DBMV}, the rescaled particles $(\lambda^{(n)}_1(t)/\sqrt n,\lambda^{(n)}_2(t)/\sqrt n,\cdots,\lambda^{(n)}_n(t)/\sqrt n)$ follow the $\beta$-ensemble \eqref{e:beta} with potential $V(x)$. 
The first statement \eqref{e:la1bound} follows the first statement of \Cref{t:Dtmoment}. We recall from \eqref{e:deltasw} that
\begin{align}\begin{split}\label{e:Dtexp}
    \Delta^{(n)}_{t}(w)=\chi \left(m^{(n)}_{\zeta t}(E+\chi w)- \sqrt n m_V(B+n^{-1/2} \chi w)\right)+\OO(n^{-1/3} |w|).
\end{split}\end{align}
Note that $m^{(n)}_{\zeta t}(E+\chi w)$ has the same law as $\sqrt n s_n((E+\chi w)/\sqrt{n})$ in \Cref{t:Dtmoment}. Then \eqref{e:moment2} and \eqref{e:moment3} follow from the second statement of \Cref{t:Dtmoment} and \eqref{e:Dtexp}. 
\end{proof}

\begin{proof}[Proof of \Cref{l:onepointDelta}]
By taking $C$ large enough, we have $\bP(\widetilde \lambda^{(n)}_1(t)> C)<\varepsilon/3$, by \eqref{e:la1bound}. 

In the following we prove that \eqref{e:Dtw} holds (for any $w$ as specified), with probability at least $1-2\varepsilon/3$.
We denote $w=z^2$, $z=\kappa+\ri\eta$ with $\kappa\geq 0$. Then $a=\kappa^2-\eta^2$ and $b=2\kappa\eta$. Take $K>0$ large enough depending on $\varepsilon$. Then \eqref{e:Dtw} follows from the following two statements:
\begin{enumerate}
    \item  Denote $\cD_1=\{\kappa+\ri \eta: \eta \geq K, K/\eta\leq \kappa\leq 10\eta\}$. It holds with probability at least $1-\varepsilon/3$ that $|\Delta^{(n)}_t(w)|\leq C/(\kappa\sqrt{\eta})$ uniformly for $\sqrt w\in \cD_1$, $|w|\le n^{1/6}$.
    \item  Denote $\cD_2=\{\kappa+\ri \eta: \eta\geq K, \kappa\geq 10\eta\}$. It holds with probability at least $1-\varepsilon/3$ that 
    $|\Delta^{(n)}_t(w)|\leq C/(\kappa\sqrt{\eta})$ uniformly for $\sqrt w\in \cD_2$, $|w|\le n^{1/6}$.
\end{enumerate}
We first prove (1). For $\kappa\leq 10\eta$, by taking $q=\sqrt u/(Ce)$ in \eqref{e:moment2}, Markov's inequality implies
\begin{align} \label{e:kappasmall}
\bP\left(|\Delta^{(n)}_t(w)|\geq \frac{u}{\kappa \eta}+\frac{C|w|}{n^{1/3}}\right)\leq e^{-2\sqrt u /(Ce)},\quad u\geq 1.
\end{align}
By taking $u=\sqrt\eta/4$ in \eqref{e:kappasmall}, and noticing that for $|w|\leq n^{1/6}$, $C|w|/n^{1/3}\le 1/(2\kappa\sqrt \eta)$,  
\begin{align}\label{e:kappasmall2}
\bP\left(|\Delta^{(n)}_t(w)|\geq \frac{1}{\kappa \sqrt{\eta}}\right)\leq e^{-\eta^{1/4} /(Ce)}
\end{align}
We take a lattice $\cL_1=\{(\kappa+\ri\eta\in \cD_1: \eta=j^{1/3}, j\in\bN, \kappa\in\bN/\eta^2\}$, which is a discretization of $\cD_1$. By a union bound using \eqref{e:kappasmall2}, we have that $|\Delta^{(n)}_t(w)|\leq 1/(\kappa\sqrt{\eta})$ for each $w\in \cL_1$, $|w|\le n^{1/6}$, except for an event with probability at most
\[
    \sum_{j= \lceil K^3\rceil}^\infty 10 j e^{-j^{1/12}/(Ce)}\leq \varepsilon/3.
\]
By our construction of the lattice $\cL_1$, for any $\kappa+\ri\eta\in \cD_1$, there exists some $\kappa'+\ri\eta'\in \cL_1$ such that $|(\kappa+\ri \eta)-(\kappa'+\ri\eta')|\lesssim \eta^{-2}$.
By \Cref{l:Yproperty} we have
\begin{align}\label{e:dwDwbound}
    |\del_w (\Delta^{(n)}_t(w)+\sqrt w)|\leq \frac{\Im[\Delta^{(n)}_t(w)+\sqrt w]}{\Im[w]}
\end{align}
Using \eqref{e:dwDwbound}, via an argument similar to the proof of \Cref{p:rigidity}, the estimates for points in the lattice $\cL_1$ extend to all points in the domain $\cD_1$, implying (1).

Now consider (2). For $\kappa/10\geq \eta\geq K, \kappa^2\eta\geq Cq^{1/2}$, and $|w|\le n^{1/6}$, \eqref{e:moment3} simplifies to
\begin{align}
    \bE[|\Delta^{(n)}_t(w)+\OO(|w|n^{-1/3})|^{2q}]\leq  \frac{(Cq)^{2q}}{(\kappa\eta)^{2q} \kappa^{2q}}. 
\end{align}
By taking $q=\kappa\sqrt\eta /(2Ce)$, 
and using that $C|w|/n^{1/3}\le 1/(2\kappa\sqrt \eta)$ for $|w|\leq n^{1/6}$, Markov's inequality implies
\begin{align}\label{e:kappalarge}
      \bP\left(|\Delta^{(n)}_t(w)|\geq \frac{1}{\kappa\sqrt \eta}\right)\leq  e^{-2q}= e^{-\kappa\sqrt \eta /(Ce)}.
\end{align}
Then (2) follows by first taking a union bound of \eqref{e:kappalarge} over the lattice $\{(\kappa+\ri\eta\in \cD_2: \kappa=j^{1/3}, j\in\bN, \eta\in\bN/\kappa^2\}$, and (as before) using \eqref{e:dwDwbound} to extend the estimates for all points in $\cD_2$.
\end{proof}

\subsubsection{Tightness as time dependent processes}
\label{s:interval_tight}

The claim of \Cref{p:tight} follows from the following tightness statements.
\begin{lem}\label{l:Mt_tight}
    For any $T>0$ and compact $\cK\subset\bH$, both $\{M^{(n)}_t(w)-M^{(n)}_{-T}(w)\}_{t\in [-T,T], w\in \cK}$ and $\{\Delta^{(n)}_t(w)\}_{t\in [-T, T], w\in \cK}$ are tight as $n\to\infty$ .
\end{lem}
Assuming this, we now prove the (subsequential) convergence announced at the beginning of this subsection.
\begin{proof}[Proof of \Cref{p:tight}]
Using \Cref{l:Mt_tight}, and the Skorokhod representation theorem, we can take a subsequence, $n_1<n_2<\cdots$, such that almost surely, as $k\to\infty$, $\{\Delta_t^{(n_k)}(w)+\sqrt w\}_{t\in\bR,w\in \bH}$ converges (uniformly in compact sets) to a random process $\{Y_t(w)\}_{t\in\bR,w\in \bH}$, and $\{M^{(n_k)}_t(w)\}_{t\in\bR,w\in \bH}$ also converges to a random process $\{M_t(w)\}_{t\in\bR,w\in \bH}$.
Both convergences are under the uniform in compact topology.

Moreover, using \Cref{p:one_time_tight}, and by passing to a further subsequence, we can assume that for each rational $t$,  $Y_t$ is $(\fd, C_{*,t})$-Airy-like, 
where $\fd$ is a universal constant, and $\{C_{*,t}\}_{t\in\bQ}$ is a tight family of random variables; and 
$\{\widetilde \lambda^{(n_k)}_i(t)\}_{i\in\bN}$ converges almost surely to $\{\lambda_i(t)\}_{i\in\bN}$, which are the poles of $Y_t$.
(Note that here we cannot derive the same for all $t\in\bR$, as \Cref{l:onepointDelta} is not uniform in $t$.)

From the continuity of $Y_t$ in $t$, we can deduce that $Y_t$ is particle-generated from each $t\in\bR$, by \Cref{lem:StoMconv} and \Cref{lem:measuretoparticle}.
\smallskip

\noindent\textbf{Verify semimartingale decomposition.} From its construction we see that $Y_t$ satisfies \Cref{a:Infinite}.
Next, we check that \Cref{a:SDE} is also satisfied, i.e., we verify \eqref{e:defYt},  and \eqref{eq:qvmwwp}.

We can deduce the following uniform in compact convergence:
\begin{itemize}
    \item $\del_w^3 (\Delta^{(n_k)}_t(w)+\sqrt w) \to \del_w^3 Y_t(w)$. 
    \item $\del_w \del_{w'}\left(\frac{(\Delta^{(n_k)}_t(w)+\sqrt w)-(\Delta^{(n_k)}_t(w')+\sqrt w')}{w-w'}\right)\to \del_w \del_{w'}\left(\frac{Y_t(w)-Y_t(w')}{w-w'}\right)$, with the space of $(w, w')$ being $\{(w,w')\in \bH^2, w\neq w'\}$.
\end{itemize}
Both of these convergences follow from the uniform in compact convergence of $\Delta^{(n_k)}_t(w)+\sqrt w$, and using Cauchy's integral formula (and taking contour integrals around $w$ and $w'$) to compute the derivatives via taking a contour integral around $w$.
Then from \eqref{e:Mnvard} and \eqref{e:Mnvar}, and that $|\hcE_t^{(n_k)}(w,w')|\rightarrow 0$ by \Cref{prop:Dteq}, we get \eqref{eq:qvmwwp}.

Now in \eqref{e:Dteq}, from the uniform in compact convergence of $\Delta^{(n_k)}_t(w)+\sqrt w$, we can deduce the following uniform in compact convergence:
\begin{itemize}
    \item  $\del_w^2(\Delta^{(n)}_t(w)+\sqrt w)$ and $\del_w(\Delta^{(n)}_t(w)+\sqrt w)^2$ converge to $\del^2_w Y_t(w)$ and $\del_w (Y_t(w)^2)$, respectively.
    This again follows by using Cauchy's integral formula, to compute the derivatives via taking a contour integral around $w$.
\end{itemize}
Then since $|\cE_t^{(n_k)}(w)|\rightarrow 0$ by \Cref{prop:Dteq}, we get \eqref{e:defYt}.

\smallskip

\noindent\textbf{Uniform convergence.}
Now by \Cref{cor:31}, the poles of $\{Y_t\}_{t\in\bR}$ give a line ensemble, which we denote by $\{\lambda_i(t)\}_{i\in\bN, t\in\bR}$. And we have
\begin{equation}  \label{eq:Ytlambda}
    Y_t(w)=\sum_{i=1}^\infty\left(\frac{1}{\lambda_i(t)-w}-\frac{1}{\fa_i}\right)-\frac{\Ai'(0)}{\Ai(0)}
\end{equation}
We note that this is consistent with the previously defined $\{\lambda_i(t)\}_{i\in\bN}$ for rational $t$.

Take $T>0$.
We next show that almost surely, for each $i\in \bN$, $\lambda^{(n_k)}_i(t)\to\lambda_i(t)$ as $k\to\infty$, uniformly in $[-T, T]$.

First, from \Cref{cor:31} we take $\Lambda=\sup_{t\in [-T, T], i\in \bN}|\lambda_i(t)-\fa_i|<\infty$.
Let $X=\Lambda+D$ for some $D>0$. Then for each $t\in [-T, T]$ we have $X\ge \lambda_1(t) - \fa_1 + D>\lambda_1(t)+D$, and
\[
\Im[Y_t(X+\ri)] = \sum_{i=1}^\infty \frac{1}{|\lambda_i(t)-X|^2+1} < \sum_{i=1}^\infty \frac{1}{(D-\fa_i)^2+1},
\]
by \eqref{eq:Ytlambda}. Thus by \eqref{eq:weairy}, we can take $D$ to be a large universal constant, such that
\[
\max_{t\in[-T,T]}\Im[Y_t(X+\ri)] < 0.01.\]
Then by the uniform in compact convergence of $\{\Delta_t^{(n_k)}(w)+\sqrt w\}_{t\in\bR,w\in \bH}$, for $k$ large enough, we have \[\max_{t\in[-T,T]}\Im[\Delta_t^{(n_k)}(X+\ri)+\sqrt{X+\ri}]<0.02,\] therefore $X\not\in \{\lambda_i^{(n_k)}(t)\}_{i\in\bN}$ for any $t\in [-T,T]$.
On the other hand, by the convergence of $\lambda^{(n_k)}_1(0)\to\lambda_i(0)$, for $k$ large enough, we have $\lambda^{(n_k)}_1(0)<X$.
Thus, $\max_{t\in [-T,T]} \lambda_1^{(n_k)}(t)<X$ whenever $k$ is large enough.

Suppose that for some $i\in\bN$, $\lambda^{(n_k)}_i(t)$ does not uniformly converge to $\lambda_i(t)$ in $[-T, T]$ (as $k\to\infty$).
Then there is a sequence $t_1, t_2, \cdots$, $\lim_{k\to\infty} t_k = t_0\in [-T,T]$, such that as $k\to\infty$, $\lambda^{(n_k)}_i(t_k)$ does not converge to $\lambda_i(t_0)$.

However, by the uniform (in $t$ and $w$) convergence of $\Delta_t^{(n)}(w)+\sqrt{w}$ to $Y_t(w)$, we have that $\lim_{k\to\infty}\Delta_{t_k}^{(n_k)}(w)+\sqrt{w}=Y_{t_0}(w)$, uniformly for $w$ in any compact subset of $\bH$.
With the upper bound of $\lambda^{(n_k)}_1(t_k)<X$, by  \Cref{lem:StoMconv} and \Cref{lem:measuretoparticle} we have that $\lim_{k\to\infty}\lambda^{(n_k)}_i(t_k)=\lambda_i(t_0)$ for each $i\in\bN$, thereby arriving at a contradiction.
\end{proof}

\begin{proof}[Proof of \Cref{l:Mt_tight}]
Take $p$ to be a large constant.
In this proof, all the constants in $\lesssim$ and $\OO(\cdot)$ can depend on $\cK$, $T$, and $p$.

We start with the martingale terms.
From \eqref{e:Mnvardu}, we have
\begin{align*}
\left|\frac{\rd}{\rd t}\langle  M^{(n)}(w), M^{(n)}(\overline{w})\rangle_t\right|
    &
    \le \frac{2}{\beta}\sum_{i=1}^n \frac{1}{|\widetilde \lambda^{(n)}_i(t)-w|^4}+|\hcE^{(n)}_t(w,\overline{w})|\\
    &\le \frac{2}{\beta}\cdot\frac{\Im[\Delta^{(n)}_t(w)+\sqrt w]}{\Im^3[w]}+|\hcE^{(n)}_t(w,\overline{w})|,
\end{align*}
and
\begin{align*}
    &\left|\frac{\rd}{\rd t}\langle  M^{(n)}(w)-M^{(n)}(w'), M^{(n)}(\overline{w})-M^{(n)}(\overline{w'})\rangle_t\right|
    \\ & \leq \frac{2}{\beta}\sum_{i=1}^n \frac{|w-w'|^2|\widetilde \lambda^{(n)}_i(t)-w+\widetilde \lambda^{(n)}_i(t)-w'|^2}{|\widetilde \lambda^{(n)}_i(t)-w|^4|\widetilde \lambda^{(n)}_i(t)-w'|^4}+|\hcE^{(n)}_t(w,\overline{w})-2\hcE^{(n)}_t(w,\overline{w'})+\hcE^{(n)}_t(w',\overline{w'})|\\
    &\leq \frac{2}{\beta}|w-w'|^2\left(\frac{\Im[\Delta^{(n)}_t(w)+\sqrt w]}{\Im^3[w]\Im^2[w']}+\frac{\Im[\Delta^{(n)}_t(w')+\sqrt w']}{\Im^2[w]\Im^3[w']}\right)\\&+|\hcE^{(n)}_t(w,\overline{w})-2\hcE^{(n)}_t(w,\overline{w'})+\hcE^{(n)}_t(w',\overline{w'})|,
\end{align*}
for $w\neq w'$.
Using \eqref{e:herr} to bound the error terms in the Laguerre and Jacobi cases, we have that the above two are $\OO(1+|\Delta^{(n)}_t(w)+\sqrt w|)$ and $\OO(|w-w'|^2(1+|\Delta^{(n)}_t(w)+\sqrt w|))$, respectively.
Notice that $\overline{M_t^{(n)}(w)}=M_t^{(n)}(\overline{w})$. Then thanks to Burkholder-Davis-Gundy inequality, for any $w,w'\in \cK$, and $-T\le t\le t'\le T$,
\begin{equation}  \label{eq:difMp}
\bE\left[\left|M^{(n)}_{t'}(w)-M^{(n)}_t(w)\right|^p\right]
\lesssim \bE\left[\left|\int_t^{t'}\rd \langle M^{(n)}(w),M^{(n)}(\overline{w})\rangle_s\right|^{p/2}\right]
\lesssim |t'-t|^{p/2} ,
\end{equation}
and
\begin{multline}  \label{eq:difMpd}
\bE\left[\left|(M^{(n)}_t(w)-M^{(n)}_{-T}(w)) - (M^{(n)}_t(w')-M^{(n)}_{-T}(w'))\right|^p\right] \\
\lesssim \bE\left[\left|\int_{-T}^{t}\rd \langle M^{(n)}(w)- M^{(n)}(w'), M^{(n)}(\overline{w})- M^{(n)}(\overline{w'})\rangle_s\right|^{p/2}\right]
\lesssim |w-w'|^p,
\end{multline}
where in the last inequalities we used the above bounds and \eqref{e:moment2}.

The tightness of $\{M^{(n)}_t(w)-M^{(n)}_{-T}(w)\}_{t\in [-T,T], w\in \cK}$ then follows from these estimates with $p>4$ (see e.g., \cite[Lemma 5.9]{MR3403994}, which can be proved by bounding the modulus of continuity via a union bound across scale, and using e.g., \cite[Theorem 7.3]{Bil}).

We now turn to $\Delta^{(n)}_t(w)$.
From the one time tightness in \Cref{p:one_time_tight}, it suffices to derive the tightness of $\{\Delta^{(n)}_t(w)-\Delta^{(n)}_{-T}(w)\}_{t\in [-T, T], w\in \cK}$.
From \eqref{e:Dteq}, and using \Cref{l:Yproperty}, for any $w\in\cK$ and $-T\le t\le t'\le T$ we have
\begin{align*}
    &\Delta^{(n)}_{t'}(w)-\Delta^{(n)}_{t}(w)=M^{(n)}_{t'}(w)-M^{(n)}_t(w)\\
    & +\int_t^{t'}\OO\left(1+|\cE^{(n)}_s(w)|+\frac{\Im[\Delta_t^{(n)}(w)+\sqrt w]}{\Im[w]^2}+\frac{\Im[\Delta_t^{(n)}(w)+\sqrt w]|\Delta_t^{(n)}(w)+\sqrt w|}{\Im[w]}\right)\rd s.
\end{align*}
Then by \eqref{e:err}, \eqref{e:moment2}, and \eqref{eq:difMp}, we have
\[
\bE\left[\left|\Delta^{(n)}_{t'}(w)-\Delta^{(n)}_t(w)\right|^p\right]
\lesssim |t'-t|^{p/2}.
\]
Besides, for any $w,w'\in\cK$, for each $r=0, 1, 2$ we have
\begin{align*}
    &\phantom{{}={}}|\partial_w^r(\Delta^{(n)}_t(w)+\sqrt w)-\partial_w^r(\Delta^{(n)}_t(w')+\sqrt{w'})|
    \leq (r+1)!|w-w'|\sum_{i=1}^n\frac{|\wt \la_t^{(n)}-w|^{r}+|\wt \la_t^{(n)}-w'|^{r}}{|\wt \la_t^{(n)}-w|^{r+1}|\wt \la_t^{(n)}-w'|^{r+1}}\\
    &\leq (r+1)!|w-w'|\left(\frac{\Im[\Delta^{(n)}_t(w)+\sqrt w]}{\Im[w]^{r+1}}+\frac{\Im[\Delta^{(n)}_t(w')+\sqrt{w'}]}{\Im[w']^{r+1}}\right).
\end{align*}
Then from \eqref{e:Dteq}, using this and \eqref{e:err}, \eqref{e:moment2}, \eqref{eq:difMpd}, we get
\[
\bE\left[\left|(\Delta^{(n)}_t(w)-\Delta^{(n)}_{-T}(w)) - (\Delta^{(n)}_t(w')-\Delta^{(n)}_{-T}(w'))\right|^p\right]
\lesssim |w-w'|^p.
\]
Thereby the tightness of $\{\Delta^{(n)}_t(w)-\Delta^{(n)}_{-T}(w)\}_{t\in [-T, T], w\in \cK}$ follows, from these moments bounds with $p>4$.
\end{proof}

\appendix
\section{Particle location and Stieltjes transform}

In this appendix we give some results on particle-generated Nevanlinna functions, which are used frequently in the main text.

\subsection{Airy function properties}  \label{ss:afas}
We first prove the Nevanlinna representation of $-\frac{\Ai'(w)}{\Ai(w)}$.
\begin{proof}[Proof of \eqref{eq:weairy}]
The Weierstrass factorization of $\Ai$ is given by
\[
\Ai(w)
= \Ai(0)\,
\exp\!\left( \frac{\Ai'(0)}{\Ai(0)} w \right)
\prod_{i=1}^\infty
\left( 1 - \frac{w}{\fa_i} \right)
\exp\!\left( \frac{w}{\fa_i} \right).
\]
See e.g., \cite[eq.(2.8)]{KTzeros}. Then (modulos $2\pi\ri\bZ$) we have
\[
\log(\Ai(w)) = \log(\Ai(0)) + \frac{\Ai'(0)}{\Ai(0)} w + \sum_{i=1}^\infty \log\left( 1 - \frac{w}{\fa_i} \right) + \frac{w}{\fa_i},
\]
where the convergence of the summation in $i$ follows from \eqref{e:aklocation}, and the fact that $|\log(1-z)+z| \lesssim |z|^{-2}$ for $|z|<\frac{1}{2}$. Then by taking the derivative in $w$ we get \eqref{eq:weairy}.
\end{proof}
We then establish the square root growth of $-\frac{\Ai'(w)}{\Ai(w)}$, as $w\to\infty$.
\begin{proof}[Proof of \eqref{e:aiasymp}]
The Airy function has the following asymptotic formula. 
For $|\arg(w)|<\pi$, 
\[
\Ai(w)\sim \frac{\exp(-\zeta)}{w^{1/4}}\sum_{n=0}^\infty \frac{\Gamma(n+5/6)\Gamma(n+1/6)}{4\pi^{3/2}n! (-2\zeta)^n},
\]
\[
\Ai'(w)\sim -w^{1/4}\exp(-\zeta)\sum_{n=0}^\infty \frac{1+6n}{1-6n}\cdot \frac{\Gamma(n+5/6)\Gamma(n+1/6)}{4\pi^{3/2}n! (-2\zeta)^n},
\]
where $\zeta=\frac{2}{3}w^{3/2}$. 
In particular, there is
\[
\left|w^{1/4}\Ai(w)-\frac{\exp(-2w^{3/2}/3)}{2\sqrt{\pi}} \left(1-\frac{5}{48w^{3/2}}\right) \right| \lesssim  (\pi - |\arg(w)|)^{-7/6} \frac{|\exp(-2w^{3/2}/3)|}{|w^3|},
\]
\[
\left|w^{-1/4}\Ai'(w)+\frac{\exp(-2w^{3/2}/3)}{2\sqrt{\pi}} \left(1+\frac{7}{48w^{3/2}}\right) \right| \lesssim (\pi - |\arg(w)|)^{-1} \frac{|\exp(-2w^{3/2}/3)|}{|w^3|},
\]
for any $w\in\bC\setminus \bR_{\le 0}$.
See e.g., \cite[Section 9.7(iv)]{NIST:DLMF} and \cite[Appendix B]{nemes2017error}.
By taking the ratio of the above two estimates, we get \eqref{e:aiasymp}, for any $w\in\bC$ with $|w|$ large enough and $|\arg(w)| < \pi - |w|^{-9/7}$.
\end{proof}

The rest of this appendix connects the Airy-like property (from \Cref{defn:nvali}) and the locations of the poles, for particle-generated Nevanlinna functions.
We note that most of the arguments to be presented are classical in random matrix theory. 

Our starting point is the following relation between any Nevanlinna function and its corresponding measure. It can be viewed as a version of the Helffer-Sj{\"o}strand formula (see e.g., \cite[Section 11.2]{erdHos2017dynamical}).

\subsection{Helffer-Sj{\"o}strand formula}
Take any compactly supported smooth test functions $f$ and $\chi$ on $\bR$, such that $\chi=1$ in a neighborhood of $0$.
Define $\tilde{f}(x+\ri y) := (f(x)+\ri y f'(x))\chi(y)$.
Then for any $\la \in \bR$, we have (for $\del_{\bar{z}}=\frac{1}{2}(\del_x+\ri\del_y)$) 
\[
    f(\la)=\frac{1}{\pi}\int_{\bR^2} \frac{\del_{\bar{z}}\tilde f(x+\ri y)}{\lambda-x-\ri y} \rd x \rd y,
\]
which further leads to the following statement.
\begin{lem}  \label{lem:HSf}
Take any Nevanlinna function $Y:\bH\to \bH\cup\bR$, with Nevanlinna representation
\[
 Y(w)=b+cw+\int \left( \frac{1}{\la-w} - \frac{\la}{1+\la^2} \right)  \rd \mu(\la),
\]
for some $b,c\in \bR$, $c\geq 0$, and $\mu$ being a Borel measure on $\bR$.
Then we have
\[
\int f(\la) \rd \mu(\la) = \frac{2}{\pi} \int_{x+\ri y\in \bH} \Re\left[\del_{\bar{z}} \tilde{f}(x+\ri y) Y(x+\ri y)\right] \rd x \rd y.
\]
\end{lem}
These identities are classical and follow directly from standard complex analysis arguments; see e.g., \cite[Section 11.2]{erdHos2017dynamical}. We omit the proofs here.

\subsection{Proof of \Cref{lem:parti-clo}}\label{s:HS}
Our strategy is to extract the particle locations from $Y$, using \Cref{lem:HSf}.

We denote $\sigma=\delta/6$, and take any $s>0$, and a smooth test function $f=f^{(s)}:\bR\to\bR_{\ge 0}$, satisfying
\begin{itemize}
    \item $f=1$ on $[-s^{2/3}, K]$;
    \item $f=0$ on $\bR\setminus [-s^{-\sigma}-s^{2/3}, 2K]$;
    \item $|f'|\lesssim s^{\sigma}$ and $|f''|\lesssim s^{2\sigma}$ on $[-s^{-\sigma}-s^{2/3}, -s^{2/3}]$;
    \item $|f'|\lesssim 1$ and $|f''|\lesssim 1$ on $[K, 2K]$.
\end{itemize}
We claim that
\begin{equation}\label{e:particlediff0}
\left|\sum_{i=1}^\infty f(x_i)-\frac{1}{\pi}\int_{\bR_+} f(-x) \sqrt x \rd x\right|\lesssim K^4 s^{1/3-\sigma},
\end{equation}
for any $s>10K^3$. 
Assuming this, we have that for some universal constant $C>0$,
\begin{equation}\label{e:particlediff01}
| \{i \in \bN: x_i \in [-s^{2/3}, K] \} | < \frac{2}{3\pi} (s^{-\sigma}+s^{2/3})^{3/2} + CK^4s^{1/3-\sigma},
\end{equation}
\begin{equation}\label{e:particlediff02}
 | \{ i \in \bN: x_i \in [-s^{-\sigma}-s^{2/3}, K] \} | > \frac{2s}{3\pi} - CK^4s^{1/3-\sigma}.
\end{equation}

Let $C_*$ be a universal constant that is large enough depending on $C$.
For each $i\in\bN$, take $s_{i,-}=\left(\left(\frac{3\pi i}{2}\right)^{2/3} + C_*K^4i^{-\sigma}\right)^{3/2}>10K^3$.
Then $\frac{2s_{i,-}}{3\pi} - CK^4s_{i,-}^{1/3-\sigma} > i$.
By taking $s=s_{i,-}$ in \eqref{e:particlediff02}, and using the fact that $K\ge x_1 \ge x_2 \ge \cdots$, we have  $x_i > -s_{i,-}^{-\sigma} - s_{i,-}^{2/3} > -\left(\frac{3\pi i}{2}\right)^{2/3} - 2C_*K^4i^{-\sigma}$.

For $i\in \bN$, if $i<K^5$, we have $x_i \le K <  -\left(\frac{3\pi i}{2}\right)^{2/3} + C_*K^4i^{-\sigma}$.
Otherwise, take $s_{i,+}=\left(\left(\left(\frac{3\pi i}{2}\right)^{2/3} - C_*K^4i^{-\sigma}\right)\vee 10K^2\right)^{3/2}>10K^3$. 
Then $\frac{2}{3\pi} (s_{i,+}^{-\sigma}+s_{i,+}^{2/3})^{3/2} + CK^4s_{i,+}^{1/3-\sigma} < i$.
By taking $s=s_{i,+}$ in \eqref{e:particlediff01}, and using the fact that $K\ge x_1 \ge x_2 \ge \cdots$, we have  $x_i < - s_{i,+}^{2/3} < -\left(\frac{3\pi i}{2}\right)^{2/3} + C_*K^4i^{-\sigma}$.

Thus we have that $\left|x_i + \left(\frac{3\pi i}{2}\right)^{2/3} \right|\lesssim K^4i^{-\sigma}$.
With \eqref{e:aklocation}, the estimate \Cref{lem:parti-clo} follows.

We now prove \eqref{e:particlediff0}. Take another smooth test function $\chi=\chi^{(s)}:\bR\to\bR_{\ge 0}$, with $\chi=1$ on $[-s^{2/3}, s^{2/3}]$, $\chi=0$ on $\bR\setminus [-s^{2/3}-1, s^{2/3}+1]$, and $|\chi'|\lesssim 1$ on $\bR$. 
Let $\tilde{f}(x+\ri y)=(f(x)+\ri y f'(x))\chi(y)$ for any $x,y\in\bR$.
Then by \Cref{lem:HSf} applied to the Nevanlinna functions $Y$ and $w\mapsto \sqrt{w}$, we have
\begin{align}\label{e:fdiff}
    \sum_{i=1}^\infty f(x_i)-\frac{1}{\pi}\int_{\bR_+} f(-x) \sqrt x \rd x=\frac{2}{\pi}\int_{x+\ri y\in \bH}\Re\left[\del_{\bar{z}}\tilde f(x+\ri y)(Y(x+\ri y)-\sqrt{x+\ri y})\right]\rd x\rd y.
\end{align}
Denote $\cE(x+\ri y):=Y(x+\ri y)-\sqrt{x+\ri y}$. The RHS of \eqref{e:fdiff} decomposes into
\begin{align}\label{e:fdiff2}
-\frac{1}{\pi}\int_{x+\ri y\in \bH} \left( f''(x) y\chi(y) \Im[\cE(x+\ri y)] + \Im\left[ (f(x)+\ri y f'(x))\chi'(y) \cE(x+\ri y) \right] \right)\rd x\rd y.
\end{align}
\noindent\textbf{The first term.} For the term $f''(x) y\chi(y) \Im[\cE(x+\ri y)]$, we note that $f''$ is non-zero only on the two intervals $[-s^{-\sigma}-s^{2/3}, -s^{2/3}]$ and $[K, 2K]$. 

For the first interval, we break the integral in $y$ into two parts: $0<y<4K^2s^{-\sigma}$, and $y\ge 4K^2s^{-\sigma}$.
Using that $y\Im[Y(x+\ri y)]$ is increasing in $y$ (from \Cref{l:Yproperty}), we have that for any $x\in [-s^{-\sigma}-s^{2/3}, -s^{2/3}]$ and $0<y<4K^2s^{-\sigma}$,
\[
y\Im[\cE(x+\ri y)] < y\Im[Y(x+\ri y)] \le 4K^2s^{-\sigma}\Im[Y(x+4\ri K^2s^{-\sigma})] ,
\]
which, by the second condition of $Y$, is further bounded by (using that $\sigma\le \delta/3$ and $s>K^3$)
\[
4K^2s^{-\sigma}\sqrt{|x+4\ri K^2s^{-\sigma}|} + |x+4\ri K^2s^{-\sigma}|^{(1-\delta)/2} \lesssim K^2s^{1/3-\sigma}.
\]
It follows that (using $|f''|\lesssim s^{2\sigma}$)
\[
\frac{1}{\pi}\int_0^{4K^2s^{-\sigma}}\int_{-s^{-\sigma}-s^{2/3}}^{-s^{2/3}} |f''(x)| y\chi(y)\Im [\cE(x+\ri y)]\rd x\rd y \lesssim K^4 s^{1/3-\sigma}.
\]
For integrating over $y\ge 4K^2s^{-\sigma}$, we perform an integration by parts in $x$, and get
\begin{align*}
&\frac{1}{\pi}\int_{4K^2s^{-\sigma}}^\infty \int_{-s^{-\sigma}-s^{2/3}}^{-s^{2/3}} f''(x) y\chi(y) \Im[\cE(x+\ri y)]\rd x\rd y\\
=&\frac{1}{\pi}\int_{4K^2s^{-\sigma}}^\infty \int_{-s^{-\sigma}-s^{2/3}}^{-s^{2/3}} f'(x) y\chi(y) \Im[\del_x \cE(x+\ri y)]\rd x\rd y\\
=&\frac{1}{\pi}\int_{4K^2s^{-\sigma}}^\infty \int_{-s^{-\sigma}-s^{2/3}}^{-s^{2/3}}  f'(x) y\chi(y) \Re[\del_y \cE(x+\ri y)]\rd x\rd y.
\end{align*}
Via another integration by parts in $y$, this equals
\begin{align*}
&\frac{1}{\pi} \int_{-s^{-\sigma}-s^{2/3}}^{-s^{2/3}}  f'(x) y\chi(y) \Re[ \cE(x+\ri y)]\rd x \bigg|_{y=4K^2s^{-\sigma}} \\
-&\frac{1}{\pi}\int_{4K^2s^{-\sigma}}^\infty \int_{-s^{-\sigma}-s^{2/3}}^{-s^{2/3}}  f'(x) \del_y (y\chi(y)) \Re[\cE(x+\ri y)]\rd x\rd y.
\end{align*}
The first line above is bounded (using the second condition of $Y$) by
\[
\frac{1}{\pi}|s^{-\sigma}+s^{2/3}+4\ri K^2s^{-\sigma}|^{(1-\delta)/2} \lesssim s^{(1-\delta)/3};
\]
and as for the second line, its absolute value is bounded by
\[
\frac{1}{\pi}\int_{4K^2s^{-\sigma}}^\infty |\partial_y (y\chi(y)) | \frac{|s^{-\sigma}+s^{2/3}+\ri y|^{(1-\delta)/2}}{y} \rd y
\lesssim s^{(1-\delta)/3} \int_{4K^2s^{-\sigma}}^\infty \frac{|\partial_y (y\chi(y)) |}{y} \rd y \lesssim s^{1/3-\sigma},
\]
where the last inequality holds whenever $\sigma\le\delta/6$.

For the integral of $f''(x) y\chi(y) \Im[\cE(x+\ri y)]$ for $x\in [K, 2K]$, we similarly break the integral in $y$ into two parts: $0<y<4K^2$ and $y\ge 4K^2$.
Using that $y\Im[Y(x+\ri y)]$ is increasing in $y$ (from \Cref{l:Yproperty}), we have that for any $x\in [K, 2K]$ and $0<y<4K^2$,
\[
y\Im[\cE(x+\ri y)] < y\Im[Y(x+\ri y)] \le 4K^2\Im[Y(x+4\ri K^2)] \lesssim K^3,
\]
where we used the second condition of $Y$ for the last inequality.
Therefore (using $|f''|\lesssim 1$ on $[K, 2K]$) we have
\[
\frac{1}{\pi}\int_0^{4K^2}\int_K^{2K} |f''(x)| y\chi(y)\Im [\cE(x+\ri y)]\rd x\rd y \lesssim K^6.
\]
For integrating over $y\ge 4K^2$, using a similar integration by parts procedure, we get
\begin{align*}
&\frac{1}{\pi} \int_K^{2K}  f'(x) y\chi(y) \Re[ \cE(x+\ri y)]\rd x \bigg|_{y=4K^2} \\
-&\frac{1}{\pi}\int_{4K^2}^\infty \int_K^{2K}  f'(x) \del_y (y\chi(y)) \Re[\cE(x+\ri y)]\rd x\rd y.
\end{align*}
Using the second condition of $Y$, the first line above is $\lesssim K$, and the absolute value of the second line is bounded by (using again $\sigma\le\delta/3$)
\[
\frac{1}{\pi}\int_{4K^2}^\infty |\partial_y (y\chi(y)) | \frac{|2K+\ri y|^{(1-\delta)/2}}{y} \rd y
\lesssim s^{(1-\delta)/3} < s^{1/3-\sigma}.
\]

\noindent\textbf{The second term.} We now consider the term $\Im\left[ (f(x)+\ri y f'(x))\chi'(y) \cE(x+\ri y) \right]$ in \eqref{e:fdiff2}.
Note that $\chi'(y)\neq 0$ only for $y\in [s^{2/3}, s^{2/3}+1]$. For such $y$ and $x\in [-s^{-\sigma}-s^{2/3}, -s^{2/3}]\cup [K, 2K]$, we have $|\cE(x+\ri y)|\lesssim s^{-1/3-\delta/3}$, by the second condition of $Y$. Therefore, 
\begin{align}
&\frac{1}{\pi}\int_{x+\ri y \in \bH} \left|\Im\left[ (f(x)+\ri y f'(x))\chi'(y) \cE(x+\ri y) \right] \right| \rd x\rd y\\
\lesssim & 
s^{-1/3-\delta/3}
\int_{s^{2/3}}^{s^{2/3}+1}\int (1+s^{2/3}|f'(x)|)|\chi'(y)| \rd x\rd y \lesssim s^{1/3-\delta/3}\leq s^{1/3-\sigma},
\end{align}
using $\sigma\le\delta/3$.

In summary, we have proved  \eqref{e:particlediff0} by putting together the above estimates. Thus the conclusion follows.
\qed

\bibliographystyle{plain}
\bibliography{bibliography.bib}

\begin{thebibliography}{100}

\bibitem{adhikari2020dyson}
Arka Adhikari and Jiaoyang Huang.
\newblock Dyson {B}rownian motion for general {$\beta$} and potential at the
  edge.
\newblock {\em Probab. Theory Related Fields}, 178(3-4):893--950, 2020.

\bibitem{PWC}
Mark Adler, Patrik~L. Ferrari, and Pierre van Moerbeke.
\newblock Airy processes with wanderers and new universality classes.
\newblock {\em Ann. Probab.}, 38(2):714--769, 2010.

\bibitem{ULS}
Amol Aggarwal.
\newblock Universality for lozenge tiling local statistics.
\newblock {\em Ann. of Math. (2)}, 198(3):881--1012, 2023.

\bibitem{aggarwal2024scaling}
Amol Aggarwal, Ivan Corwin, and Milind Hegde.
\newblock {Scaling limit of the colored ASEP and stochastic six-vertex models}.
\newblock {\em arXiv preprint arXiv:2403.01341}, 2024.

\bibitem{aggarwal2021edge}
Amol Aggarwal and Jiaoyang Huang.
\newblock Edge statistics for lozenge tilings of polygons, {II}: {A}iry line
  ensemble.
\newblock {\em arXiv preprint arXiv:2108.12874}, 2021.

\bibitem{aggarwal2023strong}
Amol Aggarwal and Jiaoyang Huang.
\newblock Strong characterization for the {A}iry line ensemble.
\newblock {\em arXiv preprint arXiv:2308.11908}, 2023.

\bibitem{aizenman2015ubiquity}
Michael Aizenman and Simone Warzel.
\newblock On the ubiquity of the cauchy distribution in spectral problems.
\newblock {\em Probability Theory and Related Fields}, 163:61--87, 2015.

\bibitem{amir2011probability}
Gideon Amir, Ivan Corwin, and Jeremy Quastel.
\newblock Probability distribution of the free energy of the continuum directed
  random polymer in 1+ 1 dimensions.
\newblock {\em Communications on pure and applied mathematics}, 64(4):466--537,
  2011.

\bibitem{MR2760897}
Greg~W. Anderson, Alice Guionnet, and Ofer Zeitouni.
\newblock {\em An introduction to random matrices}, volume 118 of {\em
  Cambridge Studies in Advanced Mathematics}.
\newblock Cambridge University Press, Cambridge, 2010.

\bibitem{OSP}
Omer Angel, Alexander Holroyd, and Dan Romik.
\newblock The oriented swap process.
\newblock {\em Ann. Probab.}, 37(5):1970--1998, 2009.

\bibitem{ashbury2022random}
Lucas Ashbury-Bridgwood.
\newblock {\em Random Canonical Products and the Secular Function of the
  Stochastic Airy Operator}.
\newblock PhD thesis, University of Toronto (Canada), 2022.

\bibitem{TZisde}
Theodoros Assiotis and Zahra~Sadat Mirsajjadi.
\newblock {ISDE} with logarithmic interaction and characteristic polynomials.
\newblock {\em arXiv preprint arXiv:2408.00717}, 2024.

\bibitem{MR1682248}
Jinho Baik, Percy Deift, and Kurt Johansson.
\newblock On the distribution of the length of the longest increasing
  subsequence of random permutations.
\newblock {\em J. Amer. Math. Soc.}, 12(4):1119--1178, 1999.

\bibitem{MR1845180}
Jinho Baik and Eric~M. Rains.
\newblock The asymptotics of monotone subsequences of involutions.
\newblock {\em Duke Math. J.}, 109(2):205--281, 2001.

\bibitem{MR3855355}
Guillaume Barraquand, Alexei Borodin, Ivan Corwin, and Michael Wheeler.
\newblock Stochastic six-vertex model in a half-quadrant and half-line open
  asymmetric simple exclusion process.
\newblock {\em Duke Math. J.}, 167(13):2457--2529, 2018.

\bibitem{bauerschmidt2020edge}
Roland Bauerschmidt, Jiaoyang Huang, Antti Knowles, and Horng-Tzer Yau.
\newblock Edge rigidity and universality of random regular graphs of
  intermediate degree.
\newblock {\em Geometric and Functional Analysis}, 30(3):693--769, 2020.

\bibitem{bekerman2018transport}
Florent Bekerman.
\newblock Transport maps for $\beta$-matrix models in the multi-cut regime.
\newblock {\em Random Matrices: Theory and Applications}, 7(01):1750013, 2018.

\bibitem{bekerman2015transport}
Florent Bekerman, Alessio Figalli, and Alice Guionnet.
\newblock Transport maps for $\beta$-matrix models and universality.
\newblock {\em Communications in mathematical physics}, 338(2):589--619, 2015.

\bibitem{Bil}
Patrick Billingsley.
\newblock {\em Convergence of probability measures}.
\newblock John Wiley \& Sons Inc., New York, 1999.

\bibitem{borodin2014macdonald}
Alexei Borodin and Ivan Corwin.
\newblock Macdonald processes.
\newblock {\em Probability Theory and Related Fields}, 158(1-2):225--400, 2014.

\bibitem{MR2363389}
Alexei Borodin, Patrik~L. Ferrari, Michael Pr\"{a}hofer, and Tomohiro Sasamoto.
\newblock Fluctuation properties of the {TASEP} with periodic initial
  configuration.
\newblock {\em J. Stat. Phys.}, 129(5-6):1055--1080, 2007.

\bibitem{borodin2015general}
Alexei Borodin and Vadim Gorin.
\newblock General $\beta$-{J}acobi corners process and the {Gaussian Free
  Field}.
\newblock {\em Communications on Pure and Applied Mathematics},
  68(10):1774--1844, 2015.

\bibitem{borodin2010q}
Alexei Borodin, Vadim Gorin, and Eric~M Rains.
\newblock q-distributions on boxed plane partitions.
\newblock {\em Selecta Mathematica}, 16(4):731--789, 2010.

\bibitem{bourgade2021extreme}
Paul Bourgade.
\newblock Extreme gaps between eigenvalues of {W}igner matrices.
\newblock {\em Journal of the European Mathematical Society}, 24(8):2823--2873,
  2021.

\bibitem{MR3253704}
Paul Bourgade, L{\'a}szl{\'o} Erd{\"o}s, and Horng-Tzer Yau.
\newblock Edge universality of beta ensembles.
\newblock {\em Comm. Math. Phys.}, 332(1):261--353, 2014.

\bibitem{bourgade2022optimal}
Paul Bourgade, Krishnan Mody, and Michel Pain.
\newblock Optimal local law and central limit theorem for $\beta$-ensembles.
\newblock {\em Communications in Mathematical Physics}, 390(3):1017--1079,
  2022.

\bibitem{de1995statistical}
Anne Boutet~de Monvel, Leonid Pastur, and Maria Shcherbina.
\newblock On the statistical mechanics approach in the random matrix theory:
  integrated density of states.
\newblock {\em Journal of statistical physics}, 79:585--611, 1995.

\bibitem{bru1991wishart}
Marie-France Bru.
\newblock Wishart processes.
\newblock {\em Journal of Theoretical Probability}, 4:725--751, 1991.

\bibitem{AOSP}
Alexey Bufetov, Vadim Gorin, and Dan Romik.
\newblock Absorbing time asymptotics in the oriented swap process.
\newblock {\em Ann. Appl. Probab.}, 32(2):753--763, 2022.

\bibitem{bufetov2018asymptotics}
Alexey Bufetov and Alisa Knizel.
\newblock Asymptotics of random domino tilings of rectangular {A}ztec diamonds.
\newblock In {\em Annales de l'Institut Henri Poincar{\'e}, Probabilit{\'e}s et
  Statistiques}, volume~54, pages 1250--1290. Institut Henri Poincar{\'e},
  2018.

\bibitem{MR4421174}
Alexey Bufetov and Peter Nejjar.
\newblock Cutoff profile of {ASEP} on a segment.
\newblock {\em Probab. Theory Related Fields}, 183(1-2):229--253, 2022.

\bibitem{BGS}
Anna Bykhovskaya, Vadim Gorin, and Sasha Sodin.
\newblock Uniform confidence intervals for signal strength in signal plus noise
  models.
\newblock {\em in preparation}, 2024.

\bibitem{cepa2006equations}
Emmanuel C{\'e}pa.
\newblock Equations diff{\'e}rentielles stochastiques multivoques.
\newblock In {\em S{\'e}minaire de Probabilit{\'e}s XXIX}, pages 86--107.
  Springer, 2006.

\bibitem{cepa1997diffusing}
Emmanuel C{\'e}pa and Dominique L{\'e}pingle.
\newblock Diffusing particles with electrostatic repulsion.
\newblock {\em Probability theory and related fields}, 107(4):429--449, 1997.

\bibitem{choodnovsky1977pole}
D.V. Choodnovsky and G.V. Choodnovsky.
\newblock Pole expansions of nonlinear partial differential equations.
\newblock {\em Il Nuovo Cimento B (1971-1996)}, 40:339--353, 1977.

\bibitem{corwin2014brownian}
Ivan Corwin and Alan Hammond.
\newblock Brownian gibbs property for {A}iry line ensembles.
\newblock {\em Inventiones mathematicae}, 195(2):441--508, 2014.

\bibitem{dauvergne2023wiener}
Duncan Dauvergne.
\newblock Wiener densities for the {A}iry line ensemble.
\newblock {\em arXiv preprint arXiv:2302.00097}, 2023.

\bibitem{dauvergne2018directed}
Duncan Dauvergne, Janosch Ortmann, and B{\'a}lint Vir{\'a}g.
\newblock The directed landscape.
\newblock {\em arXiv preprint arXiv:1812.00309}, 2018.

\bibitem{BPLE}
Duncan Dauvergne and B\'{a}lint Vir\'{a}g.
\newblock Bulk properties of the {A}iry line ensemble.
\newblock {\em Ann. Probab.}, 49(4):1738--1777, 2021.

\bibitem{DV}
Duncan Dauvergne and Lingfu Zhang.
\newblock Disjoint optimizers and the directed landscape.
\newblock arXiv preprint arXiv:2102.00954, 2021.

\bibitem{deconinck2000pole}
Bernard Deconinck and Harvey Segur.
\newblock Pole dynamics for elliptic solutions of the {K}orteweg-de{V}ries
  equation.
\newblock {\em Mathematical Physics, Analysis and Geometry}, 3:49--74, 2000.

\bibitem{MR1702716}
P.~Deift, T.~Kriecherbauer, K.~T.-R. McLaughlin, S.~Venakides, and X.~Zhou.
\newblock Uniform asymptotics for polynomials orthogonal with respect to
  varying exponential weights and applications to universality questions in
  random matrix theory.
\newblock {\em Comm. Pure Appl. Math.}, 52(11):1335--1425, 1999.

\bibitem{deift2007universality}
Percy Deift and Dimitri Gioev.
\newblock Universality at the edge of the spectrum for unitary, orthogonal, and
  symplectic ensembles of random matrices.
\newblock {\em Communications on Pure and Applied Mathematics: A Journal Issued
  by the Courant Institute of Mathematical Sciences}, 60(6):867--910, 2007.

\bibitem{deift2009random}
Percy Deift and Dimitri Gioev.
\newblock {\em Random matrix theory: invariant ensembles and universality},
  volume~18.
\newblock American Mathematical Soc., 2009.

\bibitem{demni2010beta}
N.~Demni.
\newblock $\beta$-{J}acobi processes.
\newblock {\em Advances in Pure and Applied Mathematics}, 1(3):325--344, 2010.

\bibitem{dimitrov2024airy}
Evgeni Dimitrov.
\newblock Airy wanderer line ensembles.
\newblock {\em arXiv preprint arXiv:2408.08445}, 2024.

\bibitem{dimitrov2019log}
Evgeni Dimitrov and Alisa Knizel.
\newblock Log-gases on quadratic lattices via discrete loop equations and
  q-boxed plane partitions.
\newblock {\em Journal of Functional Analysis}, 276(10):3067--3169, 2019.

\bibitem{NIST:DLMF}
{\it NIST Digital Library of Mathematical Functions}.
\newblock \url{https://dlmf.nist.gov/}, Release 1.2.1 of 2024-06-15.
\newblock F.~W.~J. Olver, A.~B. {Olde Daalhuis}, D.~W. Lozier, B.~I. Schneider,
  R.~F. Boisvert, C.~W. Clark, B.~R. Miller, B.~V. Saunders, H.~S. Cohl, and
  M.~A. McClain, eds.

\bibitem{doumerc2005matrices}
Yan Doumerc.
\newblock {\em Matrices al{\'e}atoires, processus stochastiques et groupes de
  r{\'e}flexions}.
\newblock PhD thesis, Toulouse 3, 2005.

\bibitem{MR1936554}
Ioana Dumitriu and Alan Edelman.
\newblock Matrix models for beta ensembles.
\newblock {\em J. Math. Phys.}, 43(11):5830--5847, 2002.

\bibitem{dumitriu2006global}
Ioana Dumitriu and Alan Edelman.
\newblock Global spectrum fluctuations for the $\beta$-{H}ermite and
  $\beta$-{L}aguerre ensembles via matrix models.
\newblock {\em Journal of Mathematical Physics}, 47(6), 2006.

\bibitem{dumitriu2012global}
Ioana Dumitriu and Elliot Paquette.
\newblock Global fluctuations for linear statistics of $\beta$-{J}acobi
  ensembles.
\newblock {\em Random Matrices: Theory and Applications}, 1(04):1250013, 2012.

\bibitem{edelman2024limit}
Alan Edelman, Sungwoo Jeong, and Ron Nissim.
\newblock On the limit of the tridiagonal model for $\beta$-{Dyson Brownian}
  motion.
\newblock {\em arXiv preprint arXiv:2411.01633}, 2024.

\bibitem{MR2331033}
Alan Edelman and Brian~D. Sutton.
\newblock From random matrices to stochastic operators.
\newblock {\em J. Stat. Phys.}, 127(6):1121--1165, 2007.

\bibitem{MR2964770}
L{\'a}szl{\'o} Erd{\H{o}}s, Antti Knowles, Horng-Tzer Yau, and Jun Yin.
\newblock Spectral statistics of {E}rd{\H o}s-{R}\'enyi {G}raphs {II}:
  {E}igenvalue spacing and the extreme eigenvalues.
\newblock {\em Comm. Math. Phys.}, 314(3):587--640, 2012.

\bibitem{erdHos2017dynamical}
L{\'a}szl{\'o} Erd{\H{o}}s and Horng-Tzer Yau.
\newblock {\em A dynamical approach to random matrix theory}, volume~28.
\newblock American Mathematical Soc., 2017.

\bibitem{erdHos2012bulk}
L{\'a}szl{\'o} Erd{\H{o}}s, Horng-Tzer Yau, and Jun Yin.
\newblock Bulk universality for generalized {W}igner matrices.
\newblock {\em Probability Theory and Related Fields}, 154(1):341--407, 2012.

\bibitem{MR2871147}
L{\'a}szl{\'o} Erd{\H{o}}s, Horng-Tzer Yau, and Jun Yin.
\newblock Rigidity of eigenvalues of generalized {W}igner matrices.
\newblock {\em Adv. Math.}, 229(3):1435--1515, 2012.

\bibitem{MR2641363}
P.~J. Forrester.
\newblock {\em Log-gases and random matrices}, volume~34 of {\em London
  Mathematical Society Monographs Series}.
\newblock Princeton University Press, Princeton, NJ, 2010.

\bibitem{gorin2020universal}
Vadim Gorin and Victor Kleptsyn.
\newblock Universal objects of the infinite beta random matrix theory.
\newblock {\em J. Eur. Math. Soc.}, 2023.

\bibitem{gorin2019universality}
Vadim Gorin and Leonid Petrov.
\newblock Universality of local statistics for noncolliding random walks.
\newblock {\em The Annals of Probability}, 47(5):2686--2753, 2019.

\bibitem{MR3418747}
Vadim Gorin and Mykhaylo Shkolnikov.
\newblock Multilevel {D}yson {B}rownian motions via {J}ack polynomials.
\newblock {\em Probab. Theory Related Fields}, 163(3-4):413--463, 2015.

\bibitem{MR3813993}
Vadim Gorin and Mykhaylo Shkolnikov.
\newblock Stochastic {A}iry semigroup through tridiagonal matrices.
\newblock {\em Ann. Probab.}, 46(4):2287--2344, 2018.

\bibitem{GXZ}
Vadim Gorin, Jiaming Xu, and Lingfu Zhang.
\newblock Airy$_\beta$ line ensemble and its {L}aplace transform.
\newblock {\em arXiv preprint arXiv:2411.10829}, 2024.

\bibitem{MR3877550}
Vadim Gorin and Lingfu Zhang.
\newblock Interlacing adjacent levels of {$\beta$}-{J}acobi corners processes.
\newblock {\em Probab. Theory Related Fields}, 172(3-4):915--981, 2018.

\bibitem{graczyk2014strong}
Piotr Graczyk and Jacek Małecki.
\newblock {Strong solutions of non-colliding particle systems}.
\newblock {\em Electronic Journal of Probability}, 19(none):1 -- 21, 2014.

\bibitem{MR3987722}
Alice Guionnet and Jiaoyang Huang.
\newblock Rigidity and edge universality of discrete {$\beta$}-ensembles.
\newblock {\em Comm. Pure Appl. Math.}, 72(9):1875--1982, 2019.

\bibitem{MR3987302}
Alan Hammond.
\newblock A patchwork quilt sewn from {B}rownian fabric: regularity of polymer
  weight profiles in {B}rownian last passage percolation.
\newblock {\em Forum Math. Pi}, 7:e2, 69, 2019.

\bibitem{MR4403929}
Alan Hammond.
\newblock Brownian regularity for the {A}iry line ensemble, and multi-polymer
  watermelons in {B}rownian last passage percolation.
\newblock {\em Mem. Amer. Math. Soc.}, 277(1363):v+133, 2022.

\bibitem{he2024boundary}
Jimmy He.
\newblock {Boundary current fluctuations for the half-space ASEP and six-vertex
  model}.
\newblock {\em Proceedings of the London Mathematical Society}, 128(2):e12585,
  2024.

\bibitem{he2024spectral}
Yukun He.
\newblock Spectral gap and edge universality of dense random regular graphs.
\newblock {\em Communications in Mathematical Physics}, 405(8):1--40, 2024.

\bibitem{he2021fluctuations}
Yukun He and Antti Knowles.
\newblock Fluctuations of extreme eigenvalues of sparse
  {E}rd{\H{o}}s--{R}{\'e}nyi graphs.
\newblock {\em Probability Theory and Related Fields}, 180(3-4):985--1056,
  2021.

\bibitem{holcomb2012edge}
Diane Holcomb and Gregorio~R. Moreno~Flores.
\newblock Edge scaling of the $\beta$-{J}acobi ensemble.
\newblock {\em Journal of Statistical Physics}, 149:1136--1160, 2012.

\bibitem{huang2017beta}
Jiaoyang Huang.
\newblock {$\beta$}-nonintersecting {P}oisson random walks: law of large
  numbers and central limit theorems.
\newblock {\em Int. Math. Res. Not. IMRN}, (8):5898--5942, 2021.

\bibitem{huang2024edge}
Jiaoyang Huang.
\newblock Edge statistics for lozenge tilings of polygons, {I}: concentration
  of height function on strip domains.
\newblock {\em Probab. Theory Related Fields}, 188(1-2):337--485, 2024.

\bibitem{MR4009708}
Jiaoyang Huang and Benjamin Landon.
\newblock Rigidity and a mesoscopic central limit theorem for {D}yson
  {B}rownian motion for general {$\beta$} and potentials.
\newblock {\em Probab. Theory Related Fields}, 175(1-2):209--253, 2019.

\bibitem{MR4771179}
Jiaoyang Huang, Fan Yang, and Lingfu Zhang.
\newblock Pearcey universality at cusps of polygonal lozenge tilings.
\newblock {\em Comm. Pure Appl. Math.}, 77(9):3708--3784, 2024.

\bibitem{huang2022edge}
Jiaoyang Huang and Horng-Tzer Yau.
\newblock Edge universality of sparse random matrices.
\newblock {\em arXiv preprint arXiv:2206.06580}, 2022.

\bibitem{huang2023edge}
Jiaoyang Huang and Horng-Tzer Yau.
\newblock Edge universality of random regular graphs of growing degrees.
\newblock {\em arXiv preprint arXiv:2305.01428}, 2023.

\bibitem{imamura2022solvable}
Takashi Imamura, Matteo Mucciconi, and Tomohiro Sasamoto.
\newblock {Solvable models in the KPZ class: approach through periodic and free
  boundary Schur measures}.
\newblock arXiv preprint arXiv:2204.08420, 2022.

\bibitem{jeong2016limit}
In-Jee Jeong and Sasha Sodin.
\newblock A limit theorem for stochastically decaying partitions at the edge.
\newblock {\em Random Matrices: Theory and Applications}, 5(04):1650016, 2016.

\bibitem{MR1737991}
Kurt Johansson.
\newblock Shape fluctuations and random matrices.
\newblock {\em Comm. Math. Phys.}, 209(2):437--476, 2000.

\bibitem{MR1900323}
Kurt Johansson.
\newblock Non-intersecting paths, random tilings and random matrices.
\newblock {\em Probab. Theory Related Fields}, 123(2):225--280, 2002.

\bibitem{MR2018275}
Kurt Johansson.
\newblock Discrete polynuclear growth and determinantal processes.
\newblock {\em Comm. Math. Phys.}, 242(1-2):277--329, 2003.

\bibitem{ACP}
Kurt Johansson.
\newblock The arctic circle boundary and the {A}iry process.
\newblock {\em Ann. Probab.}, 33(1):1--30, 2005.

\bibitem{MR1863961}
Iain~M. Johnstone.
\newblock On the distribution of the largest eigenvalue in principal components
  analysis.
\newblock {\em Ann. Statist.}, 29(2):295--327, 2001.

\bibitem{johnstone2006high}
Iain~M Johnstone.
\newblock High dimensional statistical inference and random matrices.
\newblock {\em arXiv preprint math/0611589}, 2006.

\bibitem{katori2009zeros}
Makoto Katori and Hideki Tanemura.
\newblock Zeros of {A}iry function and relaxation process.
\newblock {\em Journal of Statistical Physics}, 136(6):1177--1204, 2009.

\bibitem{KTzeros}
Makoto Katori and Hideki Tanemura.
\newblock Zeros of {A}iry function and relaxation process.
\newblock {\em J. Stat. Phys.}, 136(6):1177--1204, 2009.

\bibitem{katori2010non}
Makoto Katori and Hideki Tanemura.
\newblock Non-equilibrium dynamics of {D}yson's model with an infinite number
  of particles.
\newblock {\em Communications in Mathematical Physics}, 293(2):469--497, 2010.

\bibitem{kawamoto2018finite}
Yosuke Kawamoto and Hirofumi Osada.
\newblock Finite-particle approximations for interacting {B}rownian particles
  with logarithmic potentials.
\newblock {\em Journal of the Mathematical Society of Japan}, 70(3):921--952,
  2018.

\bibitem{kawamoto2022infinite}
Yosuke Kawamoto, Hirofumi Osada, and Hideki Tanemura.
\newblock Infinite-dimensional stochastic differential equations and
  tail-fields {II}: the {IFC} condition.
\newblock {\em Journal of the Mathematical Society of Japan}, 74(1):79--128,
  2022.

\bibitem{MR3034787}
Antti Knowles and Jun Yin.
\newblock Eigenvector distribution of {W}igner matrices.
\newblock {\em Probab. Theory Related Fields}, 155(3-4):543--582, 2013.

\bibitem{konig2005orthogonal}
Wolfgang K{\"o}nig.
\newblock {Orthogonal polynomial ensembles in probability theory}.
\newblock {\em Probability Surveys}, 2(none):385 -- 447, 2005.

\bibitem{konig2001eigenvalues}
Wolfgang K{\"o}nig and Neil O'Connell.
\newblock Eigenvalues of the {L}aguerre process as non-colliding squared
  {B}essel processes.
\newblock {\em Electronic Communications in Probability}, 6(none):107 -- 114,
  2001.

\bibitem{konig2002non}
Wolfgang K{\"o}nig, Neil O'Connell, and S{\'e}bastien Roch.
\newblock Non-colliding random walks, tandem queues, and discrete orthogonal
  polynomial ensembles.
\newblock {\em Electronic Journal of Probability}, 7(none):1 -- 24, 2002.

\bibitem{krichever1980elliptic}
Igor~Moiseevich Krichever.
\newblock Elliptic solutions of the {K}adomtsev--{P}etviashvili equation and
  integrable systems of particles.
\newblock {\em Funktsional'nyi Analiz i ego Prilozheniya}, 14(4):45--54, 1980.

\bibitem{MR3433632}
Manjunath Krishnapur, Brian Rider, and B\'alint Vir\'ag.
\newblock Universality of the stochastic {A}iry operator.
\newblock {\em Comm. Pure Appl. Math.}, 69(1):145--199, 2016.

\bibitem{lambert2020strong}
Gaultier Lambert and Elliot Paquette.
\newblock Strong approximation of {G}aussian $beta$-ensemble characteristic
  polynomials: the edge regime and the stochastic {A}iry function.
\newblock {\em arXiv preprint arXiv:2009.05003}, 2020.

\bibitem{landon2020edge}
Benjamin Landon.
\newblock Edge scaling limit of {Dyson Brownian motion} at equilibrium for
  general $\beta \geq 1$.
\newblock {\em arXiv preprint arXiv:2009.11176}, 2020.

\bibitem{MR3161313}
Ji~Oon Lee and Jun Yin.
\newblock A necessary and sufficient condition for edge universality of
  {W}igner matrices.
\newblock {\em Duke Math. J.}, 163(1):117--173, 2014.

\bibitem{moser1976three}
J{\"u}rgern Moser.
\newblock Three integrable hamiltonian systems connected with isospectral
  deformations.
\newblock In {\em Surveys in applied mathematics}, pages 235--258. Elsevier,
  1976.

\bibitem{nemes2017error}
Gerg{\H{o}} Nemes.
\newblock Error bounds for the large-argument asymptotic expansions of the
  {H}ankel and {B}essel functions.
\newblock {\em Acta Applicandae Mathematicae}, 150:141--177, 2017.

\bibitem{MR1802530}
Andrei Okounkov.
\newblock Random matrices and random permutations.
\newblock {\em Internat. Math. Res. Notices}, (20):1043--1095, 2000.

\bibitem{okounkov2002generating}
Andrei Okounkov.
\newblock Generating functions for intersection numbers on moduli spaces of
  curves.
\newblock {\em International Mathematics Research Notices}, 2002(18):933--957,
  2002.

\bibitem{MR1969205}
Andrei Okounkov and Nikolai Reshetikhin.
\newblock Correlation function of {S}chur process with application to local
  geometry of a random 3-dimensional {Y}oung diagram.
\newblock {\em J. Amer. Math. Soc.}, 16(3):581--603, 2003.

\bibitem{osada2012infinite}
Hirofumi Osada.
\newblock Infinite-dimensional stochastic differential equations related to
  random matrices.
\newblock {\em Probability Theory and Related Fields}, 153:471--509, 2012.

\bibitem{osada2013interacting}
Hirofumi Osada.
\newblock {Interacting {B}rownian motions in infinite dimensions with
  logarithmic interaction potentials}.
\newblock {\em The Annals of Probability}, 41(1):1 -- 49, 2013.

\bibitem{osada2013interactingII}
Hirofumi Osada.
\newblock Interacting {B}rownian motions in infinite dimensions with
  logarithmic interaction potentials {II}: {A}iry random point field.
\newblock {\em Stochastic Processes and their applications}, 123(3):813--838,
  2013.

\bibitem{osada2016strong}
Hirofumi Osada and Hideki Tanemura.
\newblock Strong {M}arkov property of determinantal processes with extended
  kernels.
\newblock {\em Stochastic Processes and their Applications}, 126(1):186--208,
  2016.

\bibitem{osada2020infinite}
Hirofumi Osada and Hideki Tanemura.
\newblock Infinite-dimensional stochastic differential equations and tail
  $\sigma$-fields.
\newblock {\em Probability Theory and Related Fields}, 177(3-4):1137--1242,
  2020.

\bibitem{osada2024infinite}
Hirofumi Osada and Hideki Tanemura.
\newblock Infinite-dimensional stochastic differential equations arising from
  {A}iry random point fields.
\newblock {\em Stochastics and Partial Differential Equations: Analysis and
  Computations}, pages 1--117, 2024.

\bibitem{petrov2015asymptotics}
Leonid Petrov.
\newblock Asymptotics of uniformly random lozenge tilings of polygons.
  {Gaussian free field}.
\newblock {\em The Annals of Probability}, 43(1):1--43, 2015.

\bibitem{SIDP}
M.~Pr\"{a}hofer and H.~Spohn.
\newblock Scale invariance of the {PNG} droplet and the {A}iry process.
\newblock {\em J. Statist. Phys.}, 108(5-6):1071--1106, 2002.
\newblock Dedicated to David Ruelle and Yasha Sinai on the occasion of their
  65th birthdays.

\bibitem{prokofev2021elliptic}
Vadim~Vyacheslavovich Prokofev and Anton~Vladimirovich Zabrodin.
\newblock Elliptic solutions of the {T}oda lattice hierarchy and the elliptic
  {R}uijsenaars--{S}chneider model.
\newblock {\em Theoretical and Mathematical Physics}, 208(2):1093--1115, 2021.

\bibitem{MR2813333}
Jos\'e~A. Ram\'irez, Brian Rider, and B\'alint Vir\'ag.
\newblock Beta ensembles, stochastic {A}iry spectrum, and a diffusion.
\newblock {\em J. Amer. Math. Soc.}, 24(4):919--944, 2011.

\bibitem{revuz2013continuous}
Daniel Revuz and Marc Yor.
\newblock {\em Continuous martingales and Brownian motion}, volume 293.
\newblock Springer Science \& Business Media, 2013.

\bibitem{MR2165697}
T.~Sasamoto.
\newblock Spatial correlations of the 1{D} {KPZ} surface on a flat substrate.
\newblock {\em J. Phys. A}, 38(33):L549--L556, 2005.

\bibitem{MR3729037}
Sasha Sodin.
\newblock Several applications of the moment method in random matrix theory.
\newblock In {\em Proceedings of the {I}nternational {C}ongress of
  {M}athematicians---{S}eoul 2014. {V}ol. {III}}, pages 451--475. Kyung Moon
  Sa, Seoul, 2014.

\bibitem{MR3403994}
Sasha Sodin.
\newblock A limit theorem at the spectral edge for corners of time-dependent
  {W}igner matrices.
\newblock {\em Int. Math. Res. Not. IMRN}, (17):7575--7607, 2015.

\bibitem{MR1727234}
Alexander Soshnikov.
\newblock Universality at the edge of the spectrum in {W}igner random matrices.
\newblock {\em Comm. Math. Phys.}, 207(3):697--733, 1999.

\bibitem{MR2717319}
Brian~D. Sutton.
\newblock {\em The stochastic operator approach to random matrix theory}.
\newblock ProQuest LLC, Ann Arbor, MI, 2005.
\newblock Thesis (Ph.D.)--Massachusetts Institute of Technology.

\bibitem{MR2669449}
Terence Tao and Van Vu.
\newblock Random matrices: universality of local eigenvalue statistics up to
  the edge.
\newblock {\em Comm. Math. Phys.}, 298(2):549--572, 2010.

\bibitem{tracy1994level}
Craig~A. Tracy and Harold Widom.
\newblock Level-spacing distributions and the {A}iry kernel.
\newblock {\em Communications in Mathematical Physics}, 159(1):151--174, 1994.

\bibitem{MR1385083}
Craig~A. Tracy and Harold Widom.
\newblock On orthogonal and symplectic matrix ensembles.
\newblock {\em Comm. Math. Phys.}, 177(3):727--754, 1996.

\bibitem{MR2187952}
Craig~A. Tracy and Harold Widom.
\newblock Matrix kernels for the {G}aussian orthogonal and symplectic
  ensembles.
\newblock {\em Ann. Inst. Fourier (Grenoble)}, 55(6):2197--2207, 2005.

\bibitem{tsai2016infinite}
Li-Cheng Tsai.
\newblock Infinite dimensional stochastic differential equations for
  {D}yson’s model.
\newblock {\em Probability Theory and Related Fields}, 166:801--850, 2016.

\bibitem{ViragICM}
B\'alint Vir\'ag.
\newblock Operator limits of random matrices.
\newblock In {\em Proceedings of the {I}nternational {C}ongress of
  {M}athematicians---{S}eoul 2014. {V}ol. {IV}}, pages 247--271. Kyung Moon Sa,
  Seoul, 2014.

\bibitem{zabrodin2020kp}
A.~Zabrodin.
\newblock {KP hierarchy and trigonometric Calogero--Moser hierarchy}.
\newblock {\em Journal of Mathematical Physics}, 61(4), 2020.

\bibitem{MR4541342}
Lingfu Zhang.
\newblock Shift-invariance of the colored {TASEP} and finishing times of the
  oriented swap process.
\newblock {\em Adv. Math.}, 415:Paper No. 108884, 60, 2023.

\bibitem{zhang2024cutoff}
Lingfu Zhang.
\newblock Cutoff profile of the {M}etropolis biased card shuffling.
\newblock {\em Ann. Probab.}, 52(2):713--736, 2024.

\end{thebibliography}

\end{document}